\documentclass[11pt,a4paper]{amsart}
\usepackage{fullpage}
\usepackage[utf8]{inputenc}
\usepackage[english]{babel}
\usepackage{amsmath}
\usepackage{amsfonts}
\usepackage{amssymb}
\usepackage{amsthm}
\usepackage{amsopn}
\usepackage{mathrsfs}
\usepackage{graphicx}
\usepackage{color}
\usepackage{tikz}
\usepackage{tikz-cd}
\usepackage{extarrows}
\usepackage[arrow, matrix, curve]{xy}
\usepackage{mathtools}
\usepackage{csquotes}
\usepackage{enumitem} 
\usetikzlibrary{shapes.misc}

\usepackage[backend=bibtex8]{biblatex}
\theoremstyle{plain}
\newtheorem{thm}{Theorem}[section]
\newtheorem*{theorem*}{Theorem}
\newtheorem{lem}[thm]{Lemma}
\newtheorem{cor}[thm]{Corollary}
\newtheorem{prop}[thm]{Proposition}

\theoremstyle{definition}

\newtheoremstyle{remark}{2ex}{2ex}{}{}{\bfseries}{.}{.5em}{}
\theoremstyle{remark}
\newtheorem{rmk}[thm]{Remark}

\DeclareMathOperator{\THH}{THH}
\DeclareMathOperator{\TC}{TC} 
\DeclareMathOperator{\Tor}{Tor}
\DeclareMathOperator{\Ext}{Ext}

\DeclareMathOperator{\K}{K}
\DeclareMathOperator{\Hom}{Hom}
\DeclareMathOperator{\im}{im}
\DeclareMathOperator{\coker}{coker}
\DeclareMathOperator{\ku}{ku}

\newcommand{\F}{\mathbb{F}}

\newcommand{\id}{\mathrm{id}}

\tikzset{commutative diagrams/.cd,
mysymbol/.style={start anchor=center,end anchor=center,draw=none}
}

\title{The topological Hochschild homology of algebraic $K$-theory of finite fields}
\author{Eva H\"oning}
\address{Fachbereich Mathematik der Universit\"at Hamburg, Bundesstra\ss{}e 55, 20146 Hamburg, Germany}
\email{eva.hoening@uni-hamburg.de}

\usepackage{hyperref}
\begin{document}
\maketitle
\begin{abstract}
Let $\K(\mathbb{F}_q)$ be the algebraic $K$-theory spectrum of the finite field with $q$ elements and let $p \geq 5$ be a prime number coprime to $q$.   In this paper we study the mod $p$ and $v_1$ topological Hochschild homology of $\K(\mathbb{F}_q)$, denoted  $V(1)_*\THH(\K(\mathbb{F}_q))$, as an $\mathbb{F}_p$-algebra. The computations are organized in four different cases, depending on the mod $p$ behaviour of the function $q^n-1$. We use different spectral sequences, in particular the Bökstedt spectral sequence  and a generalization of a spectral sequence of Brun developed in an earlier paper. 
We calculate  the $\mathbb{F}_p$-algebras $\THH_*(\K(\mathbb{F}_q); H\mathbb{F}_p)$,  and we compute $V(1)_*\THH(\K(\mathbb{F}_q))$ in the first two cases. 
\end{abstract}

\section{Introduction}

Let $q =l^n$ be a prime power and let $\K(\mathbb{F}_q)$ be the algebraic $K$-theory spectrum of the finite field with $q$ elements. Let $p$ be  a prime number with $p \neq l$ and $p \geq 5$. In this paper we study the mod $p$ and $v_1$  topological Hochschild homology of $\K(\mathbb{F}_q)$, denoted by $V(1)_*\THH(\K(\mathbb{F}_q))$.   In \cite{Brunss} we constructed a generalization of a spectral sequence of Brun and we applied it to give a short computation of the mod $p$ and $v_1$ topological Hochschild homology of $p$-completed connective complex $K$-theory $\ku_p \simeq \K(\bar{\mathbb{F}}_l)_p$. In this paper we apply the spectral sequence in a similar fashion to $\K(\mathbb{F}_q)_p$.

Topological Hochschild homology  is an ingredient  for the computation of topological cyclic homology ($\TC$) via homotopy fixed points and Tate spectral sequences. 
By  \cite{HM} and \cite{DGM}    the connective cover of  $\TC(\K(\mathbb{F}_q)_p)_p$  is equivalent to $\K(\K(\mathbb{F}_q)_p)_p$. 
 The study of iterated algebraic $K$-theory is interesting because of the red-shift conjecture predicting  that algebraic $K$-theory increases  the chromatic level by one \cite{ARred}. 

By Quillen's computations \cite{Qu} the homotopy of $\K(\mathbb{F}_q)$ is given by 
\[ 
\pi_n(\K(\mathbb{F}_q)) = \begin{cases} 
                                            \mathbb{Z}, &  n = 0; \\
                                            \mathbb{Z}/{q^i-1}, & n = 2i-1, i > 0; \\
                                            0, & \text{otherwise}. 
                                           \end{cases}
\] 
Our computations depend on the degree of the first homotopy group with $p$-torsion  and on the order of the $p$-torsion subgroup.  We define $r$ to be the order of $q$ in $(\mathbb{Z}/p)^*$, so that the first $p$-torsion appears in degree $2r-1$,  and we define $v \coloneqq v_p(q^r-1)$ to be the $p$-adic valuation of $q^r-1$, so that the $p$-torsion subgroup of $\pi_{2r-1}(\K(\mathbb{F}_q))$ has order $p^v$. 
We distinguish the following four cases:

 \begin{enumerate}
\item $r = p-1$ and $v=1$, \label{1}
\item $r = p-1$ and $v \geq 2$, \label{2}
\item $r < p-1$ and $v \geq 2$, \label{3}
\item $r < p-1$ and $ v = 1$. \label{4}
\end{enumerate}

We study $V(1)_*\THH(\K(\mathbb{F}_q))$ by means of the Bökstedt spectral sequence, the generalized  Brun spectral sequence developed in \cite{Brunss} and a spectral sequence of Veen \cite{Veen}.

The Bökstedt spectral sequence  of a commutative $S$-algebra $B$ and a commutative $B$-algebra $C$  has the form 
\[  E^2_{*,*} = \mathbb{H}_{*,*}^{\mathbb{F}_p}\bigl((H\mathbb{F}_p)_*B; (H\mathbb{F}_p)_*C\bigr) \Longrightarrow (H\mathbb{F}_p)_*\THH(B;C). \] 
Here, $H\mathbb{F}_p$ is the Eilenberg-Mac Lane spectrum of $\mathbb{F}_p$, $(H\mathbb{F}_p)_*(-)$ is mod $p$ homology  and  $\mathbb{H}_{*,*}^{\mathbb{F}_p}(-;-)$ denotes ordinary Hochschild homology over the ground ring $\mathbb{F}_p$. The Bökstedt spectral sequence is an $(H\mathbb{F}_p)_*H\mathbb{F}_p$-comodule  $(H\mathbb{F}_p)_*C$-algebra spectral sequence and under some flatness condition one additionally has a coalgebra structure.  

 We define $\K \coloneqq \K(\mathbb{F}_q)_p$. 
We use the following  instances  of the generalized Brun spectral sequence: 
\begin{enumerate}[label = \alph*)]
\item $E^2_{*,*} = V(0)_*(H\mathbb{Z}_p \wedge_{\K} H\mathbb{Z}_p) \otimes_{\mathbb{F}_p} \THH_*(H\mathbb{Z}_p; H\mathbb{F}_p) \Longrightarrow V(0)_*\THH(\K; H\mathbb{Z}_p)$, \label{a}
\item $E^2_{*,*} = V(1)_*\K \otimes_{\mathbb{F}_p} \THH_*(\K; H\mathbb{F}_p)  \Longrightarrow V(1)_*\THH(\K)$.  \label{b}
\end{enumerate}
Here, $V(0)$ denotes the mod $p$ Moore spectrum. 
Note that $V(0)_*\THH(\K; H\mathbb{Z}_p)$ is isomorphic to $\THH_*(\K; H\mathbb{F}_p)$, so that 
 the abutment  of \ref{a} is an input of \ref{b}. 
 
Veen's spectral sequence has the form 
\[ E^{2}_{*,*} = \Tor^{(H\mathbb{F}_p)_*\K}_{*,*}(\mathbb{F}_p, \mathbb{F}_p)  \Longrightarrow \THH_*(\K; H\mathbb{F}_p).\] 
We examine Veen's spectral sequence in small total degrees and use our result to determine the differentials of \ref{a}.

In case (\ref{1}) and (\ref{2}) the mod $p$ homology of $\K$  has an easy form and the Bökstedt spectral sequence converging to $(H\mathbb{F}_p)_*\THH(\K)$ has a coalgebra structure.  This is useful to compute the differentials.  In  case (\ref{1}) the Bökstedt spectral sequence has  already  been computed  by Angeltveit and Rognes 
 \cite{AnRo}.   We proceed similarly in case (\ref{2}) (Subsection \ref{compBocs2}). 
In case (\ref{2}) an Ext spectral sequence argument  shows that the $(2p-3)$th Postnikov invariant of $V(1) \wedge_S \K$ is in the image of the forgetful functor from  the homotopy category of $H\mathbb{F}_p \wedge_S \K$-modules to the homotopy category  of $\K$-modules (Lemma \ref{bimod}).  This implies that $V(1) \wedge_S \THH(\K)$ is an $H\mathbb{F}_p$-module and we can identify $V(1)_*\THH(\K)$ with the $(H\mathbb{F}_p)_*H\mathbb{F}_p$-comodule primitives in $(H\mathbb{F}_p)_*(V(1) \wedge_S \THH(\K))$. We obtain (see Theorem \ref{casetwores}):

\begin{theorem*}  
In case (\ref{2}) we have an isomorphism of $\mathbb{F}_p$-algebras
\[ V(1)_*\THH(\K) \cong E(x) \otimes_{\mathbb{F}_p} E(\lambda_1, \lambda_2) \otimes_{\mathbb{F}_p} P(\mu_2) \otimes_{\mathbb{F}_p} \Gamma(\gamma_1').\] 
Here, $P(-)$, $E(-)$ and $\Gamma(-)$ denote the polynomial, exterior and divided power algebra over $\mathbb{F}_p$ and the degrees are given by $|x| = 2p-3$, $|\lambda_i| = 2p^i-1$, $|\mu_2| = 2p^2$ and $|\gamma_1'| = 2p-2$. 
\end{theorem*}
We compute $\THH_*(\K; H\mathbb{F}_p)$ via the spectral sequence \ref{a} in all the cases except  for the subcase of case (\ref{4}) where $r = 1$ (Subsection \ref{et2}).  In case (\ref{4}) it seems harder  to compute $\THH_*(\K; H\mathbb{F}_p)$ via the Bökstedt spectral sequence, because this depends on the mod $p$ homology of $\K$ which is complicated in this case. 
  In order to  compute the differentials in the spectral sequence \ref{a} we only need to examine Veen's spectral sequence  in small total degrees. 
  This only depends  on low degrees of $(H\mathbb{F}_p)_*\K$. 

We determine  the spectral sequences \ref{b} in case (\ref{1}) (Subsection \ref{et3}). 
There is only one possible differential. Its existence follows from the fact that the mod $p$ homology of  $V(1) \wedge_S \THH(\K)$ has no non-trivial comodule primitives in degree $2p^2-1$ (Subsection \ref{Boek1}).  We obtain (see Theorem \ref{case1V1THH}):
\begin{theorem*} 
In case (\ref{1})  the $V(1)$-homotopy of $\THH(\K)$ is the homology of the differential graded algebra 
\[  E(x,a,\lambda_2) \otimes_{\mathbb{F}_p} P(\mu_2) \otimes_{\mathbb{F}_p} \Gamma(b), \,  \, \, d(\lambda_2) = xa \] 
with $|x| =  2p-3$, $|a| = p(2p-2) + 1$, $|\lambda_2| =2p^2-1$, $|\mu_2| = 2p^2$ and $|b| = p(2p-2)$. 
\end{theorem*}

The last result  was also obtained by Angelini-Knoll using a different approach \cite{Gabe}. In \cite{Gabe2}  Angelini-Knoll  also shows that $\K(\K(\mathbb{F}_q)_p)$ detects the $\beta$-family in case (\ref{1}).

There is a fiber sequence  of spectra
\[\begin{tikzcd}
\K \ar{r} & \ku_p \ar{r}{f} & \Sigma^2 \ku_p
\end{tikzcd}\]
(see \cite{Hir}).  Denoting by $\psi^q$ the Adams operation, the map $f$ is the unique lift of  $\psi^q - \id $  to the $1$-connective covering $\Sigma^2\ku_p$ of $\ku_p$. 
The relation between $f$ and the multiplication of $\ku_p$ can be informally written as 
\begin{equation} \label{derf}
 f(ab)  = f(a) b + a f(b) + f(a) (\psi^q-\id)(b) 
 \end{equation} 
(see \cite{Wat}). 
We show that this  fiber sequence can be constructed in the homotopy category of $\K$-modules and that  the  equation (\ref{derf}) also holds  in this category (Section \ref{fibseqsec}).  These observations are very useful for our computations: 
The $\K$-linearity of the fiber sequence is helpful  to determine the multiplicative structure of $V(0)_*\K$ and $(H\mathbb{F}_p)_*\K$ (Section \ref{V0andV1} and Section \ref{modphomK}).   We  use the $\K$-linearity of the fiber sequence and equation (\ref{derf}) to  compute the ring  $V(0)_*(H\mathbb{Z}_p \wedge_{\K} H\mathbb{Z}_p)$  (Subsection \ref{et1}).

\subsection*{Acknowledgments}
The content of this article is part of my PhD thesis. I am very grateful to  my PhD supervisor Christian Ausoni  for his support. 
I woud like to thank Birgit Richter for telling me that one of my initial attempts probably does not work and for her help.  I am very thankful to Magdalena Kedziorek for her feedback. I would like to thank  Gabriel  Angelini-Knoll  for helpful discussions about the subject and I would like  thank Irina Bobkova Gemma Halliwell, Ayelet Lindenstrauss, Kate Poirier, Birgit Richter and Inna Zakharevich  for helpful discussions about higher $\THH$ during the Women in topology II  and the AIM SQuaRE project. 
This work was supported by grants from R\'egion  \^Ile-de-France and the project ANR-16-CE40-0003 ChroK.  I  would like to thank the MPIM Bonn for providing ideal working conditions during final stages of this work. 

\section{Notations and recollections}
We work in the setting of Elmendorf, Kriz, Mandell and May \cite{EKMM}.  For a commutative $S$-algebra $R$  let  $\mathscr{M}_R$ be the category of $R$-modules.  We denote its symmetric monoidal smash  product by $\wedge_R$.  
The category $\mathscr{M}_R$ is a model category, where the weak equivalences are the $\pi_*$-isomorphisms, the cofibrations are the retracts of relative cell $R$-modules and where all objects are fibrant.  We denote its homotopy category by  $\mathscr{D}_R$. The category $\mathscr{D}_R$ is  a tensor triangulated category (see \cite[Definition 1.1]{Bal}), where the  distinguished triangles are given, up to isomorphism,  by the images of the cofiber sequences under $\mathscr{M}_R \to \mathscr{D}_R$. 
The functor $\mathscr{M}_R \to \mathscr{D}_R$ is  lax symmetric monoidal and, denoting the tensor product in $\mathscr{D}_R$ by $\wedge_R^L$,  the structure map
\begin{equation}  \label{MRDR}
\begin{tikzcd}
M \wedge_R^L N \ar{r} &  M \wedge_R N  
\end{tikzcd}
\end{equation} 
is  an  isomorphism in $\mathscr{D}_R$ if $M$ or $N$ is a cofibrant $R$-module. 
For a morphism of commutative $S$-algebras $R \to R'$ the functor $\mathscr{D}_{R'} \to \mathscr{D}_R$   maps distinguished triangles to distinguished triangles and is lax symmetric monoidal. 
Spheres in $\mathscr{M}_R$ are denoted by $S^n_R$  and homotopy groups are given by $\pi_n(M) = \mathscr{D}_R(S^n_R, M)$ for $M \in \mathscr{M}_R$. 
 We denote the right adjoint of $ - \wedge_R^L M$ by $F_R^L(M, -)$. 
The functor $F_R^L(M, -)$  preserves distinguished triangles. We set $\Ext_R^*(-,-) \coloneqq \pi_{-*}(F_R^L(-,-))$. Note that we have a natural isomorphism $\Ext_R^0(-,-) \cong \mathscr{D}_R(-,-)$. 
An $R$-ring spectrum is an object $A \in \mathscr{D}_R$ with maps $A \wedge_R^L A \to A $ and $R \to A$ in $\mathscr{D}_R$ satisfying the left and right unit laws.   
We denote the category of commutative $R$-algebras by $\mathscr{C}\mathscr{A}_R$.  It has a model category structure, where the weak equivalences are 
the $\pi_*$-isomorphisms and all objects are fibrant \cite[Chapter VII]{EKMM}. The category $\mathscr{C}\mathscr{A}_R$ can be identified with the the category of commutative $S$-algebras under $R$ and the model category structure on $\mathscr{C}\mathscr{A}_R$ is the one inherited from $\mathscr{C}\mathscr{A}_S$. 
By \cite[Lemma 2.2]{Brunss} the map (\ref{MRDR}) is an isomorphism of $R$-ring spectra if $R \to M$ and $R \to N$ are maps  between cofibrant commutative $S$-algebras and if $R \to M$ or $R \to N$ is a cofibration in $\mathscr{C}\mathscr{A}_S$. 

For a prime $p \geq 5$ we denote by $V(0)$ and $V(1)$ the mod $p$ Moore spectrum and the mod $(p,v_1)$  Toda-Smith complex. We can assume that $V(0)$ and $V(1)$ are cell $S$-modules. 
We have distinguished triangles in $\mathscr{D}_S$ 
\begin{equation} \begin{tikzcd} \label{V0defi}
S \ar{r}{p \cdot \id} & S \ar{r} & V(0) \ar{r} & \Sigma S  \end{tikzcd} \end{equation}
\[ \begin{tikzcd}
\Sigma^{2p-2} V(0) \ar{r} & V(0) \ar{r} & V(1) \ar{r} & \Sigma^{2p-1} V(0),  
\end{tikzcd} \]
where $S \to V(0)$ and $V(0) \to V(1)$ are maps of associative and commutative $S$-ring spectra, \cite{OkaMoore}, \cite{OkaMult}, \cite{Okafewcells}.

By \cite[Example VI.5.2]{MayRingSp},\cite{MayBip}, \cite{Maywhat} and \cite[Corollary II.3.6]{EKMM}  algebraic $K$-theory can be realized as a functor $\K(-)$  from the category of commutative rings to $\mathscr{C}\mathscr{A}_S$. For a commutative ring $R$ the $S$-algebra  $\K(R)$ is connective.  One has $\pi_0(\K(R)) = \mathbb{Z}$ and $\pi_i(\K(R))$ is Quillen's $i$th algebraic $\K$-theory group for $i > 0$ \cite[Example 6.2]{MayGood}.

 Recall that  $p$-completion is Bousfield localization with respect to $V(0)$.   
By the proofs of \cite[Lemma VII.5.8]{EKMM}, \cite[Lemma VII.5.2]{EKMM}  and \cite[Theorem VIII.2.2]{EKMM}
 the $p$-completion of a commutative $S$-algebra can be constructed as a commutative $S$-algebra in such a way that we get a functor $(-)_p: \mathscr{C}\mathscr{A}_S \to \mathscr{C}\mathscr{A}_S$  with values in cofibrant commutative $S$-algebras.
 
For an abelian group $A$ we denote by $HA$ its Eilenberg-Mac Lane spectrum. 
Recall that by  \cite{Maywhat}, \cite{MayMult}, \cite{MayRingSp}, \cite[Corollary II.3.6]{EKMM}  and by functoriality of cofibrant replacements \cite[LemmaVII.5.8]{EKMM} the Eilenberg-Mac Lane spectrum $HR$ of a commutative ring $R$ can be realized as a cofibrant commutative $S$-algebra  in such a way that we get a functor from the category of commutative rings to $\mathscr{C}\mathscr{A}_S$.

 We denote by $P(-)$, $E(-)$, $\Gamma(-)$ and $P_k(-)$ the polynomial algebra, the exterior algebra, the divided power algebra and the truncated polynomial algebra over $\mathbb{F}_p$, and  we write $\otimes$ for $\otimes_{\mathbb{F}_p}$.    Furthermore,  we write $b \doteq c$ if $b$ and $c$ are equal up to multiplication by a unit in $\mathbb{F}_p$.

 An infinite cycle in a  spectral sequence is a class $b$ such that we have $d^s(b) = 0$ for all $s$. A permanent cycle is an infinite cycle that is not in the image  of $d^s$ for any $s$.  
\section{The fiber sequence relating $\K(\mathbb{F}_q)_p$ and $\K(\bar{\mathbb{F}}_l)_p$} \label{fibseqsec}

Throughout this paper  we fix a prime power $q = l^n$ for $n \geq 1$ and denote by $\mathbb{F}_q$ the finite field with $q$ elements. 
Let $p \neq l$ be a prime number with $p \geq 5$  and let 
$V(0)$ and $V(1)$ be built   with respect to this prime.

By \cite{Qu} we have a fiber sequence of spaces 
\[\K(\F_q) \xrightarrow{ ~ ~ ~ ~ } BU \times \mathbb{Z} \xrightarrow{\psi^q-\id} BU,\]
where $\psi^q$ is the Adams operation. 
After $p$-completion one gets an analogous fiber sequence of spectra \cite{Hir}. 
In this section we construct a fiber sequence of this form in  $\mathscr{D}_{\K(\F_q)_p}$.

We define  $\K$ to be the commutative $S$-algebra   $\K(\mathbb{F}_q)_p$. Quillen's computations \cite[Theorem 8]{Qu} imply that we have
\[ V(0)_n(\K) = \begin{cases}
                                                      \mathbb{F}_p,  & n = 2ri, i \geq 0;  \\
                                                       \mathbb{F}_p, &  n = 2ri-1, i > 0;  \\
                                                       0,  & \text{otherwise};
                                                     \end{cases} \] 
where $r$ is the order of $q$ in $(\mathbb{Z}/p)^*$. 
Let $\bar{\mathbb{F}}_l$ be the algebraic closure of $\mathbb{F}_l$. Then, $\K(\bar{\mathbb{F}}_l)_p$ is equivalent as a ring spectrum to $p$-completed connective complex $\K$-theory \cite[Section7]{MayGood}. We thus have an isomorphism of rings 
\[ \pi_*(\K(\bar{\mathbb{F}}_l)_p) \cong \mathbb{Z}_p[u]\] 
with $|u| = 2$. By \cite[Section 7]{MayGood} the Adams operation $\psi^q$ corresponds to the map $\phi^q$ induced by the Frobenius automorphism $\bar{\mathbb{F}}_l \to \bar{\mathbb{F}}_l, x \mapsto x^q$. In homotopy it is the map given by $u \mapsto qu$ \cite[Subsection 5.5.1]{RogGal}.  We have  isomorphisms of $\mathbb{F}_p$-algebras 
$V(0)_*\K(\bar{\mathbb{F}}_l)_p= P(u)$ and $V(1)_*\K(\bar{\mathbb{F}}_l)_p  = P_{p-1}(u)$.

By functoriality the map $\phi^q: \K(\bar{\mathbb{F}}_l)_p \to \K(\bar{\mathbb{F}}_l)_p$ induced by the Frobenius automorphism is a map of  $\K$-algebras and therefore of $\K$-modules.
We consider the element 
\[ \phi^q - \id \in \mathscr{D}_{\K}\bigl(\K(\bar{\mathbb{F}}_l)_p,\K(\bar{\mathbb{F}}_l)_p\bigr).\]
Note that the $0$-connected covering of $\K(\bar{\mathbb{F}}_l)_p$ in $\K$-modules is given by 
the morphism  $\bar{u}: S^2_{\K}\wedge_{\K}^L \K(\bar{\mathbb{F}}_l)_p \to \K(\bar{\mathbb{F}}_l)_p$ in $\mathscr{D}_{\K}$ defined by
\[\begin{tikzcd}
 S^2_K \wedge_{\K}^L \K(\bar{\mathbb{F}}_l)_p \ar{r}{u \wedge \id} & \K(\bar{\mathbb{F}}_l)_p \wedge_{\K}^L \K(\bar{\mathbb{F}}_l)_p \ar{r}{m} & \K(\bar{\mathbb{F}}_l)_p.\end{tikzcd}\]
Here $m$ is the product of the $\K$-ring spectrum $\K(\bar{\mathbb{F}}_l)_p$.

 \begin{lem} \label{lift}
 There exist exactly one element $f \in \mathscr{D}_{\K}\bigl(\K(\bar{\mathbb{F}}_l)_p, S^2_{\K} \wedge_{\K}^L \K(\bar{\mathbb{F}}_l)_p\bigr)$ that is mapped to $\phi^q - \id$ under 
 \[\mathscr{D}_{\K}\bigl(\K(\bar{\mathbb{F}}_l)_p, S^2_{\K}  \wedge_{\K}^L  \K(\bar{\mathbb{F}}_l)_p\bigr) \xrightarrow{\bar{u}_*} \mathscr{D}_{\K}\bigl(\K(\bar{\mathbb{F}}_l)_p, \K(\bar{\mathbb{F}}_l)_p\bigr).\]

 \end{lem}
\begin{proof}
We have an exact sequence:
\begin{align*}
\Ext^{-1}_{\K}\bigl(\K(\bar{\mathbb{F}}_l)_p, H\mathbb{Z}_p\bigr) & \to \mathscr{D}_{\K}\bigl(\K(\bar{\mathbb{F}}_l)_p, S^2_{\K} \wedge_{\K}^L \K(\bar{\mathbb{F}}_l)_p\bigr) \xrightarrow{\bar{u}_*} \mathscr{D}_{\K}\bigl(\K(\bar{\mathbb{F}}_l)_p, \K(\bar{\mathbb{F}}_l)_p\bigr) \\
&  \to  \mathscr{D}_{\K}\bigl(\K(\bar{\mathbb{F}}_l)_p, H\mathbb{Z}_p\bigr).
\end{align*}
 By \cite{LewMand}  and \cite[Theorem 7.1]{Boar} we have a cohomological strongly convergent spectral sequence of the form
 \[E^{n,m}_2 = \Ext^{n,m}_{\K_*}\Bigl(\bigl(\K(\bar{\mathbb{F}}_l)_p\bigr)_*, \mathbb{Z}_p\Bigr) \Longrightarrow \Ext^{n+m}_{\K}\bigl(\K(\bar{\mathbb{F}}_l)_p, H\mathbb{Z}_p\bigr).\]
Since $E_2^{n,m} = 0$ for $m < 0$, we get $\Ext^{-1}_{\K}\bigl(\K(\bar{\mathbb{F}}_l)_p, H\mathbb{Z}_p\bigr) = 0$. 
We have a commutative diagram 
\[\begin{tikzcd}
\mathscr{D}_{\K}\bigl(\K(\bar{\mathbb{F}}_l)_p, \K(\bar{\mathbb{F}}_l)_p\bigr) \ar{r} \ar{d} & \mathscr{D}_{\K}\bigl(\K(\bar{\mathbb{F}}_l)_p, H\mathbb{Z}_p\bigr) \ar{d} \\
\Hom_{\K_*}\Bigl( \pi_*\bigl( \K(\bar{\mathbb{F}}_l)_p\bigr), \pi_*\bigl( \K(\bar{\mathbb{F}}_l)_p \bigr)\Bigr) \ar{r} & \Hom_{\K_*}\Bigl(\pi_*\bigl( \K(\bar{\mathbb{F}}_l)_p\bigr), \mathbb{Z}_p \Bigr).
\end{tikzcd}\]
The right vertical map is an isomorphism because it identifies with the edge homomorphism in the above spectral sequence \cite{LewMand}.
Since   $(\phi^q)_0$ is the identity, we have that 
\[ (\phi^q -\id)_0: \pi_0\bigl(\K(\bar{\mathbb{F}}_l)_p\bigr) \to \pi_0\bigl(\K(\bar{\mathbb{F}}_l)_p\bigr)\]
is zero. 
 \end{proof}
The map $f$ is part of a distinguished triangle  in $\mathscr{D}_{\K}$:
\[\begin{tikzcd}
F \ar{r} &  \K(\bar{\mathbb{F}}_l)_p \ar{r}{f} & S^2_{\K} \wedge_{\K}^L \K(\bar{\mathbb{F}}_l)_p \ar{r} &  F \wedge S^1.
\end{tikzcd}\]
 
 From the long exact sequence in homotopy we get: 
 \begin{lem}
 We have 
 \[ \pi_n(F) \cong \begin{cases}
                \mathbb{Z}_p, & n = 0; \\
                \mathbb{Z}_p/{(q^j-1)\mathbb{Z}_p}, & n = 2j-1, \, j> 0; \\
                0, & \text{otherwise}. 
              \end{cases} \]
 \end{lem}
 
 \begin{lem} \label{V0htpyf}
 We have 
 \[ V(0)_nF = \begin{cases}
               \mathbb{F}_p, & n = 2jr, \, j \geq 0; \\
               \mathbb{F}_p, & n = 2jr-1, \, j> 0; \\
               0, & \text{otherwise}. 
              \end{cases}\]
 \end{lem}
 \begin{proof}
 For $n \in \mathbb{Z}$ let $v_p(n)$ be the $p$-adic valuation of $n$. 
For $j > 0$  we have
 \[\mathbb{Z}_p/{(q^j-1)\mathbb{Z}_p} = \mathbb{Z}_p/{p^{v_p(q^j-1)}\mathbb{Z}_p} = \mathbb{Z}/{p^{v_p(q^j-1)}\mathbb{Z}}.\]
   The claim now  follows by using the distinguished triangle (\ref{V0defi}).
 \end{proof}
 \begin{lem} \label{iso2ir1}
 The map 
 \[V(0)_{2rj}(F) \to V(0)_{2rj}\bigl(\K(\bar{\mathbb{F}}_l)_p\bigr)\]
 is an isomorphism for $j\geq 0$. 
 \end{lem}
 \begin{proof}
For $j \geq 0$  there is an exact sequence
 \begin{equation*}  \label{kes}
 \begin{tikzcd}
   V(0)_{2jr-1}\bigl(\K(\bar{\mathbb{F}}_l)_p\bigr) \ar{r} & V(0)_{2rj}(F) \ar{r} & V(0)_{2rj}\bigl(\K(\bar{\mathbb{F}}_l)_p\bigr).
 \end{tikzcd}
 \end{equation*}
 Since the first term  is zero, and 
 since the second and third term are both $\mathbb{F}_p$, this proves the lemma. 
 \end{proof}
  We want to show that $F \cong \K$  in $\mathscr{D}_{\K}$.  We denote the map $\K \to \K(\bar{\mathbb{F}}_l)_p$ induced by the inclusion of fields $\mathbb{F}_q \to \bar{\mathbb{F}}_l$ by $i$. 
  \begin{lem} \label{iso2ir2}
  The map 
  \[i_{2rj}: V(0)_{2rj}(\K) \to V(0)_{2rj}\bigl(\K(\bar{\mathbb{F}}_l)_p\bigr)\]
  is an isomorphism for $j \geq 0$. 
  \end{lem}
  \begin{proof}
    For  $j > 0$  we have $\pi_{2rj}\bigl(\K(\bar{\mathbb{F}}_l)\bigr) = 0$ by \cite[p.585]{Qu}, so we get a map of exact sequences
    \[\begin{tikzcd}
    0 \ar{r} \ar{d} & V(0)_{2rj}\bigl(\K(\mathbb{F}_q)\bigr) \ar{r} \ar{d} &  \pi_{2rj-1}\bigl(\K(\mathbb{F}_q)\bigr) \ar{r}{\cdot p}  \ar{d} & \pi_{2rj-1}\bigl(\K(\mathbb{F}_q)\bigr) \ar{d} \\
    0 \ar{r} & V(0)_{2rj}\bigl(\K(\bar{\mathbb{F}}_l)\bigr) \ar{r} &   \pi_{2rj-1}\bigl(\K(\bar{\mathbb{F}}_l)\bigr) \ar{r}{\cdot p} & \pi_{2rj-1}\bigl(\K(\bar{\mathbb{F}}_l)\bigr).
    \end{tikzcd}\]
    Because $V(0)_{2rj}\bigl(\K(\mathbb{F}_q)\bigr)$ and $ V(0)_{2rj}\bigl(\K(\bar{\mathbb{F}}_l)\bigr)$ are both isomorphic to 
    $\mathbb{F}_p$, 
 it suffices to show that $\pi_{2rj-1}\bigl(\K(\mathbb{F}_q)\bigr) \to \pi_{2rj-1}\bigl(\K(\bar{\mathbb{F}}_l)\bigr)$ is injective. 
    This follows because by \cite[p.585]{Qu} the
   abelian group $\pi_*(\K(\bar{\mathbb{F}}_l))$ is the filtered  colimit of the  abelian groups $\pi_*\bigl(\K(k)\bigr)$,
    where $k$ runs over the finite subfields of $\bar{\mathbb{F}}_l$, and because by \cite[Theorem 8]{Qu} the maps $\pi_*\bigl(\K(k)\bigr) \to \pi_*\bigl(\K(k')\bigr)$ induced  by inclusions  of finite fields $k \hookrightarrow k'$ are injective.
  \end{proof}
  \begin{lem}
  There exist an $h \in \mathscr{D}_{\K}(\K , F)$ that is mapped to $i$ under
  \[ \mathscr{D}_{\K}(\K, F) \to \mathscr{D}_{\K}\bigl(\K, \K(\bar{\mathbb{F}}_l)_p\bigr).\]
  \end{lem}
  \begin{proof}
We have an exact sequence 
  \[ \begin{tikzcd}
  \mathscr{D}_{\K}(\K, F) \ar{r} & \mathscr{D}_{\K}\bigl(\K, \K(\bar{\mathbb{F}}_l)_p\bigr) \ar{r}{f_*} & \mathscr{D}_{\K}\bigr(\K, S^2_{\K} \wedge_{\K} \K(\bar{\mathbb{F}}_l)_p\bigl).
  \end{tikzcd} \]
  It thus suffices to show that $f_*(i) = 0$. There is an exact sequence
  \[ \begin{tikzcd}
 \Ext^{-1}_{\K}(\K, H\mathbb{Z}_p) \ar{r} & \mathscr{D}_{\K}\bigl(\K,  S^2_{\K} \wedge_{\K}^L \K(\bar{\mathbb{F}}_l)_p\bigr) \ar{r}{\bar{u}_*} & \mathscr{D}_{\K}\bigl(\K, \K(\bar{\mathbb{F}}_l)_p\bigr).  
  \end{tikzcd}\]
  Using the $\Ext$ spectral sequence one gets  that $\Ext^{-1}_{\K}(\K, H\mathbb{Z}_p) = 0$. We therefore only have to show that
  $\bar{u} \circ f \circ i = (\phi^q - \id) \circ i$ is zero. This is clear. 
  \end{proof}
  We will show that $h: \K \to F$ is an isomorphism. We show this by proving that its image under $\mathscr{D}_{\K} \to \mathscr{D}_{S}$ is a $V(0)_*$-equivalence between $V(0)$-local $S$-modules.

  \begin{lem} \label{UEC}
  The map 
  \[\begin{tikzcd}
   h_*: V(0)_* \K  \ar{r}{h_*} & V(0)_*  F 
   \end{tikzcd}\]
  is an isomorphism. 
  \end{lem}
  \begin{proof}
  It is clear that $h_*$ is an isomorphism in  the degrees 
  $2rj$ for $j \geq 0$. It thus suffices to show that $h_*$ is an isomorphism in the degrees $2rj-1$ for $j > 0$. 
  
 The map $h$ induces a map between the exact couples
  \begin{equation} \label{uec1} 
  \begin{tikzcd}
    \pi_*(\K)  \ar{rr}{\cdot p} & &  \pi_*(\K) \ar{ld}  \\
     & V(0)_*(\K) \ar{lu}{d} &  
    \end{tikzcd} 
  \end{equation}
  and 
  \begin{equation} \label{uec2}
   \begin{tikzcd}
    \pi_*(F)  \ar{rr}{\cdot p} & &  \pi_*(F) \ar{ld} \  \\
     & V(0)_*(F) \ar{lu}{d} &  
    \end{tikzcd} 
  \end{equation}
  and therefore a map of the associated singly-graded Bockstein spectral sequences. 
  Fix $j > 0$. Let $a$ be a generator of  $V(0)_{2jr}(F) = \mathbb{F}_p$.  Since $\pi_{2jr}(F) = 0$  we get
  that 
  \[ d(a) \in \pi_{2jr-1}(F) = \mathbb{Z}/{p^{v_p(q^{jr}-1)}\mathbb{Z}} \]
  is an element of  order $p$. It follows that $d(a)$ has a preimage under the map 
  \[ \begin{tikzcd}
   \pi_{2jr-1}(F)  \ar{rr}{\cdot p^{v_p(q^{jr}-1)-1}} & &  \pi_{2jr-1}(F), 
   \end{tikzcd}\]
   but not under the map
  \[ \begin{tikzcd}
   \pi_{2jr-1}(F) \ar{rr}{\cdot p^{v_p(q^{jr}-1)}} & &  \pi_{2jr-1}(F).
   \end{tikzcd}\]  
  Hence, $a$ survives to  the $E^{v_p(q^{jr}-1)}$-page in the spectral sequence associated to the  exact couple (\ref{uec2}) and  
   \[ d^{v_p(q^{jr}-1)}(a) \neq 0.\]
   Since  $\K$ has the same homotopy and mod $p$ homotopy groups as $F$, the same argument as above  shows that the preimage $b$ of $a$ under the isomorphism 
   \[ \begin{tikzcd}
   V(0)_{2jr}(\K) \ar{r}{h_{2jr}} & V(0)_{2jr}(F) 
   \end{tikzcd}\]
   has to survive to the $E^{v_p(q^{jr}-1)}$-page of the spectral sequence associated to the exact couple (\ref{uec1}) and that $ d^{v_p(q^{jr}-1)}(b) \neq 0$. 
   We  conclude that $h_*$ maps
   \[d^{v_p(q^{jr}-1)}(b) \in V(0)_{2jr-1}(\K)\] 
   to 
   \[ d^{v_p(q^{jr}-1)}(a) \in V(0)_{2jr-1}(F).\]
   Thus, ${h}_*: V(0)_*\K \to V(0)_*F$ is an isomorphism. 
     \end{proof}
  By definition, $\K$ is $V(0)$-local.
     
  \begin{lem}
  The $S$-module $F$ is $V(0)$-local.
  \end{lem}
  \begin{proof}
  Let $W$ be a $V(0)$-acyclic $S$-module. 
  We have an exact sequence
  \[\begin{tikzcd}
  \Ext^{-1}_S\bigl(W, \, S^2_{\K} \wedge_{\K}^L \K(\bar{\mathbb{F}}_l)_p\bigr) \ar{r} & \mathscr{D}_S(W, F) \ar{r} & \mathscr{D}_S\bigl(W, \, \K(\bar{\mathbb{F}}_l)_p\bigr).
 \end{tikzcd}\]
Because  $\K(\bar{\mathbb{F}}_l)_p$ is $V(0)$-local,  one has  $\mathscr{D}_S\bigl(W, \K(\bar{\mathbb{F}}_l)_p\bigr) = 0 $.  Since desuspensions of $V(0)$-acyclic $S$-modules are $V(0)$-acyclic, we get 
\[  \Ext^{-1}_S\bigl(W, \, S^2_{\K} \wedge_{\K}^L \K(\bar{\mathbb{F}}_l)_p\bigr) = 0.\] 
  \end{proof}
  \begin{cor} \label{fibseqcor}
 We have a distinguished triangle of the form 
  \begin{equation} \label{fibseq}
  \begin{tikzcd}
   \K  \ar{r}{i} & \K(\bar{\mathbb{F}}_l)_p \ar{r}{f} & S^2_{\K} \wedge_{\K}^L \K(\bar{\mathbb{F}}_l)_p \ar{r} & \Sigma \K
   \end{tikzcd}
  \end{equation}
in $\mathscr{D}_{\K}$. 
  \end{cor}

The following lemma will be useful later to determine multiplicative structures. 

\begin{lem}  \label{cofibmult}
In $\mathscr{D}_{\K}$ we have the following equality of morphisms: 
\begin{align*} 
& \K(\bar{\mathbb{F}}_l)_p \wedge_{\K}^L \K(\bar{\mathbb{F}}_l)_p  \xrightarrow{m} \K(\bar{\mathbb{F}}_l)_p \xrightarrow{f} S^2_{\K} \wedge_{\K}^L \K(\bar{\mathbb{F}}_l)_p \\
= &  \K(\bar{\mathbb{F}}_l)_p \wedge_{\K}^L \K(\bar{\mathbb{F}}_l)_p \xrightarrow{f \wedge \id} S^2_{\K} \wedge_{\K}^L \K(\bar{\mathbb{F}}_l)_p \wedge_{\K}^L\K(\bar{\mathbb{F}}_l)_p  \xrightarrow{\Sigma^2  m} S^2_{\K} \wedge_{\K}^L  \K(\bar{\mathbb{F}}_l)_p \\
& +  \K(\bar{\mathbb{F}}_l)_p \wedge_{\K}^L \K(\bar{\mathbb{F}}_l)_p \xrightarrow{\id \wedge f} \K(\bar{\mathbb{F}}_l)_p \wedge_{\K}^L S^2_{\K} \wedge_{\K}^L \K(\bar{\mathbb{F}}_l)_p \xrightarrow{\Sigma^2 m} S^2_{\K} \wedge_{\K}^L \K(\bar{\mathbb{F}}_l)_p \\
 & +  \K(\bar{\mathbb{F}}_l)_p \wedge_{\K}^L \K(\bar{\mathbb{F}}_l)_p \xrightarrow{f \wedge (\phi^q-\id)} S^2_{\K} \wedge_{\K}^L \K(\bar{\mathbb{F}}_l)_p \wedge_{\K}^L\K(\bar{\mathbb{F}}_l)_p  \xrightarrow{\Sigma^2 m} S^2_{\K} \wedge_{\K}^L  \K(\bar{\mathbb{F}}_l)_p. 
\end{align*}
\end{lem} 
\begin{proof}
By a $\Tor$ spectral sequence argument it follows that $\K(\bar{\mathbb{F}}_l)_p \wedge_{\K}^L \K(\bar{\mathbb{F}}_l)_p$ is connective. An $\Ext$ spectral sequence computation shows that 
\[ \Ext_{\K}^{-1}\bigl(\K(\bar{\mathbb{F}}_l)_p \wedge_{\K}^L \K(\bar{\mathbb{F}}_l)_p, H\mathbb{Z}_p\bigr) = 0.\]
It follows that
\[\mathscr{D}_{\K}\bigl(\K(\bar{\mathbb{F}}_l)_p \wedge_{\K}^L \K(\bar{\mathbb{F}}_l)_p, \, S^2_{\K} \wedge_{\K}^L  \K(\bar{\mathbb{F}}_l)_p\bigr) \xrightarrow{\bar{u}_*}  \mathscr{D}_{\K}\bigl(\K(\bar{\mathbb{F}}_l)_p \wedge_{\K}^L \K(\bar{\mathbb{F}}_l)_p, \, \K(\bar{\mathbb{F}}_l)_p\bigr)  \]
is injective.
 It thus suffices to show that the equality holds after composing with $\bar{u}$.  Now, we can argue as in 
 \cite[Lemma 4.1]{Wat}. In \cite{Wat}  a similar statement is proven for the category of spectra (instead of $\mathscr{D}_{\K}$) and for 
 $q$ that generates $(\mathbb{Z}/p^2)^*$. 

\end{proof}

\section{The algebras $V(0)_*\K(\mathbb{F}_q)$ and $V(1)_*\K(\mathbb{F}_q)$} \label{V0andV1}

In this section we determine the multiplicative structure of $V(0)_*\K$ and $V(1)_*\K$. In \cite[Theorem 2.6]{Brow} 
Browder computes  the ring $V(0)_*\K$. Browder  works with spaces. In this section we present a computation using the $\K$-linearity of the distinguished triangle (\ref{fibseq}). 

\begin{lem}  \label{modphpty}
We have an isomorphism of $\mathbb{F}_p$-algebras
\[ V(0)_*\K \cong E(x) \otimes P(y),\]
where $|x| = 2r-1$ and $|y| = 2r$.  
\end{lem}
\begin{proof}
By Corollary \ref{fibseqcor} we have an $V(0)_*\K$-linear exact sequence
\begin{equation*} 
 \begin{tikzcd}
\Sigma^2V(0)_* \K(\bar{\mathbb{F}}_l)_p \ar{r}{\Delta} & V(0)_*\K \ar{r}{i_*} & V(0)_*\K(\bar{\mathbb{F}}_l)_p, 
\end{tikzcd}
\end{equation*} 
where $\Delta$ is a map of degree $-1$ and where  $i_*$ is a map of degree $0$ that is 
 a map of $\mathbb{F}_p$-algebras. 

The map  $i_{2r}$ is an isomorphism, so we can define $y$ to  be the preimage of $u^r$ under $i_*$. For  $n \geq 0$ we then have $i_*(y^n) = u^{rn}$.  In particular, we get $y^n \neq 0$. 

For $i \geq 1$ we have
\[ V(0)_{2ir-1}\K(\bar{\mathbb{F}}_l)_p = 0.\]
Hence,  
\[\begin{tikzcd}
\bigl(\Sigma^2V(0)_*\K(\bar{\mathbb{F}}_l)_p\bigr)_{2ir} \ar{r}{\Delta} &  V(0)_{2ir-1}\K
\end{tikzcd}\]
is an isomorphism.  We define $x \coloneqq \Delta(\Sigma^2u^{r-1})$. In order to prove the lemma it now suffices to show that $xy^i \neq 0$ for $i \geq 0$. 
We have 
\[ xy^i = \Delta(\Sigma^2u^{r-1})y^i = \Delta\bigr(y^i(\Sigma^2 u^{r-1})\bigl) = \Delta(\Sigma^2u^{r(i+1)-1}) \neq 0.\] 
\end{proof}
We define $k:= \frac{p-1}{r}$. 
\begin{lem} \label{v1K}
We have an isomorphism of $\mathbb{F}_p$-algebras
\[ V(1)_*\K \cong E(x) \otimes P_k(y).\]
\end{lem}
\begin{proof}
We have a long exact sequence
\[\begin{tikzcd}
\cdots \ar{r} &  \Sigma^{2p-2}V(0)_*\K \ar[equal]{d}  \ar{r}{v} & V(0)_*\K \ar{r}{\rho} \ar[equal]{d} & V(1)_*\K \ar{r} & \cdots, \\ 
& \Sigma^{2p-2}E(x) \otimes P(y) & E(x) \otimes P(y) & &  
 \end{tikzcd}\]
 where $\rho$ is  a map of $\mathbb{F}_p$-algebras. It suffices to show that $y^k$ is in the image of $v$, or equivalently that $v(\Sigma^{2p-2}1) \neq 0$. 
 For this, we consider the commutative diagram
 \[\begin{tikzcd}
\Sigma^{2p-2} V(0)_*\K \ar{r}{v} \ar{d}[swap]{\Sigma^{2p-2}i_*} & V(0)_*\K  \ar{d}{i_*}   &  \\
\Sigma^{2p-2} V(0)_*\K(\bar{\mathbb{F}}_l)_p \ar{r}{\bar{v}}  \ar[equal]{d} & V(0)_*\K(\bar{\mathbb{F}}_l)_p \ar{r} {\bar{\rho}}  \ar[equal]{d} &   V(1)_*\K(\bar{\mathbb{F}}_l)_p  \ar[equal]{d} \\
\Sigma^{2p-2}P(u) \ar{r} & P(u)   \ar{r} & P_{p-1}(u). 
 \end{tikzcd}\]
 To show $v(\Sigma^{2p-2}1) \neq 0$, it suffices to prove that 
 $ \bar{v}(\Sigma^{2p-2}1) \neq 0$.   This  follows  from $\bar{\rho}(u^{p-1}) = 0$. 
\end{proof}
\section{The mod $p$ homology of $\K(\mathbb{F}_{q})$} \label{modphomK}

In this section we study the mod $p$ homology of $\K$.   We define  $v$ to be the $p$-adic valuation of $q^r-1$. By \cite[Lemma VIII.2.4]{FiedPrid}  this is equal to the $p$-adic valuation of $q^{p-1}-1$. 
We distinguish the four different cases: 
\begin{enumerate}
\item $r = p-1$ and $v=1$, 
\item $r = p-1$ and $v \geq 2$, 
\item $r < p-1$ and $v \geq 2$, 
\item $r < p-1$ and $ v = 1$. 
\end{enumerate} 
The section is in part inspired by Hirata's article \cite{Hir} which treats the cohomology of algebraic $\K$-theory of finite fields
and by Angeltveit's  and Rognes' article \cite{AnRo} which treats  the mod $p$ homology of $\K$ in  case (\ref{1}). 

Similarly to \cite{Hir}  we first split the image of the distinguished triangle (\ref{fibseq})  under  $\mathscr{D}_{\K} \to \mathscr{D}_S$ into a wedge of $p-1$ distinguished triangles corresponding to the splitting of 
$\K(\bar{\mathbb{F}}_l)_p$ into a wedge of suspensions of the $p$-completed connective Adams summand. 

Let $q'$ be a  power of a prime $l'$ (possibly different from $l$) such that $q'$ is a generator of $(\mathbb{Z}/{p^2})^*$.   
 We set
\[k' = \cup_{i \geq 0} \mathbb{F}_{{q'}^{p^i}}.\] Then, $\ell_p \coloneqq \K(k')_p$ is a commutative $S$-algebra model for the $p$-completed connective Adams summand, \cite[Proposition 9.2]{McSt}, \cite[Section 2]{BakRichUni}. 
 We define 
\[k = \cup_{i \geq 0} \mathbb{F}_{{q'}^{p^i(p-1)}}\] 
and claim that there is a weak equivalence of commutative $S$-algebras $\K(k)_p \to \K(\bar{\mathbb{F}}_l)_p$: 
By Quillen's computations \cite[pp.583--585]{Qu} 
 one gets that $\K(k)_p \to \K(\bar{\mathbb{F}}_{l'})_p$ 
 is an isomorphism in $V(0)$-homotopy and therefore a weak equivalence. 
By \cite{BakRichUni} 
and since $\K(\bar{\mathbb{F}}_{l'})_p$ and $\K(\bar{\mathbb{F}}_l)_p$ are cofibrant commutative $S$-algebras, we get a weak equivalence $\K(\bar{\mathbb{F}}_{l'})_p \to \K(\bar{\mathbb{F}}_l)_p$.  
The morphism $j: \ell_p \to \K(k)_p$, given  by the inclusion of fields, 
induces the identification $\pi_*(\ell_p) = \mathbb{Z}_p[u^{p-1}]$ as a subring of
$\pi_*(\K(k)_p)$.  We have $V(0)_*\ell_p = P(u^{p-1}) $ and $V(1)_*\ell_p = \mathbb{F}_p$ as rings. Using the Tor and Ext  spectral sequence one concludes that $V(1) \wedge_S^L \ell_p \cong H\mathbb{F}_p$ as $S$-ring spectra. The 
maps 
\[\begin{tikzcd}
 S^{2i}_S \wedge_S^L \ell_p \ar{r}{u^i \wedge j} & \K(k)_p \wedge_S^L \K(k)_p \ar{r} &  \K(k)_p
 \end{tikzcd}\]
for $i = 0, \dots p-2$ induce an isomorphism in $\mathscr{D}_S$: 
\[\bigvee_{i = 0}^{p-2}  S^{2i}_S \wedge_S^L \ell_p \, \cong \, \K(\bar{\mathbb{F}}_l)_p.\]
We get an isomorphism $\bigvee_{i = 1}^{p-1} S^{2i}_S \wedge_S^L \ell_p  \to S^2_S \wedge_S^L \K(\bar{\mathbb{F}}_l)_p$ in $\mathscr{D}_S$ which we denote by $\kappa$. 
We claim that the following diagram commutes in $\mathscr{D}_S$: 
 \[\begin{tikzcd}
 \K(\bar{\mathbb{F}}_l)_p \ar{r}{f}  &  S^2_S \wedge_S^L \K(\bar{\mathbb{F}}_l)_p  \\
 \bigvee_{i=0}^{p-2} S^{2i}_S \wedge_S^L \ell_p  \ar{r}{\vee f_i} \ar{u}{\cong}  & (S^{2p-2}_S \wedge_S^L \ell_p) \vee \bigvee_{i = 1}^{p-2} S^{2i}_S \wedge_S^L \ell_p \ar{u}{\cong}[swap]{\kappa}.
 \end{tikzcd}
 \]
Here, denoting by $k_i$ the inclusion $S^{2i}_S \wedge_S^L \ell_p \to \K(\bar{\mathbb{F}}_l)_p$  and  by $p_i$   the projection $\K(\bar{\mathbb{F}}_l)_p \to S^{2i}_S \wedge_S^L \ell_p$, 
$f_i \coloneqq  p_i \circ (\phi^q-\id) \circ k_i$ for $i = 1, \dots, p-2$   and   $f_0$ is defined to be the unique map $\ell_p \to S^{2p-2}_S \wedge_S^L \ell_p$ that is mapped to 
$p_0 \circ (\phi^q - \id) \circ k_0$   after composing with 
\[\begin{tikzcd}[column sep = large]
\bar{v}: S^{2p-2}_S \wedge_S^L \ell_p \ar{r}{u^{p-1} \wedge \id} &  \ell_p \wedge_S^L \ell_p \ar{r} &  \ell_p. 
\end{tikzcd}\]
It suffices to show commutativity after composing with 
 \[\bar{u}: S^2_S \wedge_S^L \K(\bar{\mathbb{F}}_l)_p \xrightarrow{u \wedge \id} \K(\bar{\mathbb{F}}_l)_p \wedge_S^L \K(\bar{\mathbb{F}}_l)_p \to \K(\bar{\mathbb{F}}_l)_p. \]
 Commutativity then follows because $\bar{u}$ identifies under $\kappa$  with the map that is 
 $k_0 \circ \bar{v}$ on the wedge summand $S^{2p-2}_S \wedge_S^L \ell_p$ and $k_i$ on the other wedge summands, 
and because by  \cite[Corollary 6.4.8]{Adams}
 two self-maps of $p$-completed connective complex $\K$-theory in the stable homotopy category are equal if and only if they induce the same maps on homotopy groups. 
 Let $\K_i$ be the fiber of $f_i$. 
 We get a morphism of distinguished triangles in $\mathscr{D}_S$ 
\[\begin{tikzcd}
\K \ar{r}  & \K(\bar{\mathbb{F}}_l)_p \ar{r}{f}  &  S^2_S \wedge_S^L \K(\bar{\mathbb{F}}_l)_p \ar{r}  & \Sigma \K \  \\
 \bigvee_{i = 0}^{p-2} \K_i \ar{r} \ar[dashrightarrow]{u} & \bigvee_{i=0}^{p-2} S^{2i} \wedge_S^L \ell_p  \ar{r}{\vee f_i} \ar{u}{\cong}  & \bigvee_{i = 1}^{p-1} S^{2i}_S \wedge_S^L \ell_p \ar{r} \ar{u}{\cong}[swap]{\kappa} & \Sigma  (\bigvee_{i = 0}^{p-2} K_i) \ar[dashrightarrow]{u},
\end{tikzcd}\]
which is an isomorphism by the five lemma. 

Recall that the dual Steenrod algebra $A_* \coloneqq (H\mathbb{F}_p)_*H\mathbb{F}_p$  is an $\mathbb{F}_p$-Hopf algebra and that 
  for $X \in \mathscr{D}_S$ the mod $p$ homology $(H\mathbb{F}_p)_*X$ has a natural left $A_*$-comodule structure \cite[Theorem 1.1]{BakLaz}.  We use the letter $\nu$ to denote the  $A_*$-coactions.   
 For $X, Y \in \mathscr{D}_S$ the canonical map  
 \begin{equation} \label{kuenn}
 \begin{tikzcd}
  (H\mathbb{F}_p)_*X \otimes (H\mathbb{F}_p)_*Y \ar{r} & (H\mathbb{F}_p)_*(X \wedge_S^L Y)
  \end{tikzcd}
  \end{equation}
  is an isomorphism of comodules \cite[Theorem 17.8.vii]{Swi}.  We get that  $(H\mathbb{F}_p)_*X$ is an $A_*$-comodule algebra if $X$ is an associative $S$-ring spectrum.  If $X$ and $Y$ are associative $S$-ring spectra then (\ref{kuenn}) is an isomorphism of comodule algebras. 
 Recall that 
 \[A_* = P(\xi_1, \xi_2, \dots) \otimes E(\tau_0, \tau_1, \dots) = P(\bar{\xi}_1, \bar{\xi}_2, \dots) \otimes E(\bar{\tau}_0, \bar{\tau}_1, \dots)\] 
 with $|\xi_n| = |\bar{\xi}_n| = 2p^n-2$ and $|\tau_n| = |\bar{\xi}_n| = 2p^n-1$ \cite[Section 5.1]{AnRo}.
 Here, $\xi_n$ and $\tau_n$ are the generators defined in \cite{Mil}, and $\bar{\xi}_n$ and $\bar{\tau}_n$ are their images under the antipode of $A_*$. 
 We have 
 \[ \nu(\bar{\xi}_n) = \sum_{i+j = n} \bar{\xi}_i \otimes \bar{\xi}_j^{p^i}, \, \, \, \, \, \nu(\bar{\tau}_i) = 1 \otimes \bar{\tau}_n + \sum_{i+j = n}  \bar{\tau}_i \otimes \bar{\xi}_j^{p^i},\] 
 where by convention $\bar{\xi}_0 = 1$. 
 Since the coaction of $A_*$ is the comultiplication, $A_*$ has no non-trivial comodule primitives in positive degrees.  
 Since $\ell_p$ is a connective, cofibrant commutative $S$-algebra, we have a map $\ell_p \to H\mathbb{Z}_p$  in $\mathscr{C}\mathscr{A}_S$ that is the identity on $\pi_0$ \cite[Proposition IV.3.1]{EKMM}. The morphisms
 $\ell_p \to H\mathbb{Z}_p \to H\mathbb{F}_p$ induce the maps  
 \[ P(\bar{\xi}_1, \dots) \otimes E(\bar{\tau}_2, \dots) \subset P(\bar{\xi}_1, \dots)  \otimes E(\bar{\tau}_1, \dots) \subset A_* \] 
 in mod $p$ homology \cite[Proposition 5.3]{AnRo}.  Let $u \in (H\mathbb{F}_p)_2\K(\bar{\mathbb{F}}_l)_p$  be the image  of $u \in \pi_2(\K(\bar{\mathbb{F}}_l)_p)$ under the Hurewicz map $\pi_*(\K(\bar{\mathbb{F}}_l)_p) \to \pi_*(S \wedge_S^L \K(\bar{\mathbb{F}}_l)_p) \to \pi_*(H\mathbb{F}_p \wedge_S^L \K(\bar{\mathbb{F}}_l)_p)$.   Then, $j$ induces an isomorphism 
 \[ P_{p-1}(u) \otimes (H\mathbb{F}_p)_*\ell_p \cong (H\mathbb{F}_p)_*\K(\bar{\mathbb{F}}_l)_p  \] 
(see  \cite[Theorem 2.5]{Au}).

In the following we compute the mod $p$ homology of $\K$ by computing the mod $p$ homology of the $\K_i$ separately. 

\begin{lem} \label{facnul}
For $i \in \{0, \dots, p-2\}$ with  $r \nmid i $ we have $(H\mathbb{F}_p)_*\K_i = 0$. 
\end{lem}
\begin{proof}
 For $j \geq 0$  the map
 \[\begin{tikzcd}
 \pi_{2i+(2p-2)j}(f_i): \pi_{2i+(2p-2)j}(S^{2i}_S \wedge_S^L \ell_p) \ar{r} &  \pi_{2i+(2p-2)j}(S^{2i}_S \wedge_S^L \ell_p)
 \end{tikzcd}\] 
 is the multiplication with $q^{i+j(p-1)}-1$ on $\mathbb{Z}_p$. 
 We have $r \nmid i+j(p-1)$ and thus
 \[ q^{i+k(p-1)}-1 \in \mathbb{Z}/{p\mathbb{Z}} \cong \mathbb{Z}_p/{p\mathbb{Z}_p}\] 
 is not zero. 
Therefore, $q^{i+j(p-1)} - 1$ is a unit in $\mathbb{Z}_p$.
We get that $\pi_*(f_i)$ is an isomorphism and that 
 $\K_i \cong *$ in $\mathscr{D}_S$. 
 \end{proof}
Recall that we defined $k = \frac{p-1}{r}$. 
 \begin{lem} \label{midterm}
 For  $0 < j < k$ we have a short exact sequence
 \[\begin{tikzcd}
 0 \ar{r} & \Sigma^{2jr-1} (H\mathbb{F}_p)_*\ell_p \ar{r}{\Delta_j} & (H\mathbb{F}_p)_*\K_{jr} \ar{r}{i_j} & \Sigma^{2rj}(H\mathbb{F}_p)_*\ell_p \ar{r} & 0.
 \end{tikzcd}\]
 \end{lem}
 \begin{proof}
 It suffices to show that $(H\mathbb{F}_p)_*(f_{jr})$ is zero. 
From the homotopy of $S^{2jr}_S \wedge_S^L \ell_p$ we deduce that 
\[
\pi_*(\K_{jr}) =  \begin{cases}
                          0,  & \text{for} \, \, \, * < 2jr-1; \\ 
                       \mathbb{Z}_p/{(q^{rj}-1)} = \mathbb{Z}/{p^{v_p(q^{rj}-1)}}, & \text{for} \, \, \, * = 2jr-1 . 
\end{cases}
 \]
Using that $v_p(q^{rj}-1) \geq 1 $ and the  Tor spectral sequence we get that $({H\mathbb{F}_p})_{2rj-1}\K_{jr} = \mathbb{F}_p$.  It follows that $(H\mathbb{F}_p)_*(f_{jr})$ is zero in degree $2jr$. 
We show by induction  that $(H\mathbb{F}_p)_n(f_{jr}) = 0$ for all $n$. Let $n > 2jr$  and suppose that we already know that the claim is true  for all $m < n$. Let $b \in (H\mathbb{F}_p)_n(S^{2jr}_S \wedge_S^L\ell_p)$. 
Using the   induction hypothesis one sees that 
$(H\mathbb{F}_p)_*(f_{jr})(b)$ is an $A_*$-comodule primitive. Since it has  degree $> 2jr$,  it has to be zero. 
 \end{proof}
\begin{lem} \label{homolZmodpm}
Let $m \geq 1$. Then, we have a short exact sequence
\[\begin{tikzcd}
0 \ar{r} & (H\mathbb{F}_p)_*H\mathbb{Z} \ar{r} & (H\mathbb{F}_p)_*H(\mathbb{Z}/{p^m\mathbb{Z}}) \ar{r} & \Sigma (H\mathbb{F}_p)_*H\mathbb{Z} \ar{r} & 0.
\end{tikzcd}\]
For $m \geq 2$ the unique class $b \in (H\mathbb{F}_p)_1H(\mathbb{Z}/{p^m\mathbb{Z}})$ that is mapped to $\Sigma 1$ is an $A_*$-comodule primitive. 
\end{lem}
\begin{proof}
Using that $\mathscr{D}_S(X, Y) \cong \Hom_{\mathbb{Z}}(\pi_0(X), \pi_0(Y))$  for $S$-modules $X$ and $Y$ whose homotopy is concentrated in degree zero,  we get a map of distinguished triangles 
\[\begin{tikzcd}
H\mathbb{Z} \ar{r}{p^m} \ar{d}{p^{m-1}} & H\mathbb{Z} \ar{r} \ar{d}{\id} & H(\mathbb{Z}/{p^m\mathbb{Z}}) \ar{r} \ar{d}{g}  & \Sigma H\mathbb{Z} \ar{d}{\Sigma p^{m-1}}\\
H\mathbb{Z} \ar{r}{p} & H\mathbb{Z} \ar{r} & H(\mathbb{Z}/{p\mathbb{Z}}) \ar{r} & \Sigma H\mathbb{Z}, 
\end{tikzcd}\]
where $p^n \coloneqq p^n \id$.  Thus,  after applying $(H\mathbb{F}_p)_*(-)$, we get a map of long exact sequences.  Since $H\mathbb{F}_p \wedge_S^L(-)$  is additive, this shows the first part of the statement.  
Now, let $m \geq 2$.  We have that 
\[\begin{tikzcd}
\mathbb{F}_p \cong (H\mathbb{F}_p)_0H(\mathbb{Z}/{p^m\mathbb{Z}}) \ar{r}{g_0}  & (H\mathbb{F}_p)_0H(\mathbb{Z}/{p\mathbb{Z}})  \cong \mathbb{F}_p 
\end{tikzcd}\]
is an isomorphism and that
\[\begin{tikzcd}
\mathbb{F}_p \cong (H\mathbb{F}_p)_1H(\mathbb{Z}/{p^m\mathbb{Z}}) \ar{r}{g_1} & (H\mathbb{F}_p)_1H(\mathbb{Z}/{p\mathbb{Z}})   \cong \mathbb{F}_p 
\end{tikzcd}\] 
 is zero. 
Let $b$ be the generator of $(H\mathbb{F}_p)_1H(\mathbb{Z}/{p^m\mathbb{F}_p})$ that is mapped to $1$ under
\[ \begin{tikzcd}
(H\mathbb{F}_p)_1H(\mathbb{Z}/{p^m\mathbb{F}_p}) \ar{r} & (H\mathbb{F}_p)_0H\mathbb{Z}.
\end{tikzcd}\] 
We can write the coaction of $b$ as 
$\nu(b) = 1 \otimes b + \bar{\tau}_0 \otimes a$ 
for an element $a \in (H\mathbb{F}_p)_0H(\mathbb{Z}/{p^m\mathbb{F}_p})$.  Because of $g_1(b) = 0$ we have 
$\bar{\tau}_0 \otimes g_0(a) = 0$ and therefore $a = 0$. 
\end{proof}

 \begin{lem} \label{homK0}
 We have 
 \[\dim_{\mathbb{F}_p}(H\mathbb{F}_p)_{2p-2}\K_0 = \begin{cases}
                                                      0, & \text{if~} v = 1; \\
                                                      1, & \text{if~} v \geq 2.
                                                     \end{cases}\]
 \end{lem}
 \begin{proof}
 We have an exact sequence 
 \[\begin{tikzcd} 
 (H\mathbb{F}_p)_{2p-1}(S_S^{2p-2} \wedge_S^L\ell_p) \ar{r} \ar[equal]{d} & (H\mathbb{F}_p)_{2p-2}\K_0 \ar{r} & (H\mathbb{F}_p)_{2p-2}\ell_p \ar[equal]{d}.\\
 0  &  & \mathbb{F}_p\{\bar{\xi}_1\}
  \end{tikzcd}\]
  Therefore, we have $\dim_{\mathbb{F}_p}(H\mathbb{F}_p)_{2p-2}\K_0 \leq 1$. 
  Recall that for an $(n-1)$-connected $S$-module $X$ one has a map $X \to \Sigma^n H\pi_n(X)$ realizing the identity on $\pi_n$ \cite[Theorem II.4.13]{Ru}. 
Using this   we can inductively construct a Whitehead tower in $\mathscr{D}_S$:
\[\begin{tikzcd}
\dots \ar{r} & \K_0[3] \ar{r} \ar{d} & \K_0[2] \ar{r} \ar{d} & \K_0[1] \ar{r} \ar{d}  & \K_0[0] = \K_0  \ar{d}\\
             & \Sigma^3 H\pi_3(\K_0) & \Sigma^2H\pi_2(\K_0) & \Sigma H\pi_1(\K_0) & H\pi_0(\K_0).
\end{tikzcd}
\]
Here,
the sequences $\K_0[i+1] \to \K_0[i] \to \Sigma^iH\pi_i(\K_0)$  are part of distinguished triangles
      \[\K_0[i+1] \to \K_0[i] \to \Sigma^iH\pi_i(\K_0) \to \Sigma \K_0[i+1]. \]
Applying $(H\mathbb{F}_p)_*(-)$ we get an unrolled exact couple and therefore a spectral sequence $(E^*_{*,*}, d^*)$. 
Let  $m_n$ be the $p$-adic valuation of $q^{n(p-1)}-1$.  Then, the $E^1$-page of the spectral sequence is  $(H\mathbb{F}_p)_*H\mathbb{Z}_p$ in column $0$,  $\Sigma ^{2(n(2p-2)-1)}(H\mathbb{F}_p)_*H\mathbb{Z}/{p^{m_n}}$  in column $-(n(2p-2)-1)$ for $n > 0$  and  zero in all other columns. 
We claim that the spectral sequence converges strongly to $(H\mathbb{F}_p)_*\K_0$. 
Since $\K_0[i]$ is $(i-1)$-connected, we have $(H\mathbb{F}_p)_*\K_0[i] = 0$ for $* < i$. This implies that the spectral sequence converges conditionally to $(H\mathbb{F}_p)_*\K_0$ (see  \cite[Definition 5.10]{Boar}). 
Because $E^1_{*,*}$ is finite in every bidegree, the spectral sequence converges strongly by \cite[Theorem 7.1]{Boar}.  

Since $(H\mathbb{F}_p)_*(-) \to (H\mathbb{F}_p)_*(- \wedge S^1)$ is compatible with the comodule action, the spectral sequence is a spectral sequence of $A_*$-comodules. 


It is clear that  $d^i = 0$ for $i = 1, \dots, 2p-4$. 
Let $b \in (H\mathbb{F}_p)_1H(\mathbb{Z}/{p^{m_1}\mathbb{Z}}) \cong \mathbb{F}_p$ be a non-trivial class. In total degree $2p-2$ the $E^{2p-3}$-page is  a $2$-dimensional $\mathbb{F}_p$-vector space generated by $\bar{\xi}_1$ in column $0$ and by $\Sigma^{4p-6}b$ in column $-(2p-3)$. The class $\bar{\xi}_1$ survives to the $E^{\infty}$-page if and only if $d^{2p-3}(\bar{\xi}_1) = 0$.  The class $\Sigma^{4p-6}b$ survives to the $E^{\infty}$-page if and only if for the class $\bar{\tau}_1$ in column zero the equality  $d^{2p-3}(\bar{\tau}_1) = 0$ holds. 
We can write $d^{2p-3}(\bar{\tau}_1) = \lambda \cdot \Sigma^{4p-6}b$ for an element $\lambda \in \mathbb{F}_p$.  Let $\nu$ denote the coaction of the $E^{2p-3}$-page.  We have
\begin{equation} \label{coact}
\lambda \cdot \nu(\Sigma^{4p-6}b) = 1 \otimes d^{2p-3}(\bar{\tau}_1) + \bar{\tau}_0 \otimes d^{2p-3}(\bar{\xi}_1) + \bar{\tau}_1 \otimes \underbrace{d^{2p-3}(1)}_{= 0}. 
\end{equation}
Suppose that $v \geq 2$. Then, we have $m_1 \geq 2$ and $b \in  (H\mathbb{F}_p)_1H(\mathbb{Z}/{p^{m_1}\mathbb{Z}})$ is an $A_*$-comodule primitive by Lemma \ref{homolZmodpm}. It follows that $d^{2p-3}(\bar{\xi}_1) = 0$ and that $\dim_{\mathbb{F}_p}(H\mathbb{F}_p)_{2p-2}\K_0 = 1$. 

Now, suppose $v = 1$. Then,  $b$ is not  primitive. If $d^{2p-3}(\bar{\tau}_1)$ was zero, i.e. $\lambda = 0$, the equation (\ref{coact}) would imply that  $d^{2p-3}(\bar{\xi}_1) = 0$. One would get $\dim_{\mathbb{F}_p}(H\mathbb{F}_p)_{2p-2}\K_0 = 2$, which is a contradiction. Therefore, we have $d^{2p-3}(\bar{\tau}_1) \neq 0$. With equation (\ref{coact}) we get $d^{2p-3}(\bar{\xi}_1) \neq 0$. We conclude that $\dim_{\mathbb{F}_p}(H\mathbb{F}_p)_{2p-2}\K_0 = 0$. 
 \end{proof}
 \begin{lem}
 For $v \geq 2$ we have an exact sequence
 \[\begin{tikzcd} 
 0 \ar{r} &  \Sigma^{2p-3}(H\mathbb{F}_p)_*\ell_p \ar{r}{\Delta_0} & (H\mathbb{F}_p)_*\K_0 \ar{r}{i_0} & (H\mathbb{F}_p)_*\ell_p \ar{r} & 0.
 \end{tikzcd}\]
 \end{lem}
 \begin{proof}
 By Lemma \ref{homK0} we have an exact sequence 
  \[\begin{tikzcd} 
 (H\mathbb{F}_p)_{1}(\ell_p) \ar{r} \ar[equal]{d} & (H\mathbb{F}_p)_{2p-2}\K_0  \ar{r} \ar[equal]{d}  & (H\mathbb{F}_p)_{2p-2}\ell_p \ar[equal]{d} \ar{r}{(f_0)_{2p-2}} & (H\mathbb{F}_p)_{0}(\ell_p)  \ar[equal]{d} \\ 
 0  \ar{r} &  \mathbb{F}_p \ar{r} & \mathbb{F}_p\{\bar{\xi}_1\} \ar{r} & \mathbb{F}_p\{1\}. 
  \end{tikzcd}\]
We get that $(f_0)_{2p-2}$ 
  is zero.  As in Lemma \ref{midterm} it follows by induction that $(f_0)_{n} = 0$ for all $n$. 
 \end{proof}
 For $r = p-1$ the following result is in \cite[p.1265]{AnRo}.
 \begin{lem} \label{fibcomp}
 If  $v = 1$ we have an exact sequence
 \[ \begin{tikzcd}
 0 \ar{r} & \Sigma^{2p-3}\bar{\xi}_1^{p-1}C_* \ar{r}{\Delta_0} & (H\mathbb{F}_p)_*\K_0 \ar{r}{i_0} & C_* \ar{r} & 0, 
 \end{tikzcd}\]
 where $C_*$ is given by  $P(\bar{\xi}_1^p, \bar{\xi}_2, \dots) \otimes E(\bar{\tau}_2, \dots) \subset (H\mathbb{F}_p)_*\ell_p$.
 \end{lem}
 \begin{proof}
 Because of $(H\mathbb{F}_p)_{2p-2}\K_0 = 0$ the map 
 \[ (H\mathbb{F}_p)_{2p-2}\ell_p \xrightarrow{(f_0)_{2p-2}} (H\mathbb{F}_p)_{2p-2}(S^{2p-2}_S \wedge_S^L\ell_p)\]
 is an isomorphism. 
 Let  $\lambda$  be the unit in $\mathbb{F}_p$ such that $(f_0)_{2p-2}(\bar{\xi}_1) = \Sigma^{2p-2} (\lambda \cdot 1)$ holds. 
 We claim that the diagram 
 \begin{equation} \label{strcom}
 \begin{tikzcd} 
 (H\mathbb{F}_p)_*\ell_p \ar{r}{(f_0)_*} \ar[hookrightarrow]{d}   & \Sigma^{2p-2} (H\mathbb{F}_p)_*\ell_p \ar[hookrightarrow]{d} \\
 A_* \ar{r}{\varphi}  & \Sigma^{2p-2} A_* 
 \end{tikzcd}\end{equation}
 is commutative, where $\varphi$ is defined by 
 \[\begin{tikzcd}
  A_* \ar{r}{v} & A_* \otimes A_* \ar[twoheadrightarrow]{r}  & A_* \otimes \mathbb{F}_p\{\xi_1\} \ar{r}{\id \otimes (\lambda \cdot \id)} & A_* \otimes \mathbb{F}_p\{\bar{\xi}_1\} \ar{r}{\cong} & \Sigma^{2p-2} A_* .
  \end{tikzcd}\]
  It is clear that we have commutativity in degrees $\leq 2p-2$. For $V = A_*$ and $V = \mathbb{F}_p\{\bar{\xi}_1\}$ we equip $A_* \otimes V$ with the $A_*$-coaction given by the coaction of $A_*$. Then, all the maps in (\ref{strcom}) are maps of $A_*$-comodules. Since the difference of two maps of comodules is a map of comodules, it follows as in Lemma \ref{midterm} that (\ref{strcom}) is commutative. 
We have 
\[ \varphi(\bar{\xi}_1^{n_1} \bar{\xi}_2^{n_2}\bar{\xi}_3^{n_3} \dots \bar{\tau}_2^{\epsilon_2} \bar{\tau}_3^{\epsilon_3} \dots) = \Sigma^{2p-2} \lambda n_1 \bar{\xi}_1^{n_1-1} \bar{\xi}_2^{n_2}\bar{\xi}_3^{n_3} \dots \bar{\tau}_2^{\epsilon_2} \bar{\tau}_3^{\epsilon_3} \dots,\]
where the expression on the right means zero if $n_1 = 0$. 
 It follows that 
 \[ \ker(f_0)_* = P(\bar{\xi}_1^p, \bar{\xi}_2, \dots) \otimes E(\bar{\tau}_2, \dots) \]
 and 
 \[\coker(f_0)_* = \Sigma^{2p-2}\bar{\xi}_1^{p-1}P(\bar{\xi}_1^p, \bar{\xi}_2, \dots) \otimes E(\bar{\tau}_2, \dots). \qedhere\]
 \end{proof}

We now consider  case (\ref{1}). In this case $q$ is a generator of $(\mathbb{Z}/p^2)^*$ and we can take $q' = q$ in the definition of $\ell_p$. 
The map $\K \to \K(\bar{\mathbb{F}}_l)_p$ factors through $\ell_p$ and the  diagram 
\[\begin{tikzcd}
\K_0 \ar{r}  \ar{d} & \ell_p \ar[equal]{d} \\
\K \ar{r} & \ell_p
\end{tikzcd}\]
is commutative in $\mathscr{D}_S$.  Furthermore,  the left vertical map in this diagram is an isomorphism by the proof of Lemma 
\ref{facnul}.
We define $b \in (H\mathbb{F}_p)_{p(2p-2)-1}\K$ to be the image of $\Delta_0(\Sigma^{2p-3}\bar{\xi}_1^{p-1})$ under $(H\mathbb{F}_p)_*\K_0 \to (H\mathbb{F}_p)_*\K$. 
By Lemma \ref{fibcomp} there are unique classes $\tilde{\xi}_1^p \in (H\mathbb{F}_p)_{p(2p-2)}\K$ and $\tilde{\tau}_2 \in (H\mathbb{F}_p)_{2p^2-1}\K$ that map to $\bar{\xi}_1^p$ and $\bar{\tau}_2$ under
\[(H\mathbb{F}_p)_*\K \to (H\mathbb{F}_p)_*\ell_p.\]
Recall  from \cite[Theorem III.1.1]{BMMS} that the mod $p$ homology of a commutative $S$-algebra $R$ admits natural Dyer-Lashof operations
\[ Q^k: (H\mathbb{F}_p)_*R \to (H\mathbb{F}_p)_{* + k(2p-2)}R.\]
For $i \geq 2$ we recursively define 
\[ \tilde{\tau}_{i+1} \coloneqq Q^{p^i}(\tilde{\tau}_i) \in (H\mathbb{F}_p)_{2p^{i+1}-1}\K\]
Furthermore, we set 
\[\tilde{\xi}_i \coloneqq \beta(\tilde{\tau}_i) \in (H\mathbb{F}_p)_{2p^i-2}\K \]
for $i \geq 2$, where $\beta$ is the mod $p$ homology Bockstein homomorphism. 
We get  by \cite[Proposition 7.12]{AnRo}: 
\begin{prop}[Angeltveit, Rognes] \label{homimj}
In  case (\ref{1}) the $\mathbb{F}_p$-algebra map
\[ E(b) \otimes P(\tilde{\xi}_1^p, \tilde{\xi}_2, \dots) \otimes E(\tilde{\tau}_2, \dots) \to (H\mathbb{F}_p)_*\K\]
is an isomorphism. 
The class $b$ is an $A_*$-comodule primitive and we have
\[ \nu(\tilde{\xi}_2) = 1 \otimes \tilde{\xi}_2  + \bar{\xi}_1 \otimes \tilde{\xi}_1^p + \tau_1 \otimes b + \bar{\xi}_2 \otimes 1.   \]
The map $(H\mathbb{F}_p)_*\K \to (H\mathbb{F}_p)_*\ell_p$ maps  $\tilde{\xi}_1^p$, $\tilde{\xi}_i$, $\tilde{\tau}_i$
 and $b$  to  $\bar{\xi}_1^p$, $\bar{\xi}_i$,  $\bar{\tau}_i$ and zero, respectively. 
\end{prop}

In the following lemma we use the $\K$-linearity of the distinguished triangle (\ref{fibseq}) to determine the multiplicative structure of $(H\mathbb{F}_p)_*\K$ in  case (\ref{2}) and (\ref{3}). 
One could also  use an argument similar to the one that Angeltveit and Rognes use in  case (\ref{1}).

 \begin{prop} \label{homolK}
 In cases (\ref{2}) and (\ref{3}) there is an isomorphism of $\mathbb{F}_p$-algebras 
 \[(H\mathbb{F}_p)_*\K \cong E(x) \otimes P_k(y) \otimes P(\tilde{\xi}_1, \tilde{\xi}_2, \dots) \otimes E(\tilde{\tau}_2, \tilde{\tau}_3, \dots)\]
 for certain classes $x$, $y$, $\tilde{\xi}_i$ and $\tilde{\tau}_i$ with the following properties:
 \begin{itemize}
 \item The degrees are $|x| = 2r-1$, $|y| = 2r$, $|\tilde{\xi}_i| = 2p^i-2$ and $|\tilde{\tau}_i|  = 2p^i-1$. 
 \item The map $i_*: (H\mathbb{F}_p)_*\K \to (H\mathbb{F}_p)_*\K(\bar{\mathbb{F}}_l)_p$ maps $x$ to zero, $y$ to $u^r$, $\tilde{\xi}_i$ to $\bar{\xi}_i$  and $\tilde{\tau}_i$ to  $\bar{\tau}_i$. 
 \item For $i \geq 2$ we have $Q^{p^i}(\tilde{\tau}_i)  = \tilde{\tau}_{i+1}$ and $\beta(\tilde{\tau}_i) = \tilde{\xi}_i$. 
\item  For the $A_*$-coaction $\nu$ on $(H\mathbb{F}_p)_*\K$ we have: 
       \begin{eqnarray*}
       \nu(\tilde{\xi}_1)&  = &  1 \otimes \tilde{\xi}_1 + \bar{\xi}_1 \otimes 1 + a \bar{\tau}_0 \otimes xy^{k-1}  \text{~for an~} a \in \mathbb{F}_p , \\
     \nu(\tilde{\xi}_n) & = &  \sum_{i + j = n} \bar{\xi}_i \otimes \tilde{\xi}_j^{p^i}, \\
     \nu(\tilde{\tau}_n) & = &  1 \otimes \tilde{\tau}_n + \sum_{i+j = n} \bar{\tau}_i \otimes \tilde{\xi}_j^{p^i}.
       \end{eqnarray*}
       The classes $x$ and $y$ are comodule primitives.
 \end{itemize}
 \end{prop}
  Note that we have  $k = 1$ in case (\ref{2}), so that $P_k(y) = \mathbb{F}_p$. 
 \begin{proof}
 We have a commutative diagram with exact rows: 
 \[\begin{tikzcd}
&  \Sigma (H\mathbb{F}_p)_*\K(\bar{\mathbb{F}}_l)_p \ar{r}{\Delta} & (H\mathbb{F}_p)_*\K \ar{r}{i_*} & (H\mathbb{F}_p)_*\K(\bar{\mathbb{F}}_l)_p  &  \\
 0 \ar{r} &   \bigoplus_{j = 1}^k\Sigma^{2jr-1} (H\mathbb{F}_p)_*\ell_p \ar{r}{\oplus \, \Delta_{j}} \ar[hookrightarrow]{u} & \bigoplus_{j = 0}^{k-1}(H\mathbb{F}_p)_*\K_{jr} \ar{r}{\oplus \, i_{j}} \ar{u}{\cong} & \bigoplus_{j = 0}^{k-1} \Sigma^{2jr} (H\mathbb{F}_p)_*\ell_p  \ar[hookrightarrow]{u} \ar{r} & 0 .
  \end{tikzcd}\]
 The vertical maps are injections. We treat them as inclusions. 
 
 We set $x \coloneqq \Delta(\Sigma^{2r-1}1)$. 
 We have that  $\Sigma (H\mathbb{F}_p)_*\K(\bar{\mathbb{F}}_l)_p$ is zero in degree $2r$. 
 Therefore, for $k > 1$ there is a unique class $y \in (H\mathbb{F}_p)_{2r}\K$ such that $i_*(y) = u^r = \Sigma^{2r}1$. For $k = 1$ we set $y = 0$.  
 Since $\Sigma (H\mathbb{F}_p)_*\K(\bar{\mathbb{F}}_l)_p$ is zero in degree $2p-2$, there is a unique class $\tilde{\xi}_1 \in (H\mathbb{F}_p)_{2p-2}\K$ with $i_*(\tilde{\xi}_1) = \bar{\xi}_1 = \Sigma^0 \bar{\xi}_1$.
The  vector space $\Sigma^{2p-3}(H\mathbb{F}_p)_*\ell_p$ is zero in degree $2p^2-1$.  
Thus, there is a unique class $\tilde{\tau}_2 \in (H\mathbb{F}_p)_{2p^2-1}\K_0$ with $i_0(\tilde{\tau}_2) = \Sigma^0\bar{\tau}_2$. 
For $i \geq  2$ we define recursively 
\[\tilde{\tau}_{i+1} \coloneqq Q^{p^i}(\tilde{\tau}_i) \in (H\mathbb{F}_p)_*\K.\]
Furthermore, for $i \geq 2 $ we define  $\tilde{\xi}_i \coloneqq \beta(\tilde{\tau}_i)$. 
Since in the dual Steenrod algebra the analogous equations hold \cite[pp.1244--1245]{AnRo}, we get that $i_*(\tilde{\tau}_i) = \bar{\tau}_i$ and $i_*(\tilde{\xi}_i) = \bar{\xi}_i$. 
 

We have $i_*(y^k) = u^{p-1} = 0$. Since $\Sigma (H\mathbb{F}_p)_*\K(\bar{\mathbb{F}}_l)_p$ is zero in degree $2p-2$ it follows that $y^k = 0$. 
Thus, we get a map of graded commutative $\mathbb{F}_p$-algebras 
\[\begin{tikzcd}
h: E(x) \otimes P_k(y) \otimes P(\tilde{\xi}_1, \dots) \otimes E(\tilde{\tau}_2, \dots) \ar{r} & (H\mathbb{F}_p)_*\K.
\end{tikzcd}\]
We claim that $h$ is an isomorphism. 
Given numbers $j \in \{0, \dots, k-1\}$, $n_i \in \mathbb{N}$ for $i \geq 1$ and $\epsilon_i \in \{0, 1\}$  for $i \geq 2$ that are almost all equal to zero, we have
\begin{eqnarray*}
i_*(h(y^j \tilde{\xi}_1^{n_1}\tilde{\xi}_1^{n_2} \dots \tilde{\tau}_2^{\epsilon_2} \tilde{\tau}_3^{\epsilon_3} \dots))&  = u^{rj}\bar{\xi}_1^{n_1} \bar{\xi}_1^{n_2} \dots \bar{\tau}_2^{\epsilon_2} \bar{\tau}_3^{\epsilon_3} \dots \\
 & = \Sigma^{2rj}\bar{\xi}_1^{n_1} \bar{\xi}_1^{n_2} \dots \bar{\tau}_2^{\epsilon_2} \bar{\tau}_3^{\epsilon_3} \dots
\end{eqnarray*}  
and 
\begin{eqnarray*}
h(x y^j \tilde{\xi}_1^{n_1}\tilde{\xi}_1^{n_2} \dots \tilde{\tau}_2^{\epsilon_2} \tilde{\tau}_3^{\epsilon_3} \dots) 
& = &\Delta(\Sigma u^{r-1}) y^j \tilde{\xi}_1^{n_1}\tilde{\xi}_1^{n_2} \dots \tilde{\tau}_2^{\epsilon_2} \tilde{\tau}_3^{\epsilon_3} \dots \\
& = & \pm \Delta( \Sigma u^{r(j+1)-1} \bar{\xi}_1^{n_1} \bar{\xi}_2^{n_2} \dots \bar{\tau}_2^{\epsilon_2} \bar{\tau}_3^{\epsilon_3} \dots) \\
& = & \pm \Delta(\Sigma^{2r(j+1)-1}  \bar{\xi}_1^{n_1} \bar{\xi}_2^{n_2} \dots \bar{\tau}_2^{\epsilon_2} \bar{\tau}_3^{\epsilon_3} \dots).
\end{eqnarray*}
Here, the second equality uses that $\Delta$ is  $(H\mathbb{F}_p)_*\K$-linear.  Thus, $h$ maps the canonical basis to a basis and is therefore an isomorphism.

It remains to study the comodule structure.
We first recall some facts about the Steenrod algebra $A^*$ (see \cite{Mil}):   A basis of  $A^*$  is given by
\[ \beta^{\epsilon_0} \mathcal{P}^{s_1} \beta^{\epsilon_1} \dots \mathcal{P}^{s_k} \beta^{\epsilon_k},\] 
where $\epsilon_i \in \{0,1\}$ and $s_1 \geq p s_2 + \epsilon_1, s_2 \geq ps_3 + \epsilon_2, \dots, s_{k-1} \geq p s_k +\epsilon_{k-1}, s_{k} \geq 1$. 
 Here, $\mathcal{P}^i \in A^{2i(p-1)}$ denotes the Steenrod reduced $p$th power  and $\beta \in A^1$ denotes the Bockstein.  One has $\mathcal{P}^0 = 1$  and $A^*$ is generated as an algebra by 
 \[\beta, \mathcal{P}^1, \mathcal{P}^p, \mathcal{P}^{p^2}, \dots.\] We have a right action of the Steenrod algebra on the mod $p$ homology of an $S$-module $X$ (see \cite[p.244]{BakCo}): 
 For $a \in A^*$ and $z \in (H\mathbb{F}_p)_*X$ with coaction $\nu(z) = \sum \gamma_i \otimes z_i$ the element $z \cdot a = a_*(z)$ is defined by
 \[ z \cdot a = a_*(z) = (-1)^{|a| |x|} \left\langle  a | \gamma_i \right\rangle z_i. \]
Here, $A^*$ is considered as the $\mathbb{F}_p$-linear dual of $A_*$ and $\left\langle -, - \right\rangle: A^* \otimes A_* \to \mathbb{F}_p$ is the dual pairing. 
Note that because of $\mathscr{D}_S(H\mathbb{F}_p, \Sigma H\mathbb{F}_p) = A^1 \cong \mathbb{F}_p$,  we have that 
$\beta_*$ and the mod $p$ homology Bockstein $\beta$  agree degreewise up to a unit.

To  determine the comodule structure  of $(H\mathbb{F}_p)_*\K$ we will compute $a_*(z)$ for certain classes $z \in (H\mathbb{F}_p)_*\K$ and  $a \in A^*$.  We will use that $\beta$, $Q^i$ and $\mathcal{P}^i_*$ are linked via the Nishida relations (see \cite[Section 6]{Stein}). This allows to  prove the formulas for $\nu(\tilde{\xi}_i)$ and $\nu(\tilde{\tau}_i)$ by an inductive argument.  

 It is clear that $x$ is a comodule primitive, because $\Delta$ is compatible with the comodule action. For $k = 1$ the class $y$ is obviously primitive.  For $k > 1$ we can write 
\[ \nu(y) = 1 \otimes y + \lambda \bar{\tau}_0 \otimes x\]
for an element $\lambda \in \mathbb{F}_p$. To show $\lambda = 0$, we prove $\beta(y) = 0$:  
The unit $S \to H\mathbb{Z}$ induces a map $V(0)_*\K \to (H\mathbb{F}_p)_*\K$ that commutes with the Bockstein.  Because of $v \geq 2$  the proof of Lemma   \ref{UEC} shows that  the Bockstein $V(0)_{2r}\K  \to V(0)_{2r-1}\K$  maps $y$ to zero. Thus, it suffices to prove that $V(0)_*\K  \to (H\mathbb{F}_p)_*\K$ maps $y$ to $y$.   Since $(H\mathbb{F}_p)_*\K \to (H\mathbb{F}_p)_*\K(\bar{\mathbb{F}}_l)_p$ is injective in degree $2r$,  we only need to show that $V(0)_*\K(\bar{\mathbb{F}}_l)_p \to (H\mathbb{F}_p)_*\K(\bar{\mathbb{F}}_l)_p$ maps $u^r$ to $u^r$. This follows from the commutativity of the diagram 
\[ \begin{tikzcd}
\pi_*\bigl(\K(\bar{\mathbb{F}}_l)_p\bigr) \ar{r} \ar{rd} & V(0)_*\K(\bar{\mathbb{F}}_l)_p \ar{d} \\
  & (H\mathbb{F}_p)_*\K(\bar{\mathbb{F}}_l)_p.
\end{tikzcd}\]
For degree reasons and because of $i_*(\tilde{\xi}_1) = \bar{\xi}_1$ we have 
\begin{equation} \label{coxi1}
 \nu(\tilde{\xi}_1) = 1 \otimes \tilde{\xi}_1 + \bar{\xi}_1 \otimes 1 + a \bar{\tau}_0 \otimes xy^{k-1} 
 \end{equation}
for an element $a \in \mathbb{F}_p$. 
Since $i_*(\tilde{\xi}_2) = \bar{\xi}_2$ and since  $\ker i_*= \mathbb{F}_p\{x\} \otimes P_k(y) \otimes P(\tilde{\xi}_1, \dots) \otimes E(\tilde{\tau}_2, \dots)$, we get for degree reasons
\begin{equation} \label{coxi2}
\nu(\tilde{\xi}_2) =  \bar{\xi}_2 \otimes 1 + 1 \otimes \tilde{\xi}_2 + \bar{\xi}_1 \otimes \tilde{\xi}_1^p + \sum_{i+j=p} a_{ij} \bar{\tau}_0 \bar{\xi}_1^i \otimes \tilde{\xi}_1^j xy^{k-1} + \sum_{i+j = p-1} b_{ij} \bar{\tau}_1 \bar{\xi}_1^i \otimes \tilde{\xi}^j_1 xy^{k-1}
\end{equation}
for certain $a_{ij}, b_{ij} \in \mathbb{F}_p$. The classes $\bar{\tau}_0, \bar{\tau}_0\bar{\xi}_1, \dots, \bar{\tau}_0\bar{\xi}_1^p$ lie in the degrees $(2p-2)i +1$ for $i = 0, \dots, p$.  The classes $\bar{\tau}_1, \bar{\tau}_1\bar{\xi}_1, \dots, \bar{\tau}_1\bar{\xi}_1^{p-1}$ lie in the degrees $(2p-2)i+1$ for $i = 1, \dots, p$. A basis of the Steenrod algebra in these degrees is given by 
\[ \{\beta \mathcal{P}^i | 0 \leq i \leq p\} \cup \{ \mathcal{P}^i\beta | 1 \leq i \leq p\}.\]
To prove $a_{ij} = b_{ij} = 0$ we show $\beta \mathcal{P}_*^i(\tilde{\xi}_2) = 0$ for $i = 1, \dots, p$ and $\mathcal{P}^i_*\beta(\tilde{\xi}_2) = 0$ for $i = 0, \dots, p$. 
Because of  $\beta^2 = 0$ we have $\mathcal{P}^i_*(\beta(\tilde{\xi}_2))  = 0$.  
Since $A_{2p-2}$ is one-dimensional, (\ref{coxi2}) implies that $\mathcal{P}^1_*(\tilde{\xi}_2) \doteq \tilde{\xi}_1^p$. 
By (\ref{coxi1})  we have 
\[ \nu(\tilde{\xi}_1^p) = 1 \otimes \tilde{\xi}_1^p + \bar{\xi}_1^p \otimes 1.\]
Thus, $\beta(\tilde{\xi}_1^p) = 0$ and therefore $\beta(\mathcal{P}^1_*(\tilde{\xi}_2)) = 0$. 
For $2 \leq i \leq p$ the element $\nu(\tilde{\xi}_2)$ lies in 
\[ \bigoplus_{n+j = 2p^2-2, n \neq i(2p-2)} A_n \otimes (H\mathbb{F}_p)_j\K .\]
This implies that $\mathcal{P}^i_*(\tilde{\xi}_2) = 0$ and therefore that $\beta \mathcal{P}^i_*(\tilde{\xi}_2) = 0$. 
Because of $i_*(\tilde{\tau}_2) = \bar{\tau}_2$ we have 
\[ \nu(\tilde{\tau}_2) - \bigl(1 \otimes \tilde{\tau}_2 + \bar{\tau}_2 \otimes 1 + \bar{\tau}_1 \otimes \tilde{\xi}_1^p + \bar{\tau}_0 \otimes \tilde{\xi}_2\bigr)  \in A_* \otimes \ker i_*.\] 
Because of $\tilde{\tau}_2 \in (H\mathbb{F}_p)_*\K_0$ and $\beta(\tilde{\tau}_2) = \tilde{\xi}_2$ and since $(H\mathbb{F}_p)_*\K$ and $(H\mathbb{F}_p)_*\K_0$ are one-dimensional in the degrees $0$ and $p(2p-2)$,  this class also lies in 
$A_* \otimes (H\mathbb{F}_p)_*\K_0$. 
Using that 
\[\ker i_* \cap (H\mathbb{F}_p)_*\K_0 = \mathbb{F}_p\{xy^{k-1}\} \otimes P(\tilde{\xi}_1, \dots) \otimes E(\tilde{\tau}_2, \dots)\] 
one gets 
\[\nu(\tilde{\tau}_2) = 1 \otimes \tilde{\tau}_2 + \bar{\tau}_2 \otimes 1 + \bar{\tau}_1 \otimes \tilde{\xi}_1^p + \bar{\tau}_0 \otimes \tilde{\xi}_2 + \sum_{i+j = p-1} a_{ij}\bar{\tau}_0 \bar{\tau}_1 \bar{\xi}_1^i \otimes xy^{k-1}\tilde{\xi}_1^j \]
for certain $a_{ij} \in \mathbb{F}_p$. The elements \[\bar{\tau}_0\bar{\tau}_1, \bar{\tau}_0\bar{\tau}_1\bar{\xi}_1, \dots, \bar{\tau}_0\bar{\tau}_1\bar{\xi}_1^{p-1}\] lie in the degrees $i(2p-2)+ 2$ for $i = 1, \dots, p$. 
A basis of the Steenrod algebra in these degrees is given by 
\[\{ \beta \mathcal{P}^i \beta | 1 \leq i \leq p\}.\]
Since $\beta(\mathcal{P}^i_*(\beta(\tilde{\tau}_2))) =  \beta \mathcal{P}^i_*(\tilde{\xi}_2)= 0$ for $i = 1, \dots, p$ we get that $a_{ij} = 0$.
Therefore, we have proven the formulas 
\begin{eqnarray*}
  \nu(\tilde{\xi}_n) & = &  \sum_{i + j = n} \bar{\xi}_i \otimes \tilde{\xi}_j^{p^i}, \\
     \nu(\tilde{\tau}_n) & = &  1 \otimes \tilde{\tau}_n + \sum_{i+j = n} \bar{\tau}_i \otimes \tilde{\xi}_j^{p^i}
 \end{eqnarray*}
 for $n = 2$. We suppose that $n \geq 2$ and that we have shown these formulas for $n$. 
 We can write 
 \[\nu(\tilde{\xi}_{n+1}) = \sum_{i + j = n+1} \bar{\xi}_i \otimes \tilde{\xi}_j^{p^i}  + c \]
 for an element $c \in  A_* \otimes \ker i_* = A_* \otimes \mathbb{F}_p\{x\} \otimes P_k(y) \otimes P(\tilde{\xi}_1, \dots) \otimes E(\tilde{\tau}_2, \dots)$.  In order to show $c = 0$  it is enough to prove that 
 \begin{equation*} \label{etac}
 h_*(\tilde{\xi}_{n+1}) \in P(\tilde{\xi}_1, \dots) \otimes E(\tilde{\tau}_2, \dots)
 \end{equation*}
 for all  $h \in A^*$.  Since by the induction hypothesis 
 \[ h_*(\tilde{\xi}_n^p) \in P(\tilde{\xi}_1, \dots) \otimes E(\tilde{\tau}_2, \dots)\] 
 for all  $h \in A^*$, it suffices to show that 
  $\beta(\tilde{\xi}_{n+1}) = 0$, $\mathcal{P}^1_*(\tilde{\xi}_{n+1}) \doteq \tilde{\xi}_n^p$ and $\mathcal{P}^{p^i}_*(\tilde{\xi}_{n+1}) = 0$ for $i \geq 1$. 
 We have $\beta(\tilde{\xi}_{n+1}) = \beta(\beta(\tilde{\tau}_{n+1})) = 0$. By the Nishida relations \cite[Section 6]{Stein} we have 
\begin{eqnarray*}
\mathcal{P}^{p^i}_*(\tilde{\xi}_{n+1})& = & \mathcal{P}^{p^i}_*(\beta(Q^{p^n}(\tilde{\tau}_n))) \\
& = & \sum_j (-1)^{p^i + j} \binom{(p^n-p^i)(p-1)-1}{p^i- jp} \beta(Q^{p^n - p^i + j}(\mathcal{P}^j_*(\tilde{\tau}_n ))) \\
 & + & \sum_j (-1)^{p^i + j} \binom{(p^n- p^i)(p-1)-1}{p^i-j p -1}Q^{p^n-p^i + j}( \mathcal{P}^j_* (\beta(\tilde{\tau}_n))).
\end{eqnarray*}
By the induction hypothesis we have $\mathcal{P}^j_*(\tilde{\tau}_n) = 0$ for $j > 0$. Hence, the first sum is equal to 
\[  -\binom{(p^n-p^i)(p-1)-1}{p^i} \beta(Q^{p^n - p^i}(\tilde{\tau}_n )).\]
This is zero, because we have $2(p^n-p^i) < |\tilde{\tau}_n|$ which implies that $Q^{p^n-p^i}(\tilde{\tau}_n) = 0$ by \cite[Theorem III.1.1]{BMMS}. 
The summand 
\[ (-1)^{p^i + j} \binom{(p^n- p^i)(p-1)-1}{p^i-j p -1}Q^{p^n-p^i + j}( \mathcal{P}^j_* (\beta(\tilde{\tau}_n)))\] 
in the second sum is zero if $p^i-jp-1 < 0$, because then the binomial coefficient is zero. If $p^i-jp-1 > 0$ the summand is  zero as well, because in this case $2(p^n-p^i+j) < |P^j_*(\beta (\tilde{\tau}_n))|$. The equality  $p^i-jp-1 = 0$ is only possible if $i = 0$ and $j = 0$. In this case the summand is equal to 
$-Q^{p^n-1}(\tilde{\xi}_n)$. Because of $2(p^n-1) = |\tilde{\xi}_n|$ this is equal to  $- \tilde{\xi}_n^p$ by \cite[Theorem III.1.1]{BMMS}.
We now show $\mathcal{P}^{p^i}_*(\tilde{\tau}_{n+1}) = 0$ for all $i \geq 0$. Because of $\beta(\tilde{\tau}_{n+1}) = \tilde{\xi}_{n+1}$ this implies with the same argument as above the formula for the coaction of $\tilde{\tau}_{n+1}$. 
By the Nishida relations we have 
\begin{eqnarray*}
\mathcal{P}^{p^i}_*(\tilde{\tau}_{n+1}) & = & \mathcal{P}^{p^i}_*(Q^{p^n}(\tilde{\tau}_n)) \\
& = & \sum_j (-1)^{p^i + j} \binom{(p^n-p^i)(p-1)}{p^i-jp} Q^{p^n-p^i+j}(\mathcal{P}^j_*(\tilde{\tau}_n)).
\end{eqnarray*}
Since by the induction hypothesis $\mathcal{P}^j_*(\tilde{\tau}_n)$ is zero for $j > 0$, this is equal to 
\[(-1) \binom{(p^n-p^i)(p-1)}{p^i} Q^{p^n-p^i}(\tilde{\tau}_n).\] 
This is zero, because $2(p^n-p^i) < |\tilde{\tau}_n|$. 
 \end{proof}

\begin{lem} \label{c4hom}
We consider  case (\ref{4}).  Suppose that $r > 1$. 
Then there is map  of graded $\mathbb{F}_p$-algebras 
\[ \bigl(E(x) \otimes P_k(y)\bigr)/{xy^{k-1}} \to (H\mathbb{F}_p)_*\K\]
that is an isomorphism in degrees $\leq 2p$. Here, we have $|x| = 2r-1$ and $|y| = 2r$. 
\end{lem}
\begin{proof}
We have the following commutative diagram with exact rows:
\[\begin{tikzcd}
& \Sigma (H\mathbb{F}_p)_*\K(\bar{\mathbb{F}}_l)_p \ar{r}{\Delta} & (H\mathbb{F}_p)_*\K \ar{r}{i_*} & (H\mathbb{F}_p)_*\K(\bar{\mathbb{F}}_l)_p  & \\
0 \ar{r} & \bigoplus_{j = 1}^{k-1} \Sigma^{2rj-1}\mathbb{F}_p\{1\} \ar{r}{\oplus \Delta_j} \ar[hookrightarrow]{u} & \bigl(\bigoplus_{j = 0}^{k-1}(H\mathbb{F}_p)_*\K_{jr}\bigr)_{\leq 2p} \ar{r}{\oplus i_j} \ar[hookrightarrow]{u}  & \bigoplus_{j = 0}^{k-1} \Sigma^{2rj}\mathbb{F}_p\{1\}  \ar[hookrightarrow]{u} \ar{r} & 0.
\end{tikzcd} \]
The middle vertical map is an isomorphism in degrees $\leq 2p$ and 
we treat it  as an inclusion.  We define $x$ by $\Delta_1(\Sigma^{2r-1}1)$ and $y$ by the preimage of $\Sigma^{2r}1$ under $i_1$.  Since $(H\mathbb{F}_p)_*\K$ is zero in the degrees $2p-3$ and $2p-2$, we have $xy^{k-1} = 0$ and $y^k = 0$. 
Hence, we have a map of $\mathbb{F}_p$-algebras
\[ \bigl(E(x)\otimes P_k(y)\bigr)/{x y^{k-1}} \to (H\mathbb{F}_p)_*\K.\] 
In order to show that it is an isomorphism in degrees $\leq 2p$, it suffices to show that $xy^j \neq 0$ for $j = 0, \dots, k-2$ and that $y^j \neq 0$ for $j = 0, \dots, k-1$.  
For $j = 0, \dots, k-1$ we have $i_*(y^j) = u^{rj} \neq 0$ and hence $y^j \neq 0$. For $j = 0, \dots, k-2$ we have
\begin{eqnarray*}
x y^j = \Delta(\Sigma u^{r-1}) y^j = \Delta(\Sigma u^{r(j+1)-1}) = \Delta_{j+1}(\Sigma^{2r(j+1)-1}1) \neq 0      
\end{eqnarray*}
by $(H\mathbb{F}_p)_*\K$-linearity of $\Delta$ . 
\end{proof}
\section{Computations with the B\"okstedt spectral sequence} \label{Boekcompsec}

We first recall some facts about (topological) Hochschild homology and the Bökstedt spectral sequence.  See \cite[Section 2, 3 and 4]{AnRo} and \cite[Section 3 and 4]{Au} for more details. 

Let $k$ be a field, let $R$ be a graded-commutative $k$-algebra and  let $Q$ be a graded-commutative $R$-algebra. 
Then, the Hochschild homology $\mathbb{H}_{*,*}^{k}(R;Q)$ of $R$ with coefficients in $Q$ is the homology of the chain complex associated to the 
simplicial graded $k$-vector space given by
$ [n] \mapsto Q \otimes_k  R^{\otimes_k n} $
 and the usual face and degeneracy maps. Here,  note that we equip the category of graded $k$-vector spaces
with the symmetry $a \otimes b \mapsto (-1)^{|a| |b|} b \otimes a$, so that the last face map  includes signs.                                    
We write $\mathbb{H}_{*,*}^{\mathbb{F}_p}(R)$ for $\mathbb{H}_{*,*}^{\mathbb{F}_p}(R;R)$.
We have that $\mathbb{H}_{*,*}^k(R; Q)$ is an augmented $Q$-algebra, and if   $\mathbb{H}^{k}_{*,*}(R;Q)$ is flat over $Q$, then $\mathbb{H}^{k}_{*,*}(R;Q)$ is a $Q$-bialgebra.  The map $r \mapsto 1 \otimes r \in Q \otimes_{k} R$ defines a morphism 
$\sigma: R \to \mathbb{H}^{k}_{1,*}(R; Q )$ 
that satisfies the derivation rule \cite[p.1271]{Au}. 

Now, let $B$ be a  commutative $S$-algebra and let $C$ be a commutative $B$-algebra.  We implicitly assume that the necessary cofibrancy conditions are satisfied. The topological Hochschild homology of $B$ with coefficients in $C$, denoted by $\THH(B;C)$,  is the geometric realization of the simplicial $S$-module $[n] \mapsto C \wedge_S B^{\wedge_S n}$. 
We have that   $\THH(B;C)$ is an augmented commutative $C$-algebra.
   Moreover, in the stable homotopy category $\THH(B; C)$  admits the structure of a $C$-bialgebra. 
   In the stable homotopy category there is a morphism 
   $\sigma: \Sigma B  \to \THH(B)$.
   We denote the composition $\Sigma B \to \THH(B) \to \THH(B;C)$ also by $\sigma$.
   The by $\sigma$ induced map in mod $p$ homology satisfies the Leibniz rule \cite[Proposition 5.10]{AnRo}.

Recall  that the Bökstedt spectral sequence is a 
 a strongly convergent spectral sequence of the form 
\[E^2_{n,m} =  \mathbb{H}^{\mathbb{F}_p}_{n,m}\bigl((H\mathbb{F}_p)_*B; (H\mathbb{F}_p)_*C\bigr) \Longrightarrow (H\mathbb{F}_p)_*\THH(B;  C).\]  
 The spectral sequence is an $A_*$-comodule $(H\mathbb{F}_p)_*C$-algebra spectral sequence. 
 We  will use that the $A_*$-comodule structure on the $E^2$-page is induced by the following map on the Hochschild complex: 
       \begin{eqnarray*} 
       (H\mathbb{F}_p)_*C \otimes \bigl((H\mathbb{F}_p)_*B\bigr)^{\otimes n}  & \to &  A_* \otimes (H\mathbb{F}_p)_*C \otimes \bigl(A_* \otimes (H\mathbb{F}_p)_*B\bigr)^{\otimes n} \\
       & \to & A_*^{\otimes n} \otimes (H\mathbb{F}_p)_*C \otimes \bigl((H\mathbb{F}_p)_*B\bigr)^{\otimes n} \\
       & \to & A_* \otimes (H\mathbb{F}_p)_*C \otimes \bigl((H\mathbb{F}_p)_*B\bigr)^{\otimes n}.
       \end{eqnarray*}
       Here, the second map is given by the symmetry (including signs) and the third map is given by the multiplication of $A_*$. 
 If $(H\mathbb{F}_p)_*\THH(B;C)$ is flat over $(H\mathbb{F}_p)_*C$, it is an $A_*$-comodule $(H\mathbb{F}_p)_*C$-bialgebra. 
 If each term $E^r_{*,*}$ is flat over $(H\mathbb{F}_p)_*C$, the Bökstedt spectral sequence is an $A_*$-comodule $(H\mathbb{F}_p)_*C$-bialgebra spectral sequence, and if only the terms 
 \[ E^2_{*,*}, \dots, E^r_{*,*}\] 
 are flat over $(H\mathbb{F}_p)_*C$, then these are $A_*$-comodule $(H\mathbb{F}_p)_*C$-bialgebras and the differentials respect this structure. We have that the edge homomorphism 
    \[(H\mathbb{F}_p)_*C = E^2_{0,*} \to E^{\infty}_{0,*} \to (H\mathbb{F}_p)_*\THH(B;C) \]
      is the unit map.
As a consequence one gets that the zeroth step of the filtration splits off from $(H\mathbb{F}_p)_*\THH(B;C)$ naturally. 
For $x \in E^{\infty}_{1,*}$ we can therefore  choose a natural representative $[x]$ in the first step  of the filtration.
For $x \in (H\mathbb{F}_p)_*B$ we have
$\sigma_*(x) = [\sigma x]$ 
in $(H\mathbb{F}_p)_*\THH(B;C)$   \cite[Proposition 4.4]{Au}. 
If  $x \in E^{\infty}_{*,*}$  is an arbitrary class (not necessarily of filtration degree $1$) and if there is no non-trivial class in the same total degree  and lower filtration we will use the notation $[x]$ to denote the unique representative of $x$ in $(H\mathbb{F}_p)_*\THH(B;C)$.

\subsection{The first case} \label{Boek1}
In this subsection we consider case (\ref{1}). In this case Angeltveit and Rognes obtained the following result using the B\"okstedt spectral sequence \cite[Theorem 7.15]{AnRo}: 
\begin{thm}[Angeltveit, Rognes] \label{homthhimj}
In case (\ref{1}) we have an isomorphism of $(H\mathbb{F}_p)_*\K$-algebras
\[(H\mathbb{F}_p)_*\THH(\K) \cong (H\mathbb{F}_p)_*\K \otimes E([\sigma \tilde{\xi}_1^p], [\sigma \tilde{\xi}_2]) \otimes P([\sigma \tilde{\tau}_2]) \otimes \Gamma([\sigma b]).\]
\end{thm}

In Subsection \ref{et3} we apply the generalized Brun spectral sequence to compute the $V(1)$-homotopy of $\THH(\K)$ in case (\ref{1}). For the calculation we need the following lemma: 
\begin{lem} \label{primimj}
In case (\ref{1}) there is no non-trivial $A_*$-comodule primitive in 
\[ (H\mathbb{F}_p)_{2p^2-1}\bigl(V(1) \wedge_S^L \THH(\K)\bigr).\]
\end{lem}
\begin{proof}
We have an isomorphism of comodule algebras:
\[ (H\mathbb{F}_p)_*\bigl(V(1) \wedge_S^L \THH(\K)\bigr) \cong (H\mathbb{F}_p)_*V(1) \otimes (H\mathbb{F}_p)_*\THH(\K).\]
Recall that  $ (H\mathbb{F}_p)_*V(1) \cong E(\epsilon_0, \epsilon_1)$ 
 with $|\epsilon_0| = 1$ and $|\epsilon_1| = 2p-1$ and that $(H\mathbb{F}_p)_*V(1)$ contains no non-trivial comodule primitive in positive degree.  This follows from the observation that we can map $(H\mathbb{F}_p)_*V(1)$ injectively into $A_*$ via the map of comodule algebras
\[(H\mathbb{F}_p)_*V(1) \cong (H\mathbb{F}_p)_*(V(1) \wedge_S^L S) \to (H\mathbb{F}_p)_*(V(1) \wedge_S^L \ell_p).\]


 With Theorem \ref{homthhimj}  we get that an $\mathbb{F}_p$-basis of $(H\mathbb{F}_p)_{2p^2-1}\bigl(V(1) \wedge_S^L \THH(\K)\bigr)$ is given by the classes
 \[[\sigma \tilde{\xi}_2], \, \, \, \epsilon_1 [\sigma b], \, \, \, \tilde{\tau}_2, \, \,\, \epsilon_1 \tilde{\xi}_1^p, \,\, \, \epsilon_0\tilde{\xi}_2, \, \, \, \epsilon_0 \epsilon_1 b.\] 
 By Proposition \ref{homimj} the class $[\sigma b]$ is a comodule primitive and we have 
 \[ \nu([\sigma \tilde{\xi}_2]) = 1 \otimes [\sigma \tilde{\xi}_2] +  \bar{\xi}_1 \otimes [\sigma \tilde{\xi}_1^p] + \tau_1 \otimes [\sigma b]. \]
 Let $x \in (H\mathbb{F}_p)_{2p^2-1}\bigl(V(1) \wedge_S^L \THH(\K)\bigr)$ be a comodule primitive. We write 
 \[x = \lambda_1 [\sigma \tilde{\xi}_2] +  \lambda_2 \epsilon_1 [\sigma b] + \lambda_3 \tilde{\tau}_2 + \lambda_4 \epsilon_1 \tilde{\xi}_1^p + \lambda_5 \epsilon_0 \tilde{\xi}_2 + \lambda_{6} \epsilon_0\epsilon_1 b\]
 with $\lambda_i \in \mathbb{F}_p$. 
 The $A_*$-coaction of $x$ and the $A_*$-coaction of 
 \[  \lambda_2 \epsilon_1 [\sigma b] + \lambda_3 \tilde{\tau}_2 + \lambda_4 \epsilon_1 \tilde{\xi}_1^p + \lambda_5 \epsilon_0 \tilde{\xi}_2 + \lambda_{6} \epsilon_0\epsilon_1 b\] 
 lie in 
 \[A_* \otimes (H\mathbb{F}_p)_*V(1) \otimes (H\mathbb{F}_p)_*\K \otimes \, E([\sigma \tilde{\xi}_2]) \otimes \Gamma ([\sigma b]).\] 
 Since this is not true for $\nu([\sigma \tilde{\xi}_2])$ it follows that $\lambda_1 = 0$. 
 The $A_*$-coaction of $x$ and 
 \[\lambda_3 \tilde{\tau}_2 + \lambda_4 \epsilon_1 \tilde{\xi}_1^p + \lambda_5 \epsilon_0 \tilde{\xi}_2 + \lambda_{6} \epsilon_0\epsilon_1 b\]
 lie in 
 \[ A_* \otimes (H\mathbb{F}_p)_*V(1) \otimes (H\mathbb{F}_p)_*\K \,\,\,  \oplus \,\,\, (H\mathbb{F}_p)_*V(1) \otimes \mathbb{F}_p\{[\sigma b]\}. \]
 Since $\epsilon_1$ is not primitive, this is not the case for $\epsilon_1 [\sigma b]$. We get $\lambda_2 = 0$. 
 Hence, we have 
 \[x \in (H\mathbb{F}_p)_*V(1) \otimes (H\mathbb{F}_p)_*\K \cong (H\mathbb{F}_p)_*\bigl(V(1) \wedge^L_S \K\bigr).\]
 The class $x$ has to be in the kernel of 
 \[\begin{tikzcd}
  (H\mathbb{F}_p)_*\bigl(V(1) \wedge_S^L \K) \ar{r} & (H\mathbb{F}_p)_*\bigl(V(1) \wedge_S^L \ell_p\bigr),
  \end{tikzcd}\]
  because there is no non-trivial comodule primitive in $(H\mathbb{F}_p)_*\bigl(V(1) \wedge_S^L \ell_p\bigr) \cong A_*$ in positive degree.
  By Proposition \ref{homimj} the kernel is given by 
  \[ (H\mathbb{F}_p)_*V(1) \otimes \mathbb{F}_p\{b\} \otimes P(\tilde{\xi}_1^p, \tilde{\xi}_2, \dots) \otimes E(\tilde{\tau}_2, \dots). \]
  This implies $\lambda_3 = \lambda_4 = \lambda_5 = 0$. Since $b$ is a comodule primitive and $\epsilon_0\epsilon_1$ is not primitive, $\epsilon_0\epsilon_1 b$ is not primitive. We get $\lambda_6 = 0$ and therefore $x = 0$. 
\end{proof}

\subsection{The second case} \label{compBocs2}
In this subsection we compute the $V(1)$-homotopy of $\THH(\K)$ in case (\ref{2}). We first consider the B\"okstedt spectral sequences converging to $(H\mathbb{F}_p)_*\THH(\K; H\mathbb{Z}_p)$ and $(H\mathbb{F}_p)_*\THH(\K)$. 
 In this part we apply methods from \cite{AnRo} and \cite{Au}. Furthermore, we use the naturality of the B\"okstedt spectral sequence 
 with respect to the morphism $\K \to \K(\bar{\mathbb{F}}_\ell)_p$ and  Ausoni's results \cite{Au} about the B\"okstedt spectral sequences for connective  complex $K$-theory. After computing $(H\mathbb{F}_p)_*\THH(\K)$ we show that $V(1) \wedge_S^L \THH(\K)$ is a module over the $S$-ring spectrum $H\mathbb{F}_p$ and  we deduce the $V(1)$-homotopy of $\THH(\K)$. 
 
 In this subsection we use the map 
\[\begin{tikzcd}
 \K \ar{r} & \K(\bar{\mathbb{F}}_l)_p \ar{r} & H\mathbb{Z}_p 
 \end{tikzcd}\]
  to build $\THH(\K; H\mathbb{Z}_p)$. 
  We denote  by $(E^*_{*,*}, d^*)$ and  $(\tilde{E}_{*,*}^*, \tilde{d}^*)$ the B\"okstedt spectral sequences converging to $(H\mathbb{F}_p)_*\THH(\K; H\mathbb{Z}_p)$  and $(H\mathbb{F}_p)_*\THH(\K(\bar{\mathbb{F}}_l)_p; H\mathbb{Z}_p)$. 
We have a map of spectral sequences $E^*_{*,*} \to \tilde{E}^*_{*,*}$, which we denote by $i^*$. 
We define 
\[B_* \coloneqq (H\mathbb{F}_p)_*H\mathbb{Z}_p = P(\bar{\xi}_1, \dots) \otimes E(\bar{\tau}_1, \dots). \]
Recall from \cite[p.1283]{Au} that the map $
 (H\mathbb{F}_p)_*\K(\bar{\mathbb{F}}_l)_p \to B_*$ is given by 
$\bar{\xi}_i \mapsto \bar{\xi}_i$, $\bar{\tau}_i \mapsto \bar{\tau}_i$ and $u \mapsto 0$. 
Using standard facts about Hochschild homology (see \cite[Proposition 2.4]{AnRo}, \cite[Proposition 3.2]{Au}) we get the following: We have \[E^2_{*,*} = B_* \otimes E(\sigma \tilde{\xi}_1, \sigma \tilde{\xi}_2, \dots) \otimes \Gamma(\sigma    x, \sigma \tilde{\tau}_2, \sigma\tilde{\tau}_3, \dots)\]
as a $B_*$-algebra. Every $a \in \{\sigma \tilde{\xi}_1, \sigma \tilde{\xi}_2, \dots\}$ is a coalgebra primitive. 
For the classes $a \in \{\sigma x, \sigma \tilde{\tau}_2, \sigma \tilde{\tau}_3, \dots\}$ we have the following formula for the comultiplication:
\[\psi^2(\gamma_n(a)) = \sum_{i+j= n} \gamma_i(a) \otimes_{B_*} \gamma_j(a).\]
The class  $\gamma_n(\sigma x)$ is  represented in the Hochschild complex by the cycle 
$ 1 \otimes x^{\otimes n}$. 
 Since $x \in (H\mathbb{F}_p)_*\K$ is primitive, $\gamma_n(\sigma x)$ is a comodule  primitive. 
Because $\sigma$ is a derivation we get that 
the classes $\sigma \tilde{\xi}_n$ are comodule primitives for $n \geq 2$  and that the coactions on $\sigma \tilde{\xi}_1$ and $\sigma \tilde{\tau}_n$ are given by:
\begin{eqnarray*}
\nu^2(\sigma \tilde{\xi}_1) & = &  1 \otimes \sigma \tilde{\xi}_1 + a \bar{\tau}_0 \otimes \sigma x \\
\nu^2(\sigma \tilde{\tau}_n) & = & 1 \otimes \sigma \tilde{\tau}_n + \bar{\tau}_0 \otimes \sigma \tilde{\xi}_n.
\end{eqnarray*}
Here, $a$ is the element of $\mathbb{F}_p$ defined in Proposition \ref{homolK}. 
 Recall from \cite[Section 6]{Au} that 
\[ \tilde{E}^2_{*,*} = B_* \otimes E(\sigma u, \sigma \bar{\xi}_1,  \sigma \bar{\xi}_2, \dots) \otimes \Gamma(w, \sigma \bar{\tau}_2, \sigma \bar{\tau}_3, \dots),\]
where $w$ has the bidegree $|w| = (2, 2p-2)$. For $b \in B_*$ we have $i^2(b) = b$.  Furthermore, we have $i^2(\gamma_n(\sigma x)) = 0$ for $n \geq 1$, $i^2(\sigma \tilde{\xi}_n) = \sigma \bar{\xi}_n$ and  $i^2(\gamma_j(\sigma \tilde{\tau}_n)) = \gamma_j(\sigma \bar{\tau}_n)$. 

\begin{lem}
We have $d = 0$ for $i = 2, \dots, p-2$. 
\end{lem}
\begin{proof}
The $E^2$-page is generated as a $B_*$-algebra by 
\[ T \coloneqq \{ \sigma \tilde{\xi}_i \mid i \geq 1\} \cup  \{ \gamma_{p^i}(\sigma x) \mid i \geq 0\} \cup \{\gamma_{p^i} (\sigma \tilde{\tau}_j) \mid i \geq 0, j \geq 2\}. \]
Suppose that the spectral sequence has non-trivial differentials. Let $s_0$ be the minimal number such that $d^{s_0} \neq 0$ and let $b$ be a class of minimal total degree in $T$ with $d^{s_0}(b) \neq 0$. 
Because the classes $\sigma \tilde{\xi}_i$, $\sigma x$ and $\sigma \tilde{\tau}_j$ lie in the first column they cannot support  differentials. Thus, $b$ has filtration degree at least $p$.  
Because the differential is compatible with the coalgebra structure $d^{s_0}(b)$ has to  be a coalgebra primitive \cite[Proposition 4.8]{AnRo}.  
The $B_*$-module of coalgebra primitives  of $E^{2}_{*,*}$ is given by
\[ B_* \otimes \Bigl(\bigoplus_{i \geq 1} \mathbb{F}_p\{\sigma \tilde{\xi}_i\} \oplus \mathbb{F}_p\{\sigma x\} \oplus \bigoplus_{i \geq 2} \mathbb{F}_p\{\sigma \tilde{\tau}_2\}\Bigr).\]
It follows that $d^{s_0}(b)$ has filtration degree one. 
\end{proof}
\begin{lem} \label{diffx}
We have $d^{p-1}(\gamma_n(\sigma x)) = 0$ for all $n$.
\end{lem}
\begin{proof}
We assume that there is an $n$ with $d^{p-1}(\gamma_n(\sigma x)) \neq 0$. Let $n_0$ be minimal with this property. We must have $n_0 = p^j$ for a $j > 0$. Since $\gamma_{n_0}(\sigma x)$ is a comodule primitive, $d^{p-1}(\gamma_{n_0}(\sigma x))$ is also a comodule primitive. Because of the minimality of $n_0$ it follows from the formula for $\psi^2(\gamma_{n_0}(\sigma x))$ that $d^{p-1}(\gamma_{n_0}(\sigma x))$ is a coalgebra primitive.  The coalgebra primitives have filtration degree one. Therefore, we have $j = 1$. For degree reasons we get
\[ d^{p-1}(\gamma_p(\sigma x)) \in B_* \otimes \bigl(\mathbb{F}_p\{\sigma x\} \oplus \mathbb{F}_p\{\sigma \tilde{\xi}_1\}\bigr).\] If $a = 0$ the comodule primitives in this vector space are $\mathbb{F}_p\{\sigma x\} \oplus \mathbb{F}_p\{\sigma \tilde{\xi}_1\}$. If $a \neq 0$ the comodule primitives are $\mathbb{F}_p\{\sigma x\}$. Since the total degree of $d^{p-1}(\gamma_p(\sigma x))$ is different from the total degree of $\sigma x$ and from the total degree of $\sigma \tilde{\xi}_1$ we get a contradiction. 
\end{proof}
\begin{lem}
We have $d^{p-1}(\gamma_{p+i}(\sigma \tilde{\tau}_n)) \doteq \sigma \tilde{\xi}_{n+1} \gamma_i(\sigma \tilde{\tau}_n)$ for 
$i \geq 0$ and $n \geq 2$ and therefore
\[E^p_{*,*} = B_* \otimes E(\sigma \tilde{\xi}_1, \sigma \tilde{\xi}_2) \otimes P_p(\sigma \tilde{\tau}_2, \sigma \tilde{\tau}_3, \dots) \otimes \Gamma(\sigma x).\]
\end{lem}
\begin{proof}
We first  prove by induction on $n \geq 2$ that $d^{p-1}(\gamma_p(\sigma \tilde{\tau}_n)) \doteq \sigma \tilde{\xi}_{n+1}$. 
By \cite[Lemma 6.6]{Au} we have $\tilde{d}^{p-1}(\gamma_p(\sigma \bar{\tau}_2)) = \lambda_2 \sigma \bar{\xi}_3$ for a unit $\lambda_2 \in\mathbb{F}_p$. 
Because the kernel of 
\[\begin{tikzcd} 
E^{p-1}_{1, *} \ar{r}{i^{p-1}} & \tilde{E}_{1,*}^{p-1}
\end{tikzcd}\]
is given by $B_* \otimes \mathbb{F}_p\{\sigma x\}$ we get
\[ d^{p-1}(\gamma_p(\sigma \tilde{\tau}_2)) = \lambda_2 \sigma \tilde{\xi}_3 + b \sigma x\] 
for a class $b \in B_*$ of positive degree.  By Lemma \ref{diffx} every class in a total degree less than the total degree of $\gamma_p(\sigma \tilde{\tau}_2)$ has trivial $d^{p-1}$-differential. This implies that $d^{p-1}(\gamma_p(\sigma \tilde{\tau}_2))$ is a comodule primitive. Therefore, $b$ has to be zero. Assume that  we have proven the assertion for all $2 \leq m < n$. 
By comparing with the spectral sequence $\tilde{E}^*_{*,*}$ we get
\[ d^{p-1}(\gamma_p(\sigma \tilde{\tau}_n)) = \lambda_n \sigma \tilde{\xi}_{n+1} + b \sigma x\]
for a unit $\lambda_n \in \mathbb{F}_p$ and a class $b \in B_*$  of positive degree. 
We get
\[ \nu^{p-1}(d^{p-1}(\gamma_p(\sigma \tilde{\tau}_n)) = 1 \otimes \lambda_n\sigma \tilde{\xi}_{n+1} + \nu(b) \cdot 1 \otimes \sigma x.\] 
On the other hand we can write 
\[\nu^{p-1}(\gamma_p(\sigma \tilde{\tau}_n)) = 1 \otimes \gamma_p(\sigma \tilde{\tau}_n)  + \sum_i a_i \otimes b_i\]
for certain $a_i \in A_*$ and certain $b_i \in E^{p-1}_{p,*}$ whose internal degree is less than the internal degree of $\gamma_p(\sigma \tilde{\tau}_n)$. 
This implies that 
\[\nu^{p-1}(d^{p-1}(\gamma_p(\sigma \tilde{\tau}_n))) = 1 \otimes d^{p-1}(\gamma_p(\tilde{\tau}_n)) + \sum_i  a_i \otimes d^{p-1}(b_i).\]
By the induction hypothesis and by Lemma \ref{diffx} we have
\[d^{p-1}(b_i) = \sum_{3 \leq m \leq n}  b_{i,m} \sigma \tilde{\xi}_{m}\]
for certain $b_{i,m} \in B_*$. It follows that $b = 0$. This proves the induction step. 

We now fix $n \geq 2$. Suppose that $i \geq 1$ and that we have already shown
\[ d^{p-1}(\gamma_{p+j}(\sigma \tilde{\tau}_n)) = \lambda_n \sigma \tilde{\xi}_{n+1} \gamma_j(\sigma \tilde{\tau}_n)\]
for all $0 \leq j < i$. 
Then by the induction hypothesis 
\[d^{p-1}(\gamma_{p+i}(\sigma \tilde{\tau}_n)) - \lambda_n \sigma \tilde{\xi}_{n+1} \gamma_i(\sigma \tilde{\tau}_n)\]
is a coalgebra primitive. Because it lies in a  filtration degree $> 1$, it has to be  zero.  This proves the induction step and therefore the lemma. 
\end{proof}
\begin{lem} \label{Einfty}
We have $d^s =  0$ for all  $s \geq p$. Therefore, we get
\[E^{\infty}_{*,*} = B_* \otimes E(\sigma \tilde{\xi}_1, \sigma \tilde{\xi}_2) \otimes P_p(\sigma \tilde{\tau}_2, \sigma \tilde{\tau}_3, \dots) \otimes \Gamma(\sigma x).\]
\end{lem}
\begin{proof}
Suppose that the statement is wrong. Let $s_0 \geq p$ be the minimal number with $d^{s_0} \neq 0$ and let $i \geq 1$ be the minimal number with $d^{s_0}(\gamma_{p^i}(\sigma x)) \neq 0$. Then, $d^{s_0}(\gamma_{p^i}(\sigma x))$ is a comodule and coalgebra primitive in total degree $p^i(2p-2)-1$. The coalgebra primitives are given by
\[ B_* \otimes \bigl(\mathbb{F}_p\{\sigma \tilde{\xi}_1\} \oplus  \mathbb{F}_p\{\sigma \tilde{\xi}_2\} \oplus \bigoplus_{i \geq 2}\mathbb{F}_p\{\sigma \tilde{\tau}_i\} \oplus \mathbb{F}_p\{\sigma x\}\bigl). \]
If $a = 0$ the comodule primitives in this $\mathbb{F}_p$-vector space are
\[ \mathbb{F}_p\{\sigma \tilde{\xi}_1\} \oplus \mathbb{F}_p\{\sigma \tilde{\xi}_2\} \oplus \bigoplus_{i \geq 3} \mathbb{F}_p\{\sigma \tilde{\tau}_i\} \oplus \mathbb{F}_p\{\sigma x\}.\] 
If $a \neq 0$ the comodule primitive are given by 
\[\mathbb{F}_p\{\sigma \tilde{\xi}_2\} \oplus \bigoplus_{i \geq 3} \mathbb{F}_p\{\sigma \tilde{\tau}_i\} \oplus \mathbb{F}_p\{\sigma x\}. \] 
These classes all lie in  total degrees  different from $p^i(2p-2)-1$. Thus, we get a contradiction.
\end{proof}

Recall from \cite[Proposition 6.7]{Au} that we have 
\[ (H\mathbb{F}_p)_* \THH(\K(\bar{\mathbb{F}}_l)_p; H\mathbb{Z}_p) \cong B_* \otimes E([\sigma u], [\sigma \bar{\xi}_1]) \otimes P([w]).\] 
For degree reasons $[w]$ has to be a coalgebra primitive. 
\begin{thm} \label{multerw}
In case (\ref{2}) we have an isomorphism of $B_*$-algebras 
\[ (H\mathbb{F}_p)_*\THH(\K; H\mathbb{Z}_p) \cong B_* \otimes E([\sigma \tilde{\xi}_1], [\sigma \tilde{\xi}_2]) \otimes P([\sigma \tilde{\tau}_2]) \otimes \Gamma([\sigma x]).\]
\end{thm}

\begin{proof}

By \cite[Proposition 5.9]{AnRo} we have $\sigma_*Q^{p^i} = Q^{p^i}\sigma_*$.  Since $Q^{p^i}\tilde{\tau}_i = \tilde{\tau}_{i+1}$  one gets as in \cite[Theorem 5.12]{AnRo} or \cite[Lemma 5.2]{Au} that $[\sigma \tilde{\tau}_i]^p =  [\sigma \tilde{\tau}_{i+1}]$ for $i \geq 2$ . 
We show by induction on $i$ that we can find a class 
\[ \gamma_{p^i} \in (H\mathbb{F}_p)_{p^i(2p-2)}\THH(\K; H\mathbb{Z}_p)\]
that represents the class $\gamma_{p^i}(\sigma x)$ in $E^{\infty}_{*,*}$   and that has the property $\gamma_{p^i}^p = 0$. 
We define $\gamma_1 \coloneqq [\sigma x]$. Then, we have
$\gamma_1^p =  \sigma_*(Q^{p-1}(x))$. 
We claim that $Q^{p-1}(x) = 0$. For degree reasons we have $Q^{p-1}(x) \in \mathbb{F}_p\{x\tilde{\xi}_1^{p-1}\}$. 
The comodule action of $x \tilde{\xi}_1^{p-1}$ is given by
\[ \nu(x\tilde{\xi}_1^{p-1}) = \sum_j \binom{p-1}{j} \bar{\xi}_1^j \otimes \tilde{\xi}_1^{p-1-j}x.\]
Since the Steenrod algebra is one-dimensional in degree $(p-1)(2p-2)$  we have
  $\mathcal{P}^{p-1}_*(x \tilde{\xi}_1^{p-1}) \doteq x$. 
 On the other hand, we have 
 \[ \mathcal{P}^{p-1}_*(Q^{p-1}(x)) = \sum_j (-1)^{p-1+j} \binom{0}{p-1-jp}Q^j\mathcal{P}^j_*(x) = 0 \]  
 by the Nishida relations. Hence, we can conclude $\gamma_1^p = 0$. 
Suppose that $i > 0$ and that we have already shown the assertion for all $0 \leq j < i$. 
It suffices to show that we can find a representative $\gamma_{p^i}$ for $\gamma_{p^i}(\sigma x)$ that has the property that 
$\gamma_{p^i}^p$ is a comodule and coalgebra primitive: Every non-trivial comodule and coalgebra primitive of $(H\mathbb{F}_p)_*\THH(\K; H\mathbb{Z}_p)$ gives a non-trivial comodule and coalgebra primitive in $E^{\infty}_{*,*}$. By the proof of Lemma \ref{Einfty} the simultaneous coalgebra and comodule primitives of $E^{\infty}$ lie in the total degrees $2p-2$, $2p-1$, $2p^2-1$ and $2p^j$ for $j \geq 3$, which are all different from $p^{i+1}(2p-2)$. 
By the induction hypothesis we have a map of $B_*$-algebras
\[ B_* \otimes E([\sigma \tilde{\xi}_1], [\sigma \tilde{\xi}_2]) \otimes P([\sigma \tilde{\tau}_2]) \otimes P_p(\gamma_1, \dots, \gamma_{p^{i-1}}) \to (H\mathbb{F}_p)_*\THH(\K; H\mathbb{Z}_p)\]
which is injective and an isomorphism in degree $< p^i(2p-2)$. 
 For a graded-commutative $\mathbb{F}_p$-algebra we denote by $I_p$ the homogeneous ideal of all elements $x$ with $x^p = 0$.  
 The map
\[ P(\bar{\xi}_1, \dots) \otimes P([\sigma \tilde{\tau}_2]) \to (H\mathbb{F}_p)_*\THH(\K; H\mathbb{Z}_p)/{I_p}\]
is an isomorphism in degrees $< p^i(2p-2)$. 
Furthermore, the map 
\[ P(\bar{\xi}_1, \dots) \otimes P([w]) \to (H\mathbb{F}_p)_*\THH\bigl(\K(\bar{\mathbb{F}}_l)_p; H\mathbb{Z}_p\bigr)/{I_p}\]
and the map from 
\[ P(\bar{\xi}_1, \dots) \otimes P([w]) \otimes_{P(\bar{\xi}_1, \dots)} P(\bar{\xi}_1, \dots) \otimes P([w])\]
 to \[
 \Bigl((H\mathbb{F}_p)_*\THH\bigl(\K(\bar{\mathbb{F}}_l)_p; H\mathbb{Z}_p\bigr) \otimes_{B_*} (H\mathbb{F}_p)_*\THH\bigl(\K(\bar{\mathbb{F}}_l)_p; H\mathbb{Z}_p\bigr)\Bigr)/{I_p}\]
 are isomorphisms. 
First,  let $\gamma_{p^i}$ be an arbitrary representative for $\gamma_{p^i}(\sigma x)$. We can assume that it is in the kernel of the augmentation
\[ \epsilon: (H\mathbb{F}_p)_*\THH(\K; H\mathbb{Z}_p) \to  B_*.\] 
We consider its image under 
\[ i_*: (H\mathbb{F}_p)_*\THH\bigl(\K; H\mathbb{Z}_p\bigr) \to (H\mathbb{F}_p)_*\THH\bigl(\K(\bar{\mathbb{F}}_l)_p; H\mathbb{Z}_p\bigl).\]
We have 
\[ i_*(\gamma_{p^i}) = \sum_j a_j[w]^j \text{~modulo~} I_p \]
for certain $a_j \in P(\bar{\xi}_1, \dots)$  with  $|a_j| = p^i(2p-2)-2pj$. 
Since $i_*(\gamma_{p^i})$ lies in the kernel of the augmentation, we have $a_0 = 0$. We show that $a_j = 0$ for all $j$ that are not divisible by $p$: 
The degree of every non-zero class in $P(\bar{\xi}_1, \dots)$ is divisible by $2p-2$. Thus, $P(\bar{\xi}_1, \dots)$ is zero in degree $p^i(2p-2)-2p$  and we get $a_1 = 0$. 
We have 
\begin{equation} \label{comult1}
 \psi(i_*(\gamma_{p^i})) = \sum_j a_j \sum_{n = 0}^j \binom{j}{n} [w]^n \otimes_{B_*} [w]^{j-n}  \text{~modulo~} I_p.
 \end{equation}
On the other hand, since $\gamma_{p^i}$ is in $\ker \epsilon$ we can write 
\[ \psi(\gamma_{p^i}) = 1 \otimes_{B_*} \gamma_{p^i} + \gamma_{p^i} \otimes_{B_*} 1 + \sum_j b_j \otimes_{B_*} c_j\]
for certain $b_j, c_j \in (H\mathbb{F}_p)_*\THH(\K; H\mathbb{Z}_p)$ with $|b_j|, |c_j| < p^i(2p-2)$. It follows that 
\[ \psi(\gamma_{p^i}) = 1 \otimes_{B_*} \gamma_{p^i} + \gamma_{p^i} \otimes_{B_*} 1 + \sum_{n,m} c_{n,m} [\sigma \tilde{\tau}_2]^n \otimes_{B_*} [\sigma \tilde{\tau}_2]^m \text{~modulo~} I_p\] 
for certain $c_{n,m} \in P(\bar{\xi}_1, \dots)$.   Applying $i_*$ and using that by \cite[Lemma 6.5]{Au}  the relation $[w]^p = [\sigma \bar{\tau}_2]$  holds in 
$(H\mathbb{F}_p)_*\THH\bigl(\K(\bar{\mathbb{F}}_l)_p; H\mathbb{Z}_p)$,  we get
\begin{eqnarray} \label{comult2}
 \psi(i_*(\gamma_{p^i}))  & = & 1 \otimes_{B_*} \sum_j a_j[w]^j + \sum_j a_j[w]^j \otimes_{B_*} 1 \nonumber \\
&  & + \sum_{n,m} c_{n,m}[w]^{pm} \otimes_{B_*} [w]^{pn} \text{~~~~~~~~~~~~~modulo~} I_p.
 \end{eqnarray} 
Let $j \neq 1$ be a natural number that is not divisible by $p$. In (\ref{comult1}) the coefficient of $ [w] \otimes [w]^{j-1}$ 
is $j \cdot a_j$ and in (\ref{comult2}) it is zero. We get that $a_j = 0$. 
We conclude that we have 
\[ i_*(\gamma_{p^i}) = \sum_{j \geq 1}  a_{jp} [w]^{jp} \text{~~~~modulo~~} I_p,\] 
where $a_{jp}$ is zero for  $jp > p^{i-1}(p-1)$. The class $\sum_{j \geq 1} a_{jp}[\sigma \tilde{\tau}_2]^j$ lies in filtration $< p^i$. Thus, the element $\gamma_{p^i} - \sum_{j \geq 1} a_{jp}[\sigma \tilde{\tau}_2]^j$ is also a representative for $\gamma_{p^i}(\sigma x)$, it lies in the kernel of the augmentation and  it satisfies 
\[i_*(\gamma_{p^i} - \sum_j a_{jp}[\sigma \tilde{\tau}_2]^j)^p = 0.\] 
We replace $\gamma_{p^i}$ by  $\gamma_{p^i} - \sum_j a_{jp}[\sigma \tilde{\tau}_2]^j$ and denote this class again by $\gamma_{p^i}$. 
Because the maps $(i_* \otimes_{B_*} i_*)/ {I_p}$ and $(\id_{A_*} \otimes i_*)/ {I_p}$ are injective on the images of 
\[ \bigoplus_{n < p^i(2p-2), m < p^i(2p-2)} (H\mathbb{F}_p )_n\THH(\K; H\mathbb{Z}_p) \otimes (H\mathbb{F}_p)_m\THH(\K; H\mathbb{Z}_p) \]
and
\[  \bigoplus_{m < p^i(2p-2)}  A_n \otimes (H\mathbb{F}_p)_m\THH(\K; H\mathbb{Z}_p) \]
in 
\[ \Bigl({H\mathbb{F}_p }_*\THH(\K; H\mathbb{Z}_p) \otimes_{B_*} (H\mathbb{F}_p)_*\THH(\K; H\mathbb{Z}_p)\Bigr)/{I_p}\]
and 
\[ \Bigl(A_* \otimes (H\mathbb{F}_p)_*\THH(\K; H\mathbb{Z}_p)\Bigr)/{I_p},\]
$\gamma_{p^i}^p$ now is a comodule and coalgebra primitive. 
\end{proof}
We now study the B\"okstedt spectral sequence $(E^*_{*,*}, d^*)$   converging to $(H\mathbb{F}_p)_*\THH(\K)$.  Similarly as above, one sees that
\[ E_{*,*}^2 = (H\mathbb{F}_p)_*\K \, \otimes \, E(\sigma \tilde{\xi}_1, \sigma \tilde{\xi}_2, \dots) \otimes  \Gamma(\sigma x, \sigma \tilde{\tau}_2, \sigma \tilde{\tau}_3, \dots).\]
The classes $\sigma \tilde{\xi}_i$ are coalgebra primitives. For the classes $a \in \{\sigma x, \sigma \tilde{\tau}_2, \sigma \tilde{\tau}_3, \dots\}$ we have the following formula for the comultiplication:
\[\psi^2(\gamma_n(a)) = \sum_{i+j= n} \gamma_i(a) \otimes_{(H\mathbb{F}_p)_*\K} \gamma_j(a).\]
For $n \geq 2$ the classes $\sigma \tilde{\xi}_n$ are comodule primitives. All the classes $\gamma_n(\sigma x)$ are comodule primitives. The coactions on $\sigma \tilde{\xi}_1$ and $\sigma \tilde{\tau}_n$ are given by:
\begin{eqnarray*}
\nu^2(\sigma \tilde{\xi}_1) & = &  1 \otimes \sigma \tilde{\xi}_1 + a \bar{\tau}_0 \otimes \sigma x \\
\nu^2(\sigma \tilde{\tau}_n) & = & 1 \otimes \sigma \tilde{\tau}_n + \bar{\tau}_0 \otimes \sigma \tilde{\xi}_n.
\end{eqnarray*}

\begin{lem}
For $2 \leq s \leq p-2$ the differential $d^s$ vanishes. We have
\[d^{p-1}(\gamma_n(\sigma x)) = 0\] for all $n$ and 
\[ d^{p-1}(\gamma_{p+n}(\sigma \tilde{\tau}_i)) = \sigma \tilde{\xi}_{i+1} \gamma_n(\sigma \tilde{\tau}_i)\]
for all $n \geq 0$ and $i \geq 2$. Therefore, we get
\[E^p_{*,*} = (H\mathbb{F}_p)_*\K \, \otimes \, E(\sigma \tilde{\xi}_1, \sigma \tilde{\xi}_2) \otimes P_p(\sigma \tilde{\tau}_2, \sigma \tilde{\tau}_3, \dots) \otimes \Gamma(\sigma x).\] 
\end{lem}
\begin{proof}
By \cite[Proposition 5.6]{AnRo} the differentials $d^s$ vanish for $2 \leq s \leq p-2$ and we have
\[ d^{p-1}(\gamma_p(\sigma x)) = \sigma \beta Q^{p-1}x
\text{~~~and~~~}
d^{p-1}(\gamma_p(\sigma \tilde{\tau}_i)) = \sigma \beta Q^{p^i} \tilde{\tau}_i.\]
The proof of  Theorem \ref{multerw} shows that $Q^{p-1}(x) = 0$.   Using
 Proposition \ref{homolK} we get
 \[d^{p-1}(\gamma_p(\sigma x)) = 0 \text{~~~and~~~}  d^{p-1}(\gamma_p(\sigma \tilde{\tau}_i)) = \sigma \tilde{\xi}_{i+1}.\]
Note that the coalgebra primitives of the $E^2 = E^{p-1}$-page are given by
\[ (H\mathbb{F}_p)_*\K \, \otimes \, \bigl(\bigoplus_{i \geq 1}\mathbb{F}_p\{\sigma \tilde{\xi}_i\} \oplus  \mathbb{F}_p\{\sigma x\} \oplus \bigoplus_{i \geq 2} \mathbb{F}_p\{\sigma \tilde{\tau}_i\}\bigr).\]
By induction on $n$ one proves that 
\[ d^{p-1}(\gamma_{p+n}(\sigma x)) = 0  \text{~~~and~~~}  d^{p-1}(\gamma_{p+n}(\sigma \tilde{\tau}_i)) - \sigma \tilde{\xi}_{i+1} \gamma_n(\sigma \tilde{\tau}_i) = 0. \] 
The induction step follows because the classes are coalgabra primitives in filtration degree $>  1$. 
\end{proof}
\begin{lem} \label{EinftyBSS2}
We have $d^s = 0$ for all $s \geq p$. 
Therefore, we have
\[E^{\infty}_{*,*} = (H\mathbb{F}_p)_*\K \, \otimes \, E(\sigma \tilde{\xi}_1, \sigma \tilde{\xi}_2) \otimes P_p(\sigma \tilde{\tau}_2, \sigma \tilde{\tau}_3, \dots) \otimes \Gamma(\sigma x).\]
\end{lem}
\begin{proof}
This follows as in Lemma \ref{Einfty}  noticing that 
 the coalgebra primitives of the $E^p$-page are given by 
 \[(H\mathbb{F}_p)_*\K \, \otimes \, \bigl(\mathbb{F}_p\{\sigma \tilde{\xi}_1\} \oplus \mathbb{F}_p\{\sigma \tilde{\xi}_2\}  \oplus \bigoplus_{i \geq 2} \mathbb{F}_p\{\sigma \tilde{\tau}_i\} \oplus \mathbb{F}_p\{\sigma x\}\bigr) \]
 and that the comodule primitives in this vector space are a subspace of
 \begin{eqnarray*}
 &  & \mathbb{F}_p \{\sigma \tilde{\xi}_1\} \oplus \mathbb{F}_p\{x \sigma \tilde{\xi}_1\} \oplus \mathbb{F}_p\{\sigma \tilde{\xi}_2\} \oplus \mathbb{F}_p\{x \sigma \tilde{\xi}_2\} \\
 & \oplus &  \bigoplus_{j \geq 3}\mathbb{F}_p\{\sigma \tilde{\tau}_j\} \oplus \bigoplus_{j \geq 3}\mathbb{F}_p\{x \sigma \tilde{\tau}_j\} \oplus \mathbb{F}_p\{\sigma x\} \oplus \mathbb{F}_p\{x \sigma x\}. 
  \end{eqnarray*}
\end{proof}
\begin{thm} \label{modphomolTHHK}
In case (\ref{2}) we have an isomorphism of $(H\mathbb{F}_p)_*\K$-algebras
\[ (H\mathbb{F}_p)_*\THH(\K) \cong (H\mathbb{F}_p)_*\K \, \otimes \, E([\sigma \tilde{\xi}_1], [\sigma \tilde{\xi}_2]) \otimes P([\sigma \tilde{\tau}_2]) \otimes \Gamma([\sigma x]).\]
\end{thm}
\begin{proof}
As before one shows that $[\sigma \tilde{\tau}_i]^p = [\sigma \tilde{\tau}_{i+1}]$ for $i \geq 2$.
We show by induction on $i \geq 0$ that we can find representatives
\[ \gamma_{p^i} \in (H\mathbb{F}_p)_{p^i(2p-2)}\THH(\K)\]
of $\gamma_{p^i}(\sigma x) \in E^{\infty}_{*,*}$ such that $\gamma_{p^i}^p = 0$.  As in Theorem \ref{multerw} one proves that 
the element $\gamma_1 = [\sigma x]$ satisfies $\gamma_1^p  = 0$. 
Assume that $i > 0$ and that the assertion has been shown for all $0 \leq j < i$. 
It suffices to show that we can find a representative $\gamma_{p^i}$ for $\gamma_{p^i}(\sigma x)$ such that $\gamma_{p^i}^p$ is a comodule and coalgebra primitive: By Lemma \ref{EinftyBSS2} the simultaneous comodule and coalgebra primitives of $E^{\infty}$  lie in total degrees different from $p^{i+1}(2p-2)$. 
First, let $\gamma_{p^i}$ be an arbitrary representatives of $\gamma_{p^i}(\sigma x)$ that is in the kernel of the augmentation. 
Let $g$ be the map $\THH(\K) \to \THH(\K; H\mathbb{Z}_p)$. 
We can write 
\[g_*(\gamma_{p^i}) = \sum_j a_j[\sigma \tilde{\tau}_2]^j \text{~modulo~} I_p\]
for certain $a_j \in P(\bar{\xi}_1, \dots)$ with $|a_j| = p^i(2p-2)- 2p^2j$. Since $g_*(\gamma_{p^i})$ lies in the kernel of the augmentation we have $a_0 = 0$. Let $\tilde{a}_j$ be the element of $P(\tilde{\xi}_1, \dots) \subset(H\mathbb{F}_p)_*\K$ that corresponds to $a_j$ under the canonical isomorphism 
\[P(\tilde{\xi}_1, \dots) \cong P(\bar{\xi}_1, \dots).\]
Then $\gamma_{p^i}-\sum_j \tilde{a}_j [\sigma \tilde{\tau}_2]^j$ is also a representative for $\gamma_{p^i}(\sigma x)$ that is in the kernel of the augmentation and it satisfies 
\[g_*(\gamma_{p^i}- \sum_j \tilde{a}_j[\sigma \tilde{\tau}_2]^j) = 0 \text{~modulo~} I_p. \]
We denote this new representative again by $\gamma_{p^i}$. As in the proof of Theorem \ref{multerw} one shows that $\gamma_{p^i}^p$ is a comodule and coalgebra primitive. 
\end{proof}

We want to deduce the $V(1)$-homotopy of $\THH(\K)$ in case (\ref{2}). We do this by proving that in this case $V(1) \wedge_S^L \THH(\K)$ is a module in $\mathscr{D}_S$ over the $S$-ring spectrum $H\mathbb{F}_p$. Note that this implies that it is isomorphic in $\mathscr{D}_S$ to a coproduct of $S$-modules of the form $H\mathbb{F}_p \wedge_S^L S^n_S$ and that the Hurewicz morphism induces an isomorphism between $V(1)_*\THH(\K)$ and  the comodule primitives in 
$(H\mathbb{F}_p)_*\bigl(V(1) \wedge_S^L \THH(\K)\bigr)$.

\begin{rmk} \label{natExt}
Let $R \to R'$ be a morphism of commutative $S$-algebras, and let $M$ and $N$ be $R'$-modules. Note that we have  a map $\Ext_{R'}^*(M,N) \to \Ext_R^*(M,N)$. We need that it has a compatible map of   spectral sequences.   Ext spectral sequences can be constructed by applying $\Ext^*_{-}( - , N)$ to a projective topological resolution of $M$ or by applying $\Ext_{-}^*(M, -)$ to an injective topological resolution of $N$ \cite[Section 6]{LewMand}.  Since in \cite{LewMand} conditional convergence is shown for the unrolled exact couples that are constructed from injective topological resolutions,  we use these.  Let $\pi_*(N) \to I^0_* \xrightarrow{d^0} I_*^1 \xrightarrow{d^1} I^2_* \xrightarrow{d^2} \dots $ be an $R'_*$-injective resolution. 
We consider a compatible injective topological resolution, i.e. 
fiber sequences 
\[\begin{tikzcd} \Omega^{s+1}I^s  \ar{r}{k^s} & N^{s+1} \ar{r}{i^{s+1}} & N^s  \ar{r}{j^s}  &   \Omega^s I^s \end{tikzcd}\]
in $\mathscr{D}_{R'}$  for $s \geq 0$ such that $N^0 = N$ and $\pi_*j^s$ is a monomorphism, and such that we have isomorphisms $I^s_* \cong \pi_*I^s$ under which  $\pi_*j^0$ corresponds to the augmentation $\pi_*(N) \to I^0_*$  and $\pi_*(j^{s+1} \circ k^s)$ corresponds to $\Sigma^{-(s+1)} d^s$. 
Analogously, we consider an $R_*$-injective resolution $\pi_*(N) \to J_*^*$ and a compatible injective topological resolution 
in $\mathscr{D}_R$. Let $I^*_* \to J^*_*$ be an $R_*$-linear map of resolutions lifting the identity map of $\pi_*(N)$.  Using that  $\mathscr{D}_R(K,L) \cong \Hom_{R_*}(\pi_*(K), \pi_*(L))$ if $\pi_*(L)$ is injective \cite[Corollary 5.7]{LewMand}, one inductively constructs  compatible maps of fiber sequences  in $\mathscr{D}_R$. Using the natural transformation $\Ext_{R'}^*(M, -) \to \Ext^*_R(M,-)$  we get a map of unrolled exact couples and therefore a map of spectral sequences.   On $E^1$-pages it is in bidegree $(s,t)$ for $s \geq 0$  given by
\[\begin{tikzcd}
 \Hom_{R'_*}(\Sigma^{-t} \pi_*(M), I_*^s) \to \Hom_{R_*}(\Sigma^{-t}\pi_*(M), J_*^s). \end{tikzcd} \]
Now, let $P_{*,*} \to \pi_*(M)$ be an $R'_*$-projective resolution, let $Q_{*,*} \to \pi_*(M)$ be an $R_*$-projective resolution and let  $Q_{*,*} \to P_{*,*}$ be an $R_*$-linear  chain map lifting the identity map of $\pi_*(M)$. Then, by comparing with the maps on total complexes given by
\[\begin{tikzcd}
\Hom_{R'_*}(P_{*,*}, \Sigma^t I_*^*) \ar{r} & \Hom_{R_*}(Q_{*,*}, \Sigma^tJ_*^*) ,
\end{tikzcd}\]
one sees that the map on $E^2$-pages is also induced by the maps 
\[ \begin{tikzcd}
\Hom_{R'_*}(P_{s,*}, \Sigma^t\pi_*(N)) \ar{r} & \Hom_{R_*}(Q_{s,*},  \Sigma^t \pi_*(N)).  
\end{tikzcd} \]
\end{rmk}
\begin{lem} \label{bimod}
In case (\ref{2}) the $\K$-module $V(1) \wedge_S \K$ is isomorphic in $\mathscr{D}_{\K}$ to an object in the image of the map
\[ \mathscr{D}_{H\mathbb{F}_p \wedge_S \K} \to \mathscr{D}_{\K}\]
induced by the inclusion   of $\K$ into the second smash factor of $H\mathbb{F}_p \wedge_S \K$. 
\end{lem}
\begin{proof}
Since  we have $V(1)_*^S\K = E(x)$  with $|x| = 2p-3$ and \[\mathscr{D}_{\K}( V(1) \wedge_S \K, H\mathbb{F}_p) = \Hom_{\K_0}( \pi_0(V(1) \wedge_S \K), \mathbb{F}_p)\] we get a map $V(1) \wedge_S \K \to H\mathbb{F}_p$ that is the identity on $\pi_0$ and this is part of a distinguished  triangle 
\[\begin{tikzcd}
 V(1) \wedge_S \K \ar{r} & H\mathbb{F}_p   \ar{r}{g} & \Sigma^{2p-2}H\mathbb{F}_p \ar{r} & \Sigma V(1) \wedge_S \K
 \end{tikzcd} \]
 in $\mathscr{D}_{\K}$. 
 It now suffices to show that there  is a $\tilde{g}$ that is  mapped to  $g$  under 
\[ \begin{tikzcd}
\mathscr{D}_{H\mathbb{F}_p \wedge_S \K}(H\mathbb{F}_p, \Sigma^{2p-2}H\mathbb{F}_p) \ar{r} &  \mathscr{D}_{\K}(H\mathbb{F}_p, \Sigma^{2p-2}H\mathbb{F}_p).
\end{tikzcd}\]
 That is because then the image of the fiber of $\tilde{g}$ under  $\mathscr{D}_{H\mathbb{F}_p  \wedge_S \K} \to \mathscr{D}_{\K}$ is isomorphic to $V(1) \wedge_S \K$. 
We show that 
\[\begin{tikzcd}
\mathscr{D}_{H\mathbb{F}_p \wedge_S \K}(H\mathbb{F}_p, \Sigma^{2p-2}H\mathbb{F}_p ) \ar{r} &  \mathscr{D}_{\K}\bigl(H\mathbb{F}_p, \Sigma^{2p-2}H\mathbb{F}_p\bigr)\end{tikzcd}\]
is an isomorphism.  
We have a free resolution of $\mathbb{F}_p$ as an $(H\mathbb{F}_p)_*\K$-module 
 \[ \begin{tikzcd}
\dots \ar{r} &   P_{2,*}   \ar{r}{d^2} & P_{1,*}     \ar{r}{d^1} &   P_{0,*} =(H\mathbb{F}_p)_*\K \ar{r} & \mathbb{F}_p \ar{r} &  0,
\end{tikzcd}\]
where 
\[ P_{1,*} = (H\mathbb{F}_p)_*\K\{\sigma x\} \oplus \bigoplus_{i \geq 1} (H\mathbb{F}_p)_*\K\{\sigma \tilde{\xi}_i\} \oplus \bigoplus_{i \geq 2}(H\mathbb{F}_p)_*\{\sigma \tilde{\tau}_i\}\]
and 
\begin{eqnarray*}
d^1(\sigma x) & = &  x \\
d^1(\sigma \tilde{\xi}_i) & = & \tilde{\xi}_i \text{~for~} i \geq 1 \\
d^1(\sigma \tilde{\tau}_i)&  = & \tilde{\tau}_i \text{~for~} i \geq 2,
\end{eqnarray*}
 and where $P_{i, *} = 0$ if $i \geq 2$ and $* \leq 2p-3$. 
We get
\[
\Hom_{(H\mathbb{F}_p)_*\K}(P_{n,*}, \Sigma^m\Sigma^{2p-2}\mathbb{F}_p)  = \begin{cases}
  \mathbb{F}_p, & \text{~if~}  (n,m) = (1,-1); \\
  0,  & \text{~if~}  n+m = 0 \text{~and~} (n,m) \neq (1,-1); \\
  0, & \text{~if~}  (n,m) = (0,-1);  \\
  0, & \text{~if~}  n \geq 2 \text{~and~}  n+m = 1.
\end{cases}
\]
 We have a free resolution of $\mathbb{F}_p$ as a $\K_*$-module 
\[\begin{tikzcd}
\dots \ar{r} & Q_{2,*} \ar{r} & Q_{1,*} \ar{r}{d^1} &  Q_{0,*} = \K_* \ar{r} & \mathbb{F}_p \ar{r} & 0,
\end{tikzcd}\]
where 
\[Q_{1,*} = \bigoplus_{i \geq 1} \Sigma^{2i(p-1)-1} \K_* \oplus \Sigma^0\K_*,\] 
$d^1(\Sigma^01) = p \in  \pi_0(\K) = \mathbb{Z}_p$ and 
\[d^1(\Sigma^{2i(p-1)-1}1) = 1 \in \pi_{2i(p-1)-1}(\K) = \mathbb{Z}/{p^{v_p(q^{i(p-1)}-1)}},\]
and where $Q_{i,*} = 0$ if $i \geq 2$ and $* \leq 2p-4$.
We get
\[ 
\Hom_{\K_*}(Q_{n,*}, \Sigma^m\Sigma^{2p-2}\mathbb{F}_p) = \begin{cases}
                              \mathbb{F}_p,  & \text{~if~} (n,m) = (1,-1);  \\
                              0, & \text{~if~}  n+m = 0 \text{~and~} (n,m) \neq (1,-1); \\
                       0, & \text{~if~} (n,m) = (0,-1). 
                            \end{cases}
\]
Furthermore, we have a $\K_*$-linear map of chain complexes 
\[\begin{tikzcd}
\dots \ar{r} & Q_{2,*} \ar{r} \ar{d}{f_{2,*}} & Q_{1,*} \ar{r} \ar{d}{f_{1,*}}  & K_* \ar{r} \ar{d}{f_{0,*}} &  \mathbb{F}_p \ar{r} \ar[equal]{d} & 0 \\
\dots \ar{r} & P_{2,*} \ar{r} & P_{1,*} \ar{r} & (H\mathbb{F}_p)_*\K  \ar{r} & \mathbb{F}_p \ar{r} & 0 
\end{tikzcd}\]
with $f_0(1) = 1$ and $f_1(\Sigma^{2p-3}1) = \sigma x$. To prove this, it suffices to show that the Hurewicz map  $h_{\K}: \K_* \to (H\mathbb{F}_p)_*\K$ maps
the element
\[ 1 \in \pi_{2p-3}(\K) = \mathbb{Z}/{p^{v_p(q^{(p-1)}-1)}} \]  to $x$.  
This follows from the commutativity of the diagram 
\[\begin{tikzcd}
\Sigma \pi_{*}\bigl(\K(\bar{\mathbb{F}}_l)_p\bigr) \ar{r} \ar{d}{\Sigma h_{\K(\bar{\mathbb{F}}_l)_p}} & \pi_{*}(\K)  \ar{d}{h_{\K}}  \\
\Sigma(H\mathbb{F}_p)_{*}\bigl(\K(\bar{\mathbb{F}}_l)_p\bigr) \ar{r}{\Delta} & (H\mathbb{F}_p)_{*}\K, 
\end{tikzcd}\]
and from $(\Sigma h_{\K(\bar{\mathbb{F}}_l)_p})(\Sigma u^{p-2}) = \Sigma u^{p-2}$ and $\Delta(\Sigma u^{p-2}) = x$. 

Using Remark  \ref{natExt} we see that  the map \[ \begin{tikzcd}
\Ext^*_{ H\mathbb{F}_p \wedge_S \K}(H\mathbb{F}_p, \Sigma^{2p-2}H\mathbb{F}_p) \ar{r} & \Ext_{\K}^*(H\mathbb{F}_p, \Sigma^{2p-2}H\mathbb{F}_p) 
\end{tikzcd}\] 
has a compatible map of spectral sequence that is an isomorphism on $E^{\infty}$-pages in total degree zero. Since by \cite[Theorem 6.7]{LewMand} and \cite[Theorem 7.1]{Boar} the spectral sequences converge strongly, the claim follows. 
\end{proof}

\begin{lem} \label{dimcomod}
In  case (\ref{2}) the $S$-module   $V(1) \wedge_S^L \THH(\K)$ is isomorphic in $\mathscr{D}_S$ to an $H\mathbb{F}_p$-module.  The two $\mathbb{F}_p$-vector spaces $V(1)_*\THH(\K)$ and 
\[ E(x) \otimes E([\sigma \tilde{\xi}_1], [\sigma \tilde{\xi}_2]) \otimes P([\sigma \tilde{\tau}_2]) \otimes \Gamma([\sigma x])\] 
have the same dimension in every degree. 
\end{lem}
\begin{proof}
We have that $V(1) \wedge_S \K$ is a cell $\K$-module \cite[Proposition III.4.1]{EKMM}.  We therefore have an isomorphism in $\mathscr{D}_{\K}$: 
\[ V(1) \wedge_S \THH(\K)  \cong (V(1) \wedge_{S} \K) \wedge_{\K}^L \THH(\K).\]
By Lemma \ref{bimod} the latter is isomorphic in $\mathscr{D}_{\K}$ to $M \wedge_{\K} \Gamma^{\K}\THH(\K)$, 
where $M$ is an $(H\mathbb{F}_p,\K)$-bimodule and $\Gamma^{\K}\THH(\K)$ is a cell  approximation of the $\K$-module $\THH(\K)$.
We get that $V(1) \wedge_S^L \THH(\K)$ is isomorphic in $\mathscr{D}_{S}$ to an $H\mathbb{F}_p$-module. 
As a consequence, we have 
\[ (H\mathbb{F}_p)_*\bigl(V(1) \wedge_S^L \THH(\K)\bigr)  \cong \bigoplus_{i \in I} \Sigma^{n_i}(H\mathbb{F}_p)_*H\mathbb{F}_p, \]
where the $n_i$ are natural numbers such that for all $n$ the cardinality of 
\[ \{i \in I: n_i = n\}\] 
is equal to the dimension of $V(1)_n\THH(\K)$. On the other hand  $(H\mathbb{F}_p)_*\bigl(V(1) \wedge_S^L\THH(\K)\bigr)$
  is isomorphic to 
 \[ E(\epsilon_0, \epsilon_1) \otimes P(\tilde{\xi}_1, \dots) \otimes E(\tilde{\tau}_2, \dots) 
   \otimes  E(x) \otimes E([\sigma \tilde{\xi}_1], [\sigma \tilde{\xi}_2]) \otimes P([\sigma \tilde{\tau}_2]) \otimes \Gamma([\sigma x]).\]
 This proves the lemma. 
\end{proof}
\begin{thm} \label{casetwores}
In case (\ref{2})  we have an isomorphism of $\mathbb{F}_p$-algebras
\[ V(1)_*\THH(\K) \cong E(x) \otimes E(\lambda_1, \lambda_2) \otimes P(\mu_2) \otimes \Gamma(\gamma_1')\]
with $|x| = 2p-3$, $|\lambda_i| =2p^i-1$, $|\mu_2| = 2p^2$ and $|\gamma_1'| = 2p-2$.  
\end{thm}

\begin{proof}
We compute the $A_*$-comodule primitives in
$(H\mathbb{F}_p)_*\bigl(V(1) \wedge_S^L \THH(\K)\bigr)$. Since  
we have that $(H\mathbb{F}_p)_*V(1) = E(\epsilon_0, \epsilon_1)$
injects into the dual Steenrod algebra via a map of comodule algebras,  we can assume that the $A_*$-comodule action of $\epsilon_0$ is given by
\[ \nu(\epsilon_0) = 1 \otimes \epsilon_0 + \bar{\tau}_0 \otimes 1.\] 
 
We define classes in $(H\mathbb{F}_p)_*\bigl(V(1) \wedge_S^L \THH(\K)\bigr)$  by
\begin{eqnarray*}
\hat{\xi}_1 & \coloneqq & \tilde{\xi}_1 - a \epsilon_0 x, \\
\lambda_1 & \coloneqq & [\sigma \tilde{\xi}_1] - a\epsilon_0 [\sigma x], \\
\lambda_2 & \coloneqq &  [\sigma \tilde{\xi}_2], \\
\mu_2 & \coloneqq & [\sigma \tilde{\tau}_2] - \epsilon_0 [\sigma \tilde{\xi}_2],
\end{eqnarray*}
where $a$ is the element in $\mathbb{F}_p$ that we defined in Proposition \ref{homolK}. Then, the $A_*$-coaction of $\hat{\xi}_1$ is given by 
\[ \nu(\hat{\xi}_1) = \bar{\xi}_1 \otimes 1 + 1 \otimes \hat{\xi}_1\]
and   the classes $\lambda_1$, $\lambda_2$ and $\mu_2$ are comodule primitives .   
Let $\gamma_{p^i} \in (H\mathbb{F}_p)_{(2p-2)p^i}\THH(\K)$ 
be the classes defined in Theorem \ref{modphomolTHHK}.
We set
\[A'_* \coloneqq E(\epsilon_0, \epsilon_1) \otimes P(\hat{\xi}_1, \tilde{\xi}_2, \dots) \otimes E(\tilde{\tau}_2, \dots). \]
The map of $\mathbb{F}_p$-algebras
\[A'_* \otimes E(x)  \otimes E(\lambda_1, \lambda_2) \otimes P(\mu_2) \otimes P_p(\gamma_{p^i} | i \geq 0) \to (H\mathbb{F}_p)_*\bigl(V(1) \wedge_S^L \THH(\K)\bigr)\]
is an isomorphism, because it is surjective and both sides have the same dimension over $\mathbb{F}_p$ in every degree.  
We treat it as the identity.  Note that $A'_*$ is a subcomodule algebra of $(H\mathbb{F}_p)_*\bigr(V(1) \wedge_S^L \THH(\K)\bigl)$, because 
 \[ (\tilde{\xi}_1 - a \epsilon_0x)^{p^n} = \tilde{\xi}_1^{p^n}\]
 for $n \geq 1$.  It is isomorphic to 
$(H\mathbb{F}_p)_*V(1) \otimes (H\mathbb{F}_p)_*\ell \cong A_*$.  
We show by induction on $i \geq 0$ that we can find classes 
\[\gamma'_{p^i} \in (H\mathbb{F}_p)_{(2p-2)p^i}\bigl(V(1) \wedge_S^L \THH(\K)\bigr)\]
with the following properties:
\begin{itemize}
\item The class $\gamma'_{p^i}$ is a comodule primitive.
\item We have $(\gamma'_{p^i})^p = 0$. 
\item For $D_* = P_p(\gamma'_1, \dots, \gamma'_{p^i}, \gamma_{p^{i+1}}, \gamma_{p^{i+2}} \dots)$ the map
\[A_*' \otimes E(x)  \otimes E(\lambda_1, \lambda_2) \otimes P(\mu_2) \otimes D_* \to (H\mathbb{F}_p)_*\bigl(V(1) \wedge_S^L \THH(\K)\bigr)\]
is an isomorphism. 
\end{itemize}
We set $\gamma'_1 = \gamma_1 = [\sigma x]$. Suppose  that $i > 0$ and that we have already defined $\gamma'_{p^j}$ for $0 \leq j \leq i-1$. The $\mathbb{F}_p$-vector space
\[W \coloneqq  \bigl(E(x) \otimes E(\lambda_1, \lambda_2) \otimes P(\mu_2) \otimes P_p(\gamma'_1, \dots, \gamma'_{p^{i-1}})\bigr)_{(2p-2)p^i} \]
is included in the subspace $V$ of  primitives in $(H\mathbb{F}_p)_{(2p-2)p^i}\bigl(V(1) \wedge_S^L \THH(\K)\bigr)$.
By Lemma \ref{dimcomod} we have 
\[\dim_{\mathbb{F}_p} V = \dim_{\mathbb{F}_p}W + 1.\]
Therefore, there is a class $b \in V$ with $b \notin W$. The class $b$ cannot be an element of
\[ U \coloneqq \bigl(A_*' \otimes E(x) \otimes E(\lambda_1, \lambda_2) \otimes P(\mu_2)\otimes P_p(\gamma'_1, \dots, \gamma'_{p^{i-1}})\bigr)_{p^i(2p-2)},\]
because the comodule primitives in this vector space are the elements of $W$. 
Therefore, we have 
\[b \doteq \gamma_{p^i} + b'\] 
for an $b' \in U$. 
We have 
\[ (H\mathbb{F}_p)_*\bigl(V(1) \wedge_S^L \THH(\K)\bigr)/{I_p} \cong P(\hat{\xi}_1, \tilde{\xi}_2, \dots) \otimes P(\mu_2)\] 
and 
\[\Bigl(A_* \otimes (H\mathbb{F}_p)_*\bigl(V(1) \wedge_S^L \THH(\K)\bigr)\Bigr)/{I_p} \cong P(\bar{\xi}_1, \dots) \otimes P(\hat{\xi}_1, \tilde{\xi}_1,  \dots) \otimes P(\mu_2).\] 
The $\mathbb{F}_p$-algebra map $\bar{\nu}$, induced on these quotients by the coaction $\nu$, is given by $\bar{\nu}(\mu_2) = 1  \otimes \mu_2$ and
\[\bar{\nu}(\tilde{\xi}_n) = \bar{\xi}_{n-1} \otimes \hat{\xi}_1^{p^{n-1}} +\sum_{i+j = n, \, j \neq 1} \bar{\xi}_i \otimes \tilde{\xi}_j^{p^i}.\]
Since $b$ is a comodule primitive, we get
 \[ b = \lambda \cdot \mu_2^{p^{i-2}(p-1)} \text{~modulo~} I_p\]
for a $\lambda \in \mathbb{F}_p$ if $i \geq 2$, 
and \[ b = 0 \text{~modulo~} I_p\] if $i = 1$. 
We set $\gamma'_{p^i} \coloneqq b - \lambda \cdot \mu_2^{p^{i-2}(p-1)}$ if  $i \geq 2$ and $\gamma'_{p^i} \coloneqq b$ if $i = 1$. 
Then $\gamma'_{p^i}$ has the desired properties. 
We get  
\[ (H\mathbb{F}_p)_*\bigl(V(1) \wedge_{S}^L \THH(\K)\bigr) = A_*' \otimes E(x) \otimes E(\lambda_1, \lambda_2) \otimes P(\mu_2) \otimes P_p(\gamma_{p^i}'| i \geq 0).\]
This finishes the proof.
\end{proof}

\begin{rmk}
The methods we used to compute $V(1)_*\THH(\K)$ in case (\ref{2}) do not apply in the cases (\ref{1}), (\ref{3}) and (\ref{4}):

In case (\ref{1})  the object $V(1) \wedge_S \K \in \mathscr{D}_S$ is not a module over the $S$-ring spectrum $H\mathbb{F}_p$: 
Suppose the contrary. Then, 
we would have 
\[ (H\mathbb{F}_p)_*\bigl(V(1) \wedge_S \K\bigr)  \cong A_* \oplus \Sigma^{2p-3} A_*.\] 
   This is a contradiction to Proposition \ref{homimj}.
   
   In case (\ref{3}) one can compute $(H\mathbb{F}_p)_*\THH(\K; H\mathbb{Z}_p)$ using the above methods. But since we have a tensor factor $P_k(y)$ in $(H\mathbb{F}_p)_*\K$,  the Hochschild homology  of $(H\mathbb{F}_p)_*\K$ is not flat over $(H\mathbb{F}_p)_*\K$. The B\"okstedt spectral sequence converging to $(H\mathbb{F}_p)_*\THH(\K)$  therefore has no coalgebra structure. 
 
    In case (\ref{4}) the mod $p$ homology of $\K$  has a more complicated form  and one needs different methods to compute its Hochschild homology. 
\end{rmk}

\section{Computations with the Brun spectral sequence} \label{BrunKcomsec}

In \cite{Brunss} we have constructed a generalization of the  spectral sequence of Brun  in \cite[Theorem 6.2.10]{Brun}.  We  will refer to  this generalization as  the Brun spectral sequence. 
In this section we consider the Brun spectral sequence for $\K = \K(\mathbb{F}_q)_p$. 
We pursue the same strategy that we used in \cite{Brunss}  to compute $V(1)_*\THH(\ku_p)$, where $\ku_p$  is $p$-completed connective complex $K$-theory.

By \cite[Theorem 4.11]{Brunss} and \cite[Lemma 4.13]{Brunss} we have a Brun spectral sequence of  the form
\begin{equation} \label{sset3} 
E^2_{n,m} = V(1)_m\K \otimes \THH_n(\K; H\mathbb{F}_p) \Longrightarrow V(1)_{n+m}\THH(\K) 
\end{equation}
which is multiplicative.  Here, recall   that since $\K$ is a connective cofibrant commutative $S$-algebra, we have a map of commutative $S$-algebras $\K \to H\pi_0(K) = H\mathbb{Z}_p$ realizing the identity on $\pi_0$.  We can compose this with the map induced by the ring homomorphism $\mathbb{Z}_p \to \mathbb{F}_p$ to get a map $\K \to H\mathbb{F}_p$.   
We factor the map $\K \to H\mathbb{Z}_p$ in $\mathscr{C}\mathscr{A}_S$ as a cofibration followed by an acyclic fibration: 
 \[\begin{tikzcd} 
  \K  \ar[tail]{r} & \hat{H}\mathbb{Z}_p \ar[two heads]{r}{\sim} & H\mathbb{Z}_p. 
  \end{tikzcd}\] 
Analogously to  the case of $\ku_p$ in \cite{Brunss} we have  an isomorphism of $S$-ring spectra 
\begin{equation*} \label{V0vp}
 \THH(\K; H\mathbb{F}_p) \cong V(0) \wedge_S^L \THH(\K; \hat{H}\mathbb{Z}_p).
 \end{equation*} 
Again by \cite[Theorem 4.11, Lemma 4.13]{Brunss} we have  the Brun spectral sequence
  \begin{equation} \label{sset2} E^2_{*,*} = V(0)_*(\hat{H} \mathbb{Z}_p \wedge_{\K} \hat{H}\mathbb{Z}_p) \otimes \THH_*(\hat{H} \mathbb{Z}_p; H\mathbb{F}_p)  \Longrightarrow V(0)_*\THH(\K; \hat{H}\mathbb{Z}_p).
  \end{equation}
  In Subsection \ref{et1} we will compute the ring $V(0)_*(\hat{H}\mathbb{Z}_p \wedge_{\K} \hat{H}\mathbb{Z}_p)$, in Subsection \ref{et2} we will consider the spectral sequence (\ref{sset2}), and finally in  Subsection \ref{et3} we will consider the spectral sequence (\ref{sset3}). 
\subsection{The mod $p$ homotopy of $H\mathbb{Z}_p \wedge_{\K(\mathbb{F}_q)_p}^L H\mathbb{Z}_p$} \label{et1}

In this subsection we will compute  the graded $\mathbb{F}_p$-algebra $V(0)_*(\hat{H}\mathbb{Z}_p \wedge_{\K} \hat{H}\mathbb{Z}_p)$. 

\begin{rmk}
A tempting strategy to compute $V(0)_*(\hat{H}\mathbb{Z}_p \wedge_{\K} \hat{H}\mathbb{Z}_p)$  would be to use an 
Eilenberg-Moore type spectral sequence \cite[Section IV.6]{EKMM}
\[ E^2_{*,*} = \Tor_{*,*}^{V(0)_*\K}(\mathbb{F}_p, \mathbb{F}_p)  \Longrightarrow V(0)_*(\hat{H}\mathbb{Z}_p \wedge_{\K} \hat{H}\mathbb{Z}_p). \]
Such a spectral sequence would have to collapse at the $E^2$-page and would yield 
\begin{equation} \label{gu}
 V(0)_*(\hat{H}\mathbb{Z}_p \wedge_{\K} \hat{H}\mathbb{Z}_p) = \Gamma(\sigma x) \otimes E(\sigma y)
 \end{equation} 
with   $|\sigma x| = 2r$ and $|\sigma y| = 2r+1$. 
But this requires $V(0) \wedge_S \K$ to be an $S$-algebra. 
As in  \cite[Example 3.3]{An} one can  use  that there is a $p$-fold Massey product $\left< \bar{\tau}_0, \dots, \bar{\tau}_0 \right> = - \bar{\xi}_1$ in $(H\mathbb{F}_p)_*H\mathbb{F}_p$ defined with no indeterminacy to show that $V(0) \wedge_S \K$ is no $S$-algebra in case (\ref{1}). 
  We therefore cannot use the above strategy, but we still obtain (\ref{gu}).
\end{rmk}
\begin{lem} 
We have a multiplicative spectral sequence of the form 
\begin{equation} \label{kunnSS}
 E^2_{*,*} = \Tor_{*,*}^{\pi_*(\K(\bar{\mathbb{F}}_l)_p)}\bigl(B_*, \mathbb{Z}_p\bigr) \Longrightarrow V(0)_*(\hat{H}\mathbb{Z}_p \wedge_{\K} \hat{H}\mathbb{Z}_p),
 \end{equation}
 where 
 \[B_* \cong   \pi_*\bigl((V(0) \wedge_{S} H\mathbb{Z}_p) \wedge_{\K}^L \K(\bar{\mathbb{F}}_l)_p\bigr) \cong \pi_*\bigl(H\mathbb{F}_p \wedge_{\K}^L \K(\bar{\mathbb{F}}_l)_p\bigr). \] 
\end{lem}
Here, we use that  $V(0) \wedge_S H\mathbb{Z}_p$ is an $H\mathbb{Z}_p$-ring spectrum that is isomorphic to $H\mathbb{F}_p$ and that one therefore gets isomorphic $\K$-ring spectra by applying the lax symmetric monoidal functor 
$\mathscr{D}_{H\mathbb{Z}_p} \to \mathscr{D}_{\K}$. 
\begin{proof}
The canonical map in $\mathscr{D}_{\K}$ 
\begin{equation}
\begin{tikzcd} \label{map1}
\hat{H} \mathbb{Z}_p \wedge_{\K}^L H \mathbb{Z}_p  \cong  \hat{H} \mathbb{Z}_p \wedge_{\K}^L \hat{H} \mathbb{Z}_p  \ar{r} & \hat{H} \mathbb{Z}_p \wedge_{\K} \hat{H}\mathbb{Z}_p  
\end{tikzcd}
\end{equation}
is an isomorphism of $\K$-ring spectra.  The diagram
\[\begin{tikzcd}
\K \ar{d} \ar{r} & H\mathbb{Z}_p \ar{d}{\id} \\ 
\K(\bar{\mathbb{\F}}_l)_p \ar{r} & H\mathbb{Z}_p 
\end{tikzcd}
\] 
is homotopy commutative in $\mathscr{C}\mathscr{A}_S$  and  therefore in the category of $S$-modules \cite[Proposition VII.2.11]{EKMM}.  We get   that the left-most $\K$-ring spectrum in (\ref{map1}) is isomorphic to $\hat{H} \mathbb{Z}_p \wedge_{\K}^L H \mathbb{Z}_p $, but where the $\K$-module structure of $H\mathbb{Z}_p$ is now given by 
\[\begin{tikzcd}
\K \ar{r} &  \K(\bar{\mathbb{F}}_l)_p  \ar{r} & H\mathbb{Z}_p. 
\end{tikzcd}\]
This is isomorphic as an $\K$-ring spectrum to $\hat{H}\mathbb{Z}_p \wedge_{\K} H\mathbb{Z}_p$. 
 We have isomorphisms of commutative $S$-algebras 
\[
\hat{H} \mathbb{Z}_p \wedge_{\K} H\mathbb{Z}_p \cong \hat{H} \mathbb{Z}_p \wedge_{\K} \bigl(\K(\bar{\mathbb{F}}_l)_p \wedge_{\K(\bar{\mathbb{F}}_l)_p} H\mathbb{Z}_p \bigr) \cong 
\bigl(\hat{H} \mathbb{Z}_p \wedge_{\K} \K(\bar{\mathbb{F}}_l)_p \bigr) \wedge_{\K(\bar{\mathbb{F}}_l)_p} H\mathbb{Z}_p .
\]
Since cofibrations are stable under cobase change, the map 
\[\begin{tikzcd}
\K(\bar{\mathbb{F}}_l)_p \ar{r} & \hat{H} \mathbb{Z}_p \wedge_{\K} \K(\bar{\mathbb{F}}_l)_p 
\end{tikzcd}\]
 is a cofibration of commutative $S$-algebras. Therefore,   the canonical map  
 \[\begin{tikzcd}
 \bigl(\hat{H} \mathbb{Z}_p \wedge_{\K} \K(\bar{\mathbb{F}}_l)_p \bigr)  \wedge_{\K(\bar{\mathbb{F}}_l)_p}^L H\mathbb{Z}_p \ar{r} &  \bigl(\hat{H} \mathbb{Z}_p \wedge_{\K} \K(\bar{\mathbb{F}}_l)_p \bigr) \wedge_{\K(\bar{\mathbb{F}}_l)_p} H\mathbb{Z}_p
 \end{tikzcd}\] in $\mathscr{D}_{\K(\bar{\mathbb{F}}_l)_p}$ is an isomorphism of $\K(\bar{\mathbb{F}}_l)_p$-ring spectra.  
 Thus, we have an isomorphism of $S$-ring spectra 
 \[ V(0) \wedge_S^L (\hat{H} \mathbb{Z}_p \wedge_{\K} \hat{H} \mathbb{Z}_p) \cong V(0) \wedge_S^L\Bigl(\bigl(\hat{H} \mathbb{Z}_p \wedge_{\K} \K(\bar{\mathbb{F}}_l)_p \bigr) \wedge_{\K(\bar{\mathbb{F}}_l)_p}^L H\mathbb{Z}_p\Bigr).\]
 By \cite[Remark 4.10]{Brunss} the latter  $S$-ring spectrum is isomorphic to 
 \[ \Bigl(V(0) \wedge_S \bigl(\hat{H}\mathbb{Z}_p \wedge_{\K} \K(\bar{\mathbb{F}}_l)_p\bigr)\Bigr) \wedge_{\K(\bar{\mathbb{F}}_l)_p}^L H\mathbb{Z}_p.\]
 Here, note that 
 $V(0) \wedge_S \bigl(\hat{H}\mathbb{Z}_p \wedge_{\K} \K(\bar{\mathbb{F}}_l)_p\bigr)$
  is an 
 $\hat{H}\mathbb{Z}_p \wedge_{\K} \K(\bar{\mathbb{F}}_l)_p$-ring spectrum and that we therefore get a $\K(\bar{\mathbb{F}}_l)_p$-ring spectrum after applying the lax symmetric monoidal functor $\mathscr{D}_{\hat{H}\mathbb{Z}_p \wedge_{\K} \K(\bar{\mathbb{F}}_l)_p} \to \mathscr{D}_{\K(\bar{\mathbb{F}}_l)_p}$.
 By \cite[Lemma 1.3]{BakLaz} we get a multiplicative spectral sequence
 \[ E^2_{*,*} = \Tor^{\pi_*(\K(\bar{\mathbb{F}}_l)_p)}_{*,*}(B_*, \mathbb{Z}_p) \Longrightarrow V(0)_*(\hat{H} \mathbb{Z}_p \wedge_{\K}
 \hat{H} \mathbb{Z}_p)\]
 with $B_* = \pi_*\Bigl(V(0) \wedge_S \bigl(\hat{H}\mathbb{Z}_p \wedge_{\K} \K(\bar{\mathbb{F}}_l)_p\bigr)\Bigr)$. 
 
  We have an isomorphism of $S$-ring spectra
\[V(0) \wedge_S \bigr(\hat{H}\mathbb{Z}_p \wedge_{\K} \K(\bar{\mathbb{F}}_l)_p\bigl) \, \cong V(0) \wedge_S^L \bigl(\hat{H} \mathbb{Z}_p \wedge_{\K}^L\K(\bar{\mathbb{F}}_l)_p\bigr) \cong V(0) \wedge_S^L\bigr(H\mathbb{Z}_p \wedge_{\K}^L \K(\bar{\mathbb{F}}_l)_p\bigr).\] 
Again by \cite[Remark 4.10]{Brunss}
 this  is isomorphic  to 
$(V(0) \wedge_S H\mathbb{Z}_p) \wedge_{\K}^L\K(\bar{\mathbb{F}}_l)_p$ as an $S$-ring spectrum. 
The last three identifications  are induced by  maps under $\K(\bar{\mathbb{F}}_l)_p$ in $\mathscr{D}_S$. 
\end{proof}
In the following lemmas we compute the $\pi_*(\K(\bar{\mathbb{F}}_l)_p)$-algebra $\pi_*(H\mathbb{F}_p \wedge_{\K}^L\K(\bar{\mathbb{F}}_l)_p)$. 
\begin{lem} \label{addeins}
We have
\[ \pi_n(H\mathbb{F}_p \wedge_{\K}^L \K(\bar{\mathbb{F}}_l)_p) = \begin{cases}
                                            \mathbb{F}_p, &  n = 2i, \, i \geq 0; \\
                                            0, & \text{otherwise}. 
                                            \end{cases}\]
Moreover,  the map $\pi_*(\K(\bar{\mathbb{F}}_l)_p) \to \pi_*\bigl(H\mathbb{F}_p \wedge_{\K}^L \K(\bar{\mathbb{F}}_l)_p\bigr)$ 
   factors as 
     \[ \begin{tikzcd}[column sep=20pt]
  \pi_*(\K(\bar{\F}_\ell)_p) \cong \mathbb{Z}_p[u] \ar{rr}\ar{dr} && \pi_*\bigl(H\F_p \wedge_{\K}^L \K(\bar{\F}_\ell)_p\bigr),  \\
& P_r(u) \ar{ru}   & 
   \end{tikzcd}\]
and $P_r(u) \to \pi_*\bigl(H\mathbb{F}_p \wedge_{\K}^L \K(\bar{\mathbb{F}}_l)_p\bigr)$ is an isomorphism in degrees $ < 2r$.
\end{lem}
\begin{proof}
Define $X \coloneqq V(0) \wedge_S H \mathbb{Z}_p$ and $Y \coloneqq V(0) \wedge_S \K$. The $\K$-module map $g: Y \to X$   defines a map of $\K$-ring spectra. Applying $Y \wedge_{\K}^L (-) \to   X \wedge_{\K}^L(-)$ to the distinguished triangle (\ref{fibseq}) 
we  get a map between long exact sequences 
\begin{equation} \label{LES}
\begin{tikzcd}
\ar{r} &   \pi_n(X \wedge_{\K}^L \K) \ar{r} & \pi_n(X \wedge_{\K}^L \K(\bar{\mathbb{F}}_l)_p) \ar{r} & \pi_{n-2}(X \wedge_{\K}^L \K(\bar{\mathbb{F}}_l)_p) \ar{r} &  \cdots  \phantom{\, . }  \\
\ar{r} & \pi_n(Y \wedge_{\K}^L \K) \ar{r} \ar{u} & \pi_n(Y \wedge_{\K}^L \K(\bar{\mathbb{F}}_l)_p) \ar{r} \ar{u}{(g \wedge \id)_n} & \pi_{n-2}(Y \wedge_{\K}^L \K(\bar{\mathbb{F}}_l)_p)  \ar{r} \ar{u}{(g \wedge \id)_{n-2}} &   \cdots \, .  
\end{tikzcd}
\end{equation} 
Using the $\Tor$ spectral sequence one gets that $X \wedge_{\K}^L \K(\bar{\mathbb{F}}_l)_p$  is connective. Using that $\pi_*(X \wedge_{\K}^L K)$  is $\mathbb{F}_p$ concentrated in degree $0$  the above long exact sequence yields   
\[ \pi_n(X \wedge_{\K}^L \K(\bar{\mathbb{F}}_l)_p) = \begin{cases} 
 \mathbb{F}_p,  & n = 0; \\
 0, & n=1; 
 \end{cases}\]
 and $\pi_{n}(X \wedge_{\K}^L \K(\bar{\mathbb{F}}_l)_p) \cong \pi_{n-2}(X \wedge_{\K}^L \K(\bar{\mathbb{F}}_l)_p)$ 
for $n \geq 2$. 
This shows the first part of the statement.
 We now show the second part of the statement: 
The map $\K(\bar{\mathbb{F}}_l)_p \to X \wedge_{\K}^L \K(\bar{\mathbb{F}}_l)_p$ factors in $\mathscr{D}_{\K}$ as
\[\begin{tikzcd}
\K(\bar{\mathbb{F}}_l)_p \ar{r} \ar{rd} & X \wedge_{\K}^L \K(\bar{\mathbb{F}}_l)_p \\
 &  Y \wedge_{\K}^L \K(\bar{\mathbb{F}}_l)_p \ar{u} 
\end{tikzcd}
\]  
The map  $\pi_*\bigl(\K(\bar{\mathbb{F}}_l)_p\bigr) \to \pi_*\bigl(Y \wedge_{\K}^L \K(\bar{\mathbb{F}}_l)_p\bigr)$  identifies with the canonical map  
$\mathbb{Z}_p[u] \to P(u)$. Hence, 
 it suffices to show that 
 \[ (g \wedge \id)_*: \pi_*\bigl(Y \wedge_{\K}^L \K(\bar{\mathbb{F}}_l)_p\bigr) \to \pi_*\bigl(X \wedge_{\K}^L \K(\bar{\mathbb{F}}_l)_p\bigl)\]
 maps $u^r$ to  zero and is an isomorphism in degrees $< 2r$. It is  clear that it is an isomorphism in odd degrees, because in these degrees both sides are zero. 
 It is also clear that it is an isomorphism in degree zero: In degree zero both sides are equal to $\mathbb{F}_p$,  and the map
 is not zero because it is a map of rings and therefore maps the unit to the unit.  We now suppose that $0 < n < 2r$ is even and that we have already shown that $(g \wedge \id)_m$ is an isomorphism for all even $0 \leq m < n$. 
 By Lemma \ref{modphpty}  $\pi_n(Y \wedge_{\K}^L \K)$ is zero. Therefore,  diagram (\ref{LES})
 is given by
 \[\begin{tikzcd}
\cdots \ar{r} & 0 \ar{r} & \mathbb{F}_p \ar{r} & \mathbb{F}_p  \ar{r}&  \cdots \phantom{\, .} \\
\cdots \ar{r} & 0 \ar{r} \ar{u} & \mathbb{F}_p \ar{r} \ar{u} & \mathbb{F}_p  \ar{u} \ar{r} & \cdots \, .
 \end{tikzcd}\]  
 The right vertical map is an isomorphism by the induction hypothesis. Thus, the vertical map in the middle  is also an  isomorphism, and this map is $(g \wedge \id)_n$. 
 We conclude that $(g \wedge \id)_n$ is an isomorphism for all $n < 2r$. We now consider diagram (\ref{LES}) for $n = 2r$. 
 Since by the proof of Lemma \ref{modphpty} the map $\pi_{2r}(Y \wedge_{\K}^L \K) \to \pi_{2r}\bigl(Y \wedge_{\K}^L \K(\bar{\mathbb{F}}_l)_p\bigr) $  is an isomorphism, it is given by 
  \[\begin{tikzcd}
\cdots \ar{r} & 0 \ar{r} & \mathbb{F}_p \ar{r}{\cong} & \mathbb{F}_p  \ar{r}&  \cdots \phantom{\, .} \\
\cdots \ar{r} & \mathbb{F}_p \ar{r}{\cong} \ar{u} & \mathbb{F}_p \ar{r}{0} \ar{u} & \mathbb{F}_p  \ar{u}[swap]{\cong} \ar{r} & \cdots \, .
 \end{tikzcd}\]  
 It follows that $(g \wedge \id)_{2r}$ is zero.           
\end{proof}
Let $F$ be the map 
\[ \begin{tikzcd}[column sep=42pt] 
\pi_*\bigl(H\F_p \wedge_{\K}^L \K(\bar{\F}_\ell)_p\bigr) \ar{r}{(\id \wedge f)_*} & \Sigma^2 \pi_*\bigl(H\F_p \wedge_{\K}^L \K(\bar{\F}_\ell)_p\bigr),
\end{tikzcd}
\]
where $f: \K(\bar{\F}_\ell)_p \to S^2_{\K} \wedge_{\K}^L 
\K(\bar{\F}_\ell)_p$ is the morphism in the distinguished triangle (\ref{fibseq}). 
Let $G$ be the map 
\[ \begin{tikzcd}[column sep = 55pt]
\pi_*\bigl(H\F_p \wedge_{\K}^L \K(\bar{\F}_\ell)_p \bigr ) \ar{r}{\id \wedge ( \phi^q - \id)_*} & \pi_*\bigr(H\F_p \wedge_K^L \K(\bar{\F}_\ell)_p\bigl).
\end{tikzcd} \] 
The following diagram commutes:
\begin{equation} \label{commu}
\begin{tikzcd}
\pi_*\bigl(H\mathbb{F}_p \wedge_{\K}^L \K(\bar{\mathbb{F}}_l)_p\bigr) \ar{rd}[swap]{F} \ar{rr}{G}  & & \pi_*\bigl(H\mathbb{F}_p \wedge_{\K}^L \K(\bar{\mathbb{F}}_l)_p\bigr).  \\
& \Sigma ^2 \pi_*\bigr(H\mathbb{F}_p \wedge_{\K}^L \K(\bar{\mathbb{F}}_l)_p\bigl) \ar{ru}[swap]{(-)\cdot u} &
\end{tikzcd}
\end{equation}
To determine the multiplicative structure of $\pi_*\bigl(H\mathbb{F}_p \wedge_{\K}^L \K(\bar{\mathbb{F}}_l)_p\bigr)$
we need the following lemma:
\begin{lem}
Let $a, b \in \pi_*\bigr(H\mathbb{F}_p \wedge_{\K}^L \K(\bar{\mathbb{F}}_l)_p\bigl)$. The following equations hold:
\begin{align}
F(ab) & = F(a) b + a F(b) + F(a)G(b)   \label{ef} \\
F(a^n) & = F(a)\biggl(\sum_{i = 0}^{n-1} \binom{n}{n-1-i}a^{n-1-i}G(a)^i \biggr). \label{sf}
\end{align}
\end{lem}
\begin{proof}
Formula (\ref{ef}) follows from Lemma \ref{cofibmult}. 
Formula (\ref{sf}) follows from formula (\ref{ef}) by induction. 
\end{proof}

\begin{lem} \label{multf}
There is an isomorphism of $\pi_*\bigl(\K(\bar{\mathbb{F}}_l)_p\bigr)$-algebras
\[\pi_*\bigl(H\mathbb{F}_p \wedge_{\K}^L \K(\bar{\mathbb{F}}_l)_p\bigr) \cong P_r(u) \otimes \Gamma(\sigma x),\]
where $|\sigma x| = 2r$. 
\end{lem}

\begin{proof}
By Lemma \ref{addeins} the unit $\pi_*\bigl(\K(\bar{\mathbb{F}}_l)_p\bigr) \to \pi_*\bigl(H\mathbb{F}_p \wedge_{\K}^L \K(\bar{\mathbb{F}}_l)_p\bigr)$ induces a map $P_r(u) \to \pi_*\bigl(H\mathbb{F}_p \wedge_{\K}^L \K(\bar{\mathbb{F}}_l)_p\bigr)$ that is an isomorphism in degrees $< 2r$. 
For $i \geq 0$ choose a non-zero class 
\[\gamma_{p^i}(\sigma x) \in \pi_{2p^ir}\bigl(H\mathbb{F}_p \wedge_{\K}^L \K(\bar{\mathbb{F}}_l)_p\bigr) \cong \mathbb{F}_p.\] 
By formula (\ref{sf}) we have 
\begin{equation*}
F(\gamma_{p^i}(\sigma x)^p)  =  F\bigl(\gamma_{p^i}(\sigma x)\bigr) G\bigl(\gamma_{p^i}(\sigma x)\bigr)^{p-1}.
\end{equation*}
Because of the commutativity of the diagram (\ref{commu}) the class $ G\bigl(\gamma_{p^i}(\sigma x)\bigr)^{p-1}$ is divisible by $u^{p-1}$ and is therefore zero. Since $F$ is injective in positive degrees by the proof of Lemma \ref{addeins}, it follows that 
$\gamma_{p^i}(\sigma x)^p$ is zero. We therefore have a well-defined map of $\pi_*(\K(\bar{\mathbb{F}}_l)_p)$-algebras
\[ P_p(\sigma x, \gamma_{p}(\sigma x), \dots) \otimes P_r(u) \to \pi_*\bigl(H\mathbb{F}_p \wedge_{\K}^L \K(\bar{\mathbb{F}}_l)_p\bigr).\]
We claim that it is an isomorphism. Since both sides are one-dimensional over $\mathbb{F}_p$ in each even non-negative degree  and zero in all 
other degrees, it suffices to show the following:
For numbers $j \in \{0, \dots, r-1\}$ and $i_k \in \{0, \dots, p-1\}$  that are almost all zero, the classes
\begin{equation} \label{basise}
\sigma x^{i_0} \gamma_{p}(\sigma x)^{i_1} \cdots u^j \in \pi_*\bigl(H\mathbb{F}_p \wedge_{\K}^L \K(\bar{\mathbb{F}}_l)_p\bigr) 
\end{equation}
are non-zero. 
We show this by induction on the degree of  $b = \sigma x^{i_0} \gamma_{p}(\sigma x)^{i_1} \cdots u^j$.
 If $|b| \in \{0, \dots, 2r-2\}$, we have $b = u^j$ for some $j \in \{0, \dots, r-1\}$ and we already know that the claim is true. 
 We now assume that $|b| \geq 2r$ and that we have already proven that all classes of the form (\ref{basise}) with degree less than $|b|$ are non-zero. Let $k$ be minimal with $i_k \neq 0$.
 We first consider the case $b = \gamma_{p^k}(\sigma x)^{i_{k}}$. If $i_k = 1$ the claim is true by definition of $\gamma_{p^k}(\sigma x)$.  If $i_k \geq 2$ we write 
 \[ F(b) = F\bigl(\gamma_{p^k}(\sigma x)\bigr)\biggl(\sum_{i = 0}^{i_k-1}\binom{i_k}{i_k-1-i}\gamma_{p^k}(\sigma x)^{i_k-1-i}G\bigl(\gamma_{p^k}(\sigma x)\bigr)^i\biggr).  \]
 Since $F$ is injective in positive degrees, we know  by the  induction hypothesis that 
 \[F\bigl(\gamma_{p^k}(\sigma x)\bigr) \doteq \Sigma^2 \sigma x^{p-1}\gamma_p(\sigma x)^{p-1}\cdots \gamma_{p^{k-1}}(\sigma x)^{p-1}u^{r-1}.\]
 Since $G\bigl(\gamma_{p^k}(\sigma x)\bigr)^i$ is divisible by $u$ for $i > 0$, we get 
 \[ F(b) \doteq \Sigma^2 \sigma x^{p-1}\gamma_p(\sigma x)^{p-1}\cdots \gamma_{p^{k-1}}(\sigma x)^{p-1}\gamma_{p^k}(\sigma x)^{i_k-1}u^{r-1}.\]
 By the induction hypothesis we know that this is  non-zero. Thus, $b \neq 0$ in this case. In the other cases we write
 \begin{align} \label{derr}
 F(b) = & \,  F\bigl(\gamma_{p^k}(\sigma x)^{i_k}\bigr) \gamma_{p^{k+1}}(\sigma x)^{i_{k+1}}\cdots u^j \nonumber \\
        & + \gamma_{p^k}(\sigma x)^{i_k} F\bigl(\gamma_{p^{k+1}}(\sigma x)^{i_{k+1}}\cdots u^j \bigr) \\
        & + F\bigl(\gamma_{p^k}(\sigma x)^{i_k}\bigr) G\bigr(\gamma_{p^{k+1}}(\sigma x)^{i_{k+1}}\cdots u^j \bigr). \nonumber 
  \end{align}
  By the induction hypothesis  we know that $\gamma_{p^k}(\sigma x)^{i_k}$ is non-zero. 
  Thus, we have
  \[F\bigl(\gamma_{p^k}(\sigma x)^{i_k}\bigr) \doteq \Sigma^2\sigma x^{p-1}\cdots \gamma_{p^{k-1}}(\sigma x)^{p-1}\gamma_{p^k}(\sigma x)^{i_k-1}u^{r-1}. \]
  Since $ G\bigr(\gamma_{p^{k+1}}(\sigma x)^{i_{k+1}}\cdots u^j \bigr)$ is divisible by $u$, the third summand of the right side of (\ref{derr}) is zero.
 We consider the case $j > 0$. Then the first summand in (\ref{derr}) is zero, too. By the induction hypothesis $\gamma_{p^{k+1}} (\sigma x)^{i_{k+1}}\cdots u^j$ is non-zero, and so we get 
 \[ F\bigl(\gamma_{p^{k+1}}(\sigma x)^{i_{k+1}}\cdots u^j \bigr) \doteq \Sigma^2\gamma_{p^{k+1}}(\sigma x)^{i_{k+1}}\cdots u^{j-1} \]
 and
 \[ F(b) \doteq \Sigma^2\gamma_{p^k}(\sigma x)^{i_k}\gamma_{p^{k+1}}(\sigma x)^{i_{k+1}}\cdots u^{j-1}.\]
 By the induction hypothesis this is non-zero  and thus we get $b \neq 0$.  We now consider the case $j = 0$. Let $l > k$ be  minimal with $i_l \neq 0$. By the induction hypothesis we have
 \[ F\bigl(\gamma_{p^{k+1}}(\sigma x)^{i_{k+1}}\cdots u^j \bigr) \doteq \Sigma^2\sigma x^{p-1}  \cdots \gamma_{p^{l-1}}(\sigma x)^{p-1}\gamma_{p^l}(\sigma x)^{i_l-1} \cdots u^{r-1}.\]
 Because of $\gamma_{p^k}(\sigma x)^{i_k+p-1} = 0$ the second summand in (\ref{derr}) is zero. We get that 
 \[ F(b) \doteq \Sigma^2\sigma x^{p-1}\cdots \gamma_{p^{k-1}}(\sigma x)^{p-1}\gamma_{p^k}(\sigma x)^{i_k-1}\gamma_{p^{k+1}}(\sigma x)^{i_{k+1}}\cdots u^{r-1}. \]
 By the induction hypothesis this is not zero. Thus, $b$ is non-zero. 
\end{proof}

Lemma \ref{multf} implies the following:
\begin{lem} \label{V0ten}
We have an isomorphism of graded rings
\[ V(0)_*(\hat{H}\mathbb{Z}_p \wedge_{\K} \hat{H}\mathbb{Z}_p)  \cong \Gamma(\sigma x) \otimes E(\sigma y)\]
with $|\sigma x| = 2r$ and $|\sigma y| = 2r+1$. 
\end{lem}
\begin{proof}
The sequence 
\begin{equation} \label{res}
\begin{tikzcd}
 0 \ar{r} &  \Sigma^2 \pi_*\bigl(\K(\bar{\mathbb{F}}_l)_p\bigr) \ar{r}{\cdot u} & \pi_*\bigl(\K(\bar{\mathbb{F}}_l)_p\bigr) \ar{r} & \mathbb{Z}_p \ar{r} & 0 
 \end{tikzcd}
 \end{equation}
 is a free resolution of $\mathbb{Z}_p$ as a $\pi_*\bigl(\K(\bar{\mathbb{F}}_l)_p \bigr)$-module.
 Thus, the $E^2$-page of the spectral sequence (\ref{kunnSS}) is the homology of
 \[\begin{tikzcd}
 0 \ar{r} &  \Sigma^2 \Gamma(\sigma x) \otimes P_r(u) \ar{r}{\cdot u} & \Gamma(\sigma x) \otimes P_r(u) \ar{r} & 0, 
 \end{tikzcd}\]
 which is 
 \[ \Gamma (\sigma x) \otimes \Sigma^2 \mathbb{F}_p\{u^{r-1}\} \] 
 in homological degree one and $\Gamma (\sigma x)$ in homological degree zero. 
The spectral sequence has to collapse at the $E^2$-page because the $E^2$-page is concentrated in columns $0$ and $1$. 

We denote the free resolution $(\ref{res})$ by $F_{*,*}$. Let $\sigma u$ be the element $\Sigma^21$ in $F_{1, 2}$. The usual multiplication
on 
\[\pi_*\bigl(\K(\bar{\mathbb{F}}_l)_p\bigr) \otimes_{\mathbb{Z}} E_{\mathbb{Z}}(\sigma u)\] defines a map of complexes 

\[F_{*,*} \otimes_{\pi_*(\K(\bar{\mathbb{F}}_l)_p)} F_{*,*} \to F_{*,*}\]
that lifts the multiplication of $\mathbb{Z}_p$. This implies that the $E^2$-page is multiplicatively given by $\Gamma(\sigma x) \otimes E(\sigma y)$, where
$\sigma y$ represents $u^{r-1}\sigma u$. Since $\Gamma(\sigma x)$ lies in column zero, there are no multiplicative extensions.  
\end{proof}

\subsection{The algebra $\THH_*(\K(\mathbb{F}_q)_p; H\mathbb{F}_p)$} \label{et2}
In this subsection we consider the spectral sequence (\ref{sset2}):
  \[  E^2_{*,*} = V(0)_*(\hat{H} \mathbb{Z}_p \wedge_{\K} \hat{H}\mathbb{Z}_p) \otimes \THH_*(\hat{H} \mathbb{Z}_p; H\mathbb{F}_p)  \Longrightarrow V(0)_*\THH(\K; \hat{H}\mathbb{Z}_p).
 \]
By B\"okstedt's computations \cite{Bo2} and Lemma \ref{V0ten} we have 
\[E^2_{*,*} = \Gamma(\sigma x) \otimes E(\sigma y) \otimes E(\lambda_1) \otimes P(\mu_1).\] 
with $|\sigma x| = (0, 2r)$, $|\sigma y| = (0,2r+1)$, $|\lambda_1| = (2p-1, 0)$ and $|\mu_1| = (2p,0)$.  
We will prove the following result:
\begin{thm} \label{thhkfp}
In case (\ref{1}) we have an isomorphism of rings
\[\THH_*(\K; H\mathbb{F}_p) \cong  P(\mu_1^p) \otimes E(\mu_1^{p-1} \sigma y, \,\, \lambda_1 \sigma x^{p-1}) \otimes \Gamma\bigl(\gamma_p(\sigma x)\bigr). \] 
In case (\ref{2}) we have an isomorphism of rings
\[ \THH_*(\K; H\mathbb{F}_p) \cong P(\mu_1^p) \otimes E( \mu_1^{p-1}\sigma y, \,\, \lambda_1) \otimes \Gamma(\sigma x)  .\]
In case (\ref{3}) we have an isomorphism of rings
\[ \THH_*(\K; H\mathbb{F}_p) \cong P(\mu_1) \otimes E(\sigma y, \, \, \lambda_1) \otimes \Gamma(\sigma x). \]
In case (\ref{4}) we have an isomorphism of rings
\[ \THH_*(\K; H\mathbb{F}_p) \cong P(\mu_1) \otimes E(\sigma y, \, \, \lambda_1) \otimes \Gamma(\sigma x), \]
at least if we assume $r \neq 1$. 
\end{thm}

In order to compute the differentials in the Brun spectral sequence (\ref{sset2}) we need an additional spectral sequence: 
 
\begin{lem} \label{AHL}
Let $B \to C$ be a morphism between cofibrant commutative $S$-algebras. Then, there is a multiplicative spectral sequence
of the form 
\[ E^2_{*,*} = \Tor^{\pi_*(C \wedge_S B)}_{*,*}(C_*, C_*) \Longrightarrow \THH_*(B,C).\]
\end{lem}
\begin{proof}
This follows by using a method of Veen \cite{Veen}.  Writing $S^1 = D^1 \amalg_{S^0} D^1$ one sees that $\THH(B;C) \cong C \wedge_{C \wedge_S B}^L C$ as $S$-ring spectra. 
\end{proof}

\begin{lem} \label{dimlow}
In case (\ref{1}) we have 
\[  V(0)_n\THH(\K, \hat{H}\mathbb{Z}_p) = 0\]
for $1 \leq n \leq 2p$. 
\end{lem}
\begin{proof}
By Lemma \ref{AHL}  we have a  spectral sequence of the form 
\[ E^2_{*,*} = \Tor^{(H\mathbb{F}_p)_*\K}_{*,*}\bigl(\mathbb{F}_p, \mathbb{F}_p\bigr) \Longrightarrow \THH_*(\K; H\mathbb{F}_p\bigr).\]
Using  Proposition \ref{homimj} 
one gets 
\[E^2_{*,*} = E(\sigma \tilde{\xi}_1^p, \sigma \tilde{\xi}_2, \dots) \otimes \Gamma(\sigma b,  \sigma \tilde{\tau}_2, \dots) \] 
with $|\sigma x| = (1, |x|)$. Obviously, this bigraded abelian group is zero in the total degrees $n = 1, \dots, 2p$.  
\end{proof}

\begin{lem} \label{br2c1}
In case (\ref{1}) the differentials of the spectral sequence (\ref{sset2})
are given by 
\[d^{2p-1}(\lambda_1) \doteq \sigma x \text{~~~and~~~}   d^{2p}(\mu_1) \doteq \sigma y.\] 
We have 
\[ E^{\infty}_{*,*} = P(\mu_1^p) \otimes E(\mu_1^{p-1} \sigma y, \, \lambda_1 \sigma x^{p-1}) \otimes \Gamma\bigl(\gamma_p(\sigma x)\bigr).\]
There are no multiplicative extensions. 
\end{lem}
\begin{proof}
The $E^2$-page of (\ref{sset2}) is multiplicatively generated by the classes $\lambda_1$, $\mu_1$, $\gamma_{p^i}(\sigma x)$ and $\sigma y$. 
The classes $\gamma_{p^i}(\sigma x)$ and $\sigma y$ are infinite cycles because they lie in column zero. 
For bidegree reasons the only possible differential on $\lambda_1$ and $\mu_1$ are
\[ d^{2p-1}(\lambda_1) \doteq \sigma x \text{~~~and~~~}   d^{2p}(\mu_1) \doteq \sigma y.\]
We conclude that $d^s = 0$ for $s= 2, \dots , 2p-2$. If $d^{2p-1}(\lambda_1)$ was zero, the class $\lambda_1$ would be a permanent cycle. This would contradict the fact that we have
\[ V(0)_{2p-1}\THH(\K, \hat{H}\mathbb{Z}_p) = 0\]
by Lemma \ref{dimlow}.  
Thus, we have $d^{2p-1}(\lambda_1) \doteq \sigma x$  and  
\[E^{2p}_{*,*} = E(\lambda_1 \sigma x^{p-1}) \otimes \Gamma\bigl(\gamma_p(\sigma x)\bigr) \otimes P(\mu_1) \otimes E(\sigma y).\]
This algebra is generated by the classes $\lambda_1 \sigma x^{p-1}$, $\gamma_{p^i}(\sigma x)$ for $i \geq 1$, $\mu_1$ and $\sigma y$. 
There cannot be any non-trivial differential on $\lambda_1 \sigma x^{p-1}$, because this class lies in column $2p-1$. 
Thus, the only possible differential on a generator is
\[d^{2p}(\mu_1) \doteq \sigma y.\] 
This differential must exist, because otherwise $\mu_1$ would survive to the $E^{\infty}$-page, which would contradict Lemma
\ref{dimlow}. 
It follows that 
\[E^{2p+1}_{*,*} =  P(\mu_1^p) \otimes E(\mu_1^{p-1}\sigma y) \otimes E(\lambda_1 \sigma x^{p-1}) \otimes \Gamma\bigl(\gamma_p(\sigma x)\bigr)\]
as $\mathbb{F}_p$-algebras. 
Now, the spectral sequence has to collapse: 
The class $\mu_1^{p-1}\sigma y$ has  total degree $2p^2-1$, and therefore its differentials have total degree $2p^2-2$. 
All non-trivial classes in an even degree $< 2p^2 = |\mu_1^p|$ lie in $\Gamma\bigl(\gamma_p(\sigma x)\bigr)$. Since the total degree of every class 
in $\Gamma\bigl(\gamma_p(\sigma x)\bigr)$  is divisible by $p$, it follows that $\mu_1^{p-1}\sigma y$ is an infinite cycle.  
The class $\mu_1^p$ is an infinite cycle, too: The differentials of $\mu_1^p$ have  total degree $2p^2-1$. There cannot be any differential 
\[d^{i}(\mu_1^p) \doteq \mu_1^{p-1}\sigma y\] for $i \geq 2p+1$,  because $\mu_1^p$ lies in column $2p^2$ and $\mu_1^{p-1}\sigma y$ lies in column $2p^2-2p$.  
The classes in $\mathbb{F}_p\{\sigma x^{p-1}\lambda_1\} \otimes \Gamma\bigl(\gamma_p(\sigma x)\bigr)$ have total degrees
\[2p-1 +(2p-2)(p-1)+p(2p-2)i \] for $i \geq 0$. Since this is always $1$ modulo $p$, and since $2p^2-1$ is $-1$ modulo $p$, the spectral sequence collapses at the $E^{2p+1}$-page. Since $\Gamma\bigl(\gamma_p(\sigma x)\bigr)$ lies in 
 column zero, there cannot be any multiplicative extensions.  
\end{proof}

\begin{lem} \label{lowdim2}
In case (\ref{2}) we have 
\[ \dim_{\mathbb{F}_p}\bigl(V(0)_n\THH(\K; \hat{H}\mathbb{Z}_p)\bigr) =  \begin{cases}
                                      1, & n = 2p-2; \\
                                      0, & n = 2p; \\
                                      2, & n = 2p^2-1; \\
                                      1, & n = 2p^2. 
                                     \end{cases}\]
\end{lem}
\begin{proof}
As in case (\ref{1}) we consider the spectral sequence
\[ E^2_{*,*} = \Tor^{(H\mathbb{F}_p)_*\K}_{*,*}\bigl(\mathbb{F}_p, \mathbb{F}_p\bigr) \Longrightarrow \THH_*(\K; H\mathbb{F}_p\bigr).\]
Using Proposition \ref{homolK}, we obtain
\[E^2_{*,*} = E(\sigma \tilde{\xi}_1, \dots) \otimes \Gamma(\sigma x,  \sigma \tilde{\tau}_2, \dots) \] 
with $|\sigma a| = (1, |a|)$. 
This bigraded abelian group is zero in total degree $2p$, and therefore we have 
\[ V(0)_{2p}\THH(\K; \hat{H} \mathbb{Z}_p) = 0.\]
It is zero in total degree $2p-3$, $\mathbb{F}_p\{\sigma x\}$ in total degree $2p-2$ and $\mathbb{F}_p\{\sigma \tilde{\xi}_1\}$ in total degree $2p-1$. Since $\sigma \tilde{\xi}_1$ lies in column $1$, it is an infinite cycle. Therefore, we get 
\[ \dim_{\mathbb{F}_p}\bigl(V(0)_{2p-2}\THH(\K; \hat{H}\mathbb{Z}_p)\bigl) = 1.\]
The $E^2$-page is given by $\mathbb{F}_p\{\gamma_{p+1}(\sigma x)\}$ in total degree $2p^2-2$, by 
\[ \mathbb{F}_p\{\sigma \tilde{\xi}_2\} \oplus \mathbb{F}_p\{\sigma \tilde{\xi}_1 \gamma_p(\sigma x)\} \]
in total degree $2p^2-1$, by $\mathbb{F}_p\{\sigma \tilde{\tau}_2\}$ in total degree $2p^2$ and by zero in total degree $2p^2+1$. Since $\sigma \tilde{\tau}_2$ lies in column $1$, it is an infinite cycle and we get
\[ \dim_{\mathbb{F}_p}\bigl(V(0)_{2p^2}\THH(\K, \hat{H}\mathbb{Z}_p)\bigr) = 1.\]
The classes $\sigma \tilde{\xi}_2$ and $\sigma \tilde{\xi}_1 \gamma_p(\sigma x)$ are also infinite cycles: For $\sigma \tilde{\xi}_2$ this is clear, because this class  lies in column $1$. For $\sigma \tilde{\xi}_1 \gamma_p(\sigma x)$ it follows since no differential of $\sigma \tilde{\xi}_1 \gamma_p(\sigma x)$ can hit $\gamma_{p+1}(\sigma x)$, because both classes lie in column $p+1$. 
We get 
\[\dim_{\mathbb{F}_p}\bigl(V(0)_{2p^2-1}\THH(\K; \hat{H}\mathbb{Z}_p)\bigr) = 2.\] 
\end{proof}
\begin{lem} 
In case (\ref{2}) the spectral sequence (\ref{sset2}) has the differential 
\[  d^{2p}(\mu_1) \doteq  \sigma y.\] 
We have 
\[ E^{\infty}_{*,*} = E(\lambda_1, \, \mu_1^{p-1}\sigma y) \otimes \Gamma(\sigma x) \otimes P(\mu_1^p).\]
There are no multiplicative extensions.
\end{lem}
\begin{proof}
As in case (\ref{1}) the only possible differentials on  the canonical algebra generators of the $E^2$-page are
\[ d^{2p-1}(\lambda_1) \doteq \sigma x \text{~~~and~~~}   d^{2p}(\mu_1) \doteq \sigma y.\] 
Hence, we have $d^i = 0$ for $i = 2, \dots, 2p-2$. 
The $E^{2p-1}$-page is given by $\mathbb{F}_p\{\sigma x\}$ in total degree $2p-2$. If there was a differential $d^{2p-1}(\lambda_1) \doteq \sigma x$, the $E^{\infty}$-page would be zero in total degree $2p-2$. This would contradict Lemma \ref{lowdim2}. Thus, we have $d^{2p-1} = 0$.  If $d^{2p}(\mu_1)$ was  zero, the class $\mu_1$ would survive to the $E^{\infty}$-page. This would contradict $\dim_{\mathbb{F}_p}\bigl(V(0)_{2p}\THH(\K, \hat{H}\mathbb{Z}_p)\bigr) = 0$.
Therefore, we get $d^{2p}(\mu_1) \doteq \sigma y$ and
\[E^{2p+1}_{*,*} = P(\mu_1^p) \otimes E(\mu_1^{p-1}\sigma y) \otimes E(\lambda_1) \otimes \Gamma(\sigma x). \]
The $E^{2p+1}$-page is given by 
\[ \mathbb{F}_p\{\mu_1^{p-1}\sigma y\} \oplus \mathbb{F}_p\{\lambda_1 \cdot \gamma_p(\sigma x)\} \]
in total degree $2p^2-1$ and by $\mathbb{F}_p\{\mu_1^p\}$ in total degree $2p^2$. 
So, by Lemma \ref{lowdim2}, there cannot be any differentials on $\mu_1^p$ or $\sigma y\mu_1^{p-1}$. 
We conclude that
\[E^{\infty}_{*,*} = E^{2p+1}_{*,*}.\]
\end{proof}
\begin{lem} \label{lowdim3} 
In case (\ref{3}) we have
\[ \dim_{\mathbb{F}_p}V(0)_{2p-2}\THH(\K;  \hat{H}\mathbb{Z}_p) = 1.\]
There is a class 
\[\sigma x \in V(0)_{2r}\THH(\K, \hat{H}\mathbb{Z}_p)\]
 and a class 
 \[\sigma y \in V(0)_{2r+1}\THH(\K, \hat{H}\mathbb{Z}_p)\]
 such that ${\sigma x}^{k-1}{\sigma y} \neq 0$.
\end{lem}
\begin{proof}
As in the other cases we use the multiplicative spectral sequence
\[ E^2_{*,*} = \Tor^{(H\mathbb{F}_p)_*\K}_{*,*}\bigl(\mathbb{F}_p, \mathbb{F}_p\bigr) \Longrightarrow \THH_*(\K; H\mathbb{F}_p\bigr).\]
Using Proposition \ref{homolK} 
it follows  that 
\[E^2_{*,*} = E(\sigma y, \sigma \tilde{\xi}_1, \dots) \otimes \Gamma(\sigma x, z, \sigma \tilde{\tau}_2, \dots)\]
with $|z| = (2, 2p-2)$. 
The $E^2$-page is  $\mathbb{F}_p\{\sigma x^k\}$ in total degree $2p-2$ and 
\[\mathbb{F}_p\{\sigma \tilde{\xi}_1\} \oplus \mathbb{F}_p\{\sigma x^{k-1}\sigma y\}\]
in total degree $2p-1$. The classes $\sigma x^k$, $\sigma \tilde{\xi}_1$ and $\sigma x^{k-1} \sigma y$ are infinite cycles, because they are products of classes in column $1$.  This implies  that
\[\dim_{\mathbb{F}_p}V(0)_{2p-2} \THH(\K; \hat{H}\mathbb{Z}_p) = 1.\]  
The $E^2$-page is in total degree $2p$ given by $\mathbb{F}_p\{z\}$ if $r > 1$ and by
\[\mathbb{F}_p\{z\} \oplus \mathbb{F}_p\{\gamma_p(\sigma x)\}\]
if $r = 1$. The differentials of $z$ cannot hit $\sigma x^{k-1}\sigma y$ because $z$ lies in column $2$ and $\sigma x^{k-1}\sigma y$ lies in a column $\geq 2$. If $r =1$ the class $\sigma x^{k-1}\sigma y$ lies in column $p-1$. Since $\gamma_p(\sigma x)$ lies in column $p$, its differentials cannot hit $\sigma x^{k-1}\sigma y$. Therefore, $\sigma x^{k-1}\sigma y$ is a  permanent cycle. 
\end{proof}

\begin{lem}
In case (\ref{3}) the spectral sequence (\ref{sset2}) collapses at the $E^2$-page and there are no multiplicative extensions. 
\end{lem}
\begin{proof}
The only possible differentials on the canonical algebra generators of the $E^2$-page are
\[ d^{2p-1}(\lambda_1) \doteq \sigma x^{k} \text{~~~and~~~} d^{2p}(\mu_1) \doteq \sigma x^{k-1} \sigma y .\]
This implies that $d^i = 0$ for $i= 2, \dots, 2p-2$.  In total degree $2p-2$ the $E^{2p-1}$-page is given by $\mathbb{F}_p\{\sigma x^k\}$. Therefore, by Lemma \ref{lowdim3}, the differential $d^{2p-1}(\lambda_1) \doteq \sigma x^{k}$ cannot exist. Hence, we have $d^{2p-1} = 0$. In total degree $2r$ the $E^{2p}$-page is generated as an $\mathbb{F}_p$-vector space by $\sigma x$, in total degree $2r+1$ the $E^{2p}$-page is generated by $\sigma y$.  Because the classes have filtration degree zero, Lemma \ref{lowdim3} implies that they survive to the $E^{\infty}$-page and that  $\sigma x^{k-1} \sigma y \neq 0$ in $E^{\infty}_{*,*}$. Hence, we cannot have
$d^{2p}(\mu_1) \doteq \sigma x^{k-1} \sigma y$ and the spectral sequence collapses at the $E^2$-page. 
\end{proof}

In case (\ref{4}) the mod $p$ homology of $\K$ is more complicated than in the other cases and we only consider the subcase $r > 1$. In order to be able to  compute the $E^2$-page of the spectral sequence
\[ E^2_{*,*} = \Tor^{(H\mathbb{F}_p)_*\K}(\mathbb{F}_p, \mathbb{F}_p) \Longrightarrow \THH_*(\K, H\mathbb{F}_p)\]
 in the necessary degrees we need a couple of
  lemmas.  The statements are probably well-known, but since we did not find references, we include proofs. 

\begin{lem} \label{TorChangeRing}
Let $S_* \to R_*$ be a homomorphism of non-negatively graded-commutative rings that is an isomorphism in degrees $\leq n$. 
Let $M_*$ and $N_*$ be non-negatively graded $R_*$-modules. 
Then, we have 
\[\Tor^{S_*}_{s,t}(M_*, N_*) \cong \Tor_{s,t}^{R_*}(M_*, N_*)\]
for $t \leq n$.
\end{lem}
\begin{proof}
We construct by induction a commutative diagram of graded $S_*$-modules 
\[\begin{tikzcd}
 F_{s-1, *} \ar{r}{d^{s-1}} \ar{d}{\eta_{s-1}} & F_{s-2, *} \ar{r}{d^{s-2}} \ar{d}{\eta_{s-2}} & \dots \ar{r}{d^1} & F_{0,*} \ar{d}{\eta_0} \ar{r}{d^0} & F_{-1,*} \ar{d}{\eta_{-1}} \ar{r} & 0\\
  T_{s-1, *} \ar{r}{d^{s-1}} \ar{r} & T_{s-2,*} \ar{r}{d^{s-2}} & \dots \ar{r}{d^1} & T_{0,*} \ar{r}{d^0} & T_{-1,*} \ar{r} &  0
\end{tikzcd}\]
with the following properties:
\begin{itemize}
\item We have $F_{-1,*} = M_*$, $T_{-1,*} = M_*$ and $\eta_{-1} = \id_{M_*}$. 
\item The lines are exact.
\item The $F_{i,*}$ are free non-negatively graded $S_*$-modules for $i \geq 0$. 
\item The $T_{i,*}$ are free non-negatively graded $R_*$-modules for $i \geq 0$ and the lower line is a sequence of $R_*$-modules. 
\item The maps $\eta_i: F_{i,*} \to T_{i,*}$ are isomorphisms in degrees $* \leq n$. 
\end{itemize}
We start by defining $F_{-1,*} \coloneqq M_*$,  $T_{-1,*} \coloneqq M_*$ and  $\eta_{-1} \coloneqq \id_{M_*}$. Let $d^{-1}$ be the unique map $M_* \to 0$. 
Let $s \geq 0$ and suppose that we have constructed the diagram up to $s-1$. 
We set
 \[ F_{s,*} = \bigoplus_{a \in \ker d^{s-1}\setminus \{0\}} \Sigma^{|a|}S_*\]
 and  define  $d^s$  to be the obvious map of $S_*$-modules. 
Furthermore,  we set 
 \[T_{s,*} = \bigoplus_{a \in \ker d^{s-1} \setminus \{0\}} \Sigma^{|a|} R_*\]
 and define $d^s$ to be the obvious map of $R_*$-modules.
 Let $\eta_s$ be the map of $S_*$-modules that is defined by 
 \[\Sigma^{|a|}1 \mapsto \begin{cases}
                          \Sigma^{|\eta_{s-1}(a)|}1, \text{~if~} \eta_{s-1}(a) \neq 0; \\
                           0 , \text{~otherwise}.
  \end{cases}
  \]
  It is then clear that
  \[\begin{tikzcd}
  F_{s,*} \ar{r}{d^s} \ar{d}{\eta_s} & F_{s-1,*} \ar{d}{\eta_{s-1}} \\
  T_{s,*} \ar{r}{d^s} & T_{s-1,*} 
  \end{tikzcd}\]
  is commutative. The map $\eta_s$ is an isomorphism in degrees $\leq n$. This is because  by the induction hypothesis we know that
  $\eta_{s-1}$ induces a bijection 
  \[\{a \in \ker d^{s-1} \setminus \{0\} : |a| \leq n\} \to \{ a \in\ker d^{s-1} \setminus \{0\}: |a| \leq n\}\]
  and because $S_* \to R_*$ is an isomorphism in degree $\leq n$. 
  This shows the induction step. 
We have commutative diagrams
\[\begin{tikzcd}
F_{s,*} \otimes_{\mathbb{Z}} N_* \ar{r} & T_{s,*} \otimes_{\mathbb{Z}} N_* \\
F_{s,*} \otimes_{\mathbb{Z}} S_* \otimes_{\mathbb{Z}} N_* \ar{r} \ar[xshift = 0.7 ex]{u} \ar[xshift = -0.7ex]{u} & T_{s,*} \otimes_{\mathbb{Z}} R_* \otimes_{\mathbb{Z}} N_* \ar[xshift = 0.7ex]{u} \ar[xshift = -0.7ex]{u} \, ,  
\end{tikzcd}\]  
where the vertical maps are induced by the $S_*$- and $R_*$-actions. 
The horizontal maps are isomorphism in degrees $\leq n$. 
 We get maps on the coequalizers 
 \[F_{s,*} \otimes_{S_*} N_* \to T_{s,*} \otimes_{R_*} N_*\]
  that are isomorphisms in degrees
 $\leq n$ and that give a map of complexes. Since homology is taken degreewise this shows the claim. 
 \end{proof}

\begin{lem} \label{subcom}
Let $R_*$ be a  graded-commutative ring and let $N_*$ be a non-negatively graded $R_*$-module. 
Let $Q_{*,*}$ be a complex of graded $R_*$-modules and  let $P_{*,*} \to Q_{*,*}$ be a subcomplex with the following properties:
\begin{enumerate}
 \item If $x \in Q_{*,*}$ has total degree $\leq n+1$, then we have $x \in P_{*,*}$. \label{uno}
 \item $P_{m,*}$ is a direct summand of $Q_{m,*}$.  \label{do}  
\end{enumerate}
  Then, the map 
  \[ H_*(P_{*,*} \otimes_{R_*}N_*) \to H_*(Q_{*,*} \otimes_{R_*} N_*) \]
  is an isomorphism in total degrees $\leq n$. 
 \end{lem}
 \begin{proof}
 Note that the maps
 \[P_{m,*} \otimes_{R_*} N_* \to Q_{m,*} \otimes_{R_*} N_*\] 
 are injective.  Thus,  $P_{*,*} \otimes_{R_*} N_*$ is a subcomplex of 
 $Q_{*,*} \otimes_{R_*} N_*$.  Moreover, every class $x \in Q_{*,*} \otimes_{R_*} N_*$  in total degree $\leq n+1$ lies in $ P_{*,*} \otimes_{R_*} N_*$.  
The lemma now follows from the following fact: 
Let $Z_{*,*} \subseteq W_{*,*}$ a subcomplex with the property that every class $x \in W_{*,*}$ in total degree $\leq n+1$ lies in $Z_{*,*}$, then the induced map 
\[ H_*(Z_{*,*}) \to H_*(W_{*,*})\] is an 
isomorphism in total degrees $\leq n$.
 \end{proof}

\begin{lem} \label{extendCHain}
Let $R_*$ be a non-negatively graded-commutative ring and let $M_*$  be a  graded $R_*$-module. 
Suppose that we have a chain complex
\[\begin{tikzcd}
\dots \ar{r}{d^3} & P_{2,*} \ar{r}{d^2} & P_{1,*} \ar{r}{d^1} & P_{0,*} \ar{r}{d^0} & M_* \ar{r} & 0 
\end{tikzcd}\]
with the following properties: 
\begin{enumerate}
\item The $P_{m,*}$ are free graded  $R_*$-modules. 
\item The map $d^0$ is surjective.
\item If $x \in \ker d^i$ has total degree $\leq n$, then we have $x \in \im d^{i+1}$. \label{tre}
\end{enumerate}
Then there exists a free  resolution $Q_{*,*} \to M_*$ such that $P_{*,*}$ is a subcomplex of $Q_{*,*}$ with the properties (\ref{uno}) and (\ref{do})  in Lemma \ref{subcom}. 
\end{lem}
\begin{proof}
We define the resolution $Q_{*,*}$ inductively. 
We define 
\[ \begin{tikzcd} 
Q_{0,*} \ar{r}{{d'}^0} & M_* \ar{r} & 0
\end{tikzcd}\]
to be 
\[\begin{tikzcd}
P_{0,*} \ar{r}{d^0} & M_* \ar{r} & 0.
\end{tikzcd}\]
Suppose that $i \geq 1$ and that we have already constructed an exact sequence 
\[ \begin{tikzcd}
 Q_{i-1,*} \ar{r}{{d'}^{i-1}} & Q_{i-2,*} \ar{r}{{d'}^{i-2}} & \dots \ar{r} & Q_{0,*} \ar{r}{{d'}^0} & M_* \ar{r} & 0
 \end{tikzcd}\]
 such that:
 \begin{itemize}
\item For $0 \leq m \leq i-1$ we have 
 \[ Q_{m,*} = P_{m,*} \oplus \bigoplus_{s \in I_m}\Sigma^{k_s}R_* \] 
 for a set $I_m$ and natural numbers $k_s$ for $s \in I_m$ with $k_s + m \geq n+2$. 
 \item The diagram 
 \[\begin{tikzcd}
 Q_{i-1,*} \ar{r}{{d'}^{i-1}} & Q_{i-2,*} \ar{r}{{d'}^{i-2}} & \dots \ar{r} & Q_{0,*} \ar{r}{{d'}^0} & M_* \ar{r} & 0  \\
 P_{i-1,*} \ar{r}{d^{i-1}} \ar{u} & P_{i-2,*} \ar{u} \ar{r}{d^{i-2}} &  \dots \ar{r} & P_{0,*} \ar{r}{d^0} \ar{u} & M_* \ar{r} \ar{u}{\id} & 0
 \end{tikzcd}\]
 commutes.
\end{itemize}
We define 
\[Q_{i, *} = P_{i,*} \oplus \bigoplus_{\substack{x \in \ker {d'}^{i-1} \setminus \{0\}\\ |x| + i-1 \geq n+1}}\Sigma^{|x|}R_*.\]
Here, $|x|$ means the internal degree of the element $x$. Let 
${d'}^i: Q_{i,*} \to Q_{i-1, *}$ be the map  that is given by
\[P_{i,*} \to P_{i-1,*} \to Q_{i-1,*}\] on $P_{i,*}$ and that maps
$\Sigma^{|x|}1$ to $x$. 
Then, the sequence 
\[\begin{tikzcd}
Q_{i, *} \ar{r}{{d'}^i} & Q_{i-1,*} \ar{r}{{d'}^{i-1}} & \dots \ar{r} & Q_{0,*} \ar{r}{{d'}^0} & M_* \ar{r} & 0
\end{tikzcd}\]
is exact: If $x \in \ker {d'}^{i-1} \setminus \{0\}$ and $|x| + i-1 \leq n$, then $x \in P_{i-1,*}$ and we have $x \in \im {d'}^{i}$ by item (\ref{tre}). This shows the induction step. 
\end{proof}
\begin{lem} \label{AHL3}
We consider  case (\ref{4}). If $r > 1$ we have 
\[ \dim_{\mathbb{F}_p}V(0)_n\THH(\K; \hat{H}\mathbb{Z}_p) = \begin{cases}
                                                        1, & n = 2p-2;  \\
                                                        2,  & n = 2p-1.
                                                        \end{cases}\]
\end{lem}
\begin{proof}
As in the other cases we consider the spectral sequence
\begin{equation}\label{ahlss}
 E^2_{*,*} = \Tor^{(H\mathbb{F}_p)_*\K}_{*,*}\bigl(\mathbb{F}_p, \mathbb{F}_p\bigr) \Longrightarrow \THH_*(\K; H\mathbb{F}_p\bigr).
 \end{equation} 
By Lemma \ref{c4hom}  we have a map 
\[ {E(x) \otimes P_k(y)}/{xy^{k-1}}  \to (H\mathbb{F}_p)_*\K\]
that is an isomorphism in degrees $\leq 2p$.  By Lemma \ref{TorChangeRing} we have 
\[E^2_{s,t} \cong \Tor^{{E(x) \otimes P_k(y)}/{xy^{k-1}}}_{s,t}(\mathbb{F}_p,  \mathbb{F}_p)\]
for $t \leq 2p$. We set $R_* := {E(x) \otimes P_k(y)}/{xy^{k-1}}$. 
We will construct a chain complex 
\[ \begin{tikzcd}
\dots \ar{r} &  P_{2,*} \ar{r}{d^2} & P_{1,*} \ar{r}{d^1} & P_{0,*} \ar{r}{d^0} & \mathbb{F}_p \ar{r} & 0 
\end{tikzcd}\]
of free graded $R_*$-modules with the following properties: 
\begin{itemize}
\item The map $d^0$ is surjective.
\item If $x \in \ker d^i$ has total degree $\leq 2p$, then we have $x \in \im d^{i+1}$.
\end{itemize}
By the Lemmas $\ref{subcom}$ and $\ref{extendCHain}$ we then have 
\[ H_*(P_{*,*} \otimes_{R_*} \mathbb{F}_p) \cong \Tor^{R_*}_{*,*}(\mathbb{F}_p, \mathbb{F}_p)\]
in total degrees $\leq 2p$ and therefore $E^2_{*,*} \cong H_*(P_{*,*} \otimes_{R_*} \mathbb{F}_p)$ in total degrees $\leq 2p$.

 We set $P_{0, *} = R_*$. Let $d^0$ be the $R_*$-module map defined by $1 \mapsto 1$. 
We define 
\[ P_{1,*} = R_{*}\gamma_1 \oplus  R_*\upsilon_1,\]
 $d^1(\gamma_1) = x$ and $d^1(\upsilon_1) = y$. Note that then $\gamma_1$ has to have bidegree $(1, 2r-1)$ and that 
 $\upsilon_1$ has to have bidegree $(1,2r)$. Obviously, 
 \[\begin{tikzcd}
  P_{1,*}  \ar{r}{d^1} & P_{0,*} \ar{r}{d^0} & \mathbb{F}_p \ar{r} & 0
  \end{tikzcd}\]
  is exact. The kernel of $d^1$ is given by
  \[ \bigoplus_{i = 0}^{k-2} \mathbb{F}_p\{xy^i\gamma_1\} \oplus \mathbb{F}_p\{y^{k-1}\gamma_1\} \oplus \mathbb{F}_p\{y^{k-1}\upsilon_1\} \oplus \mathbb{F}_p\{xy^{k-2}\upsilon_1\} \oplus \bigoplus_{i = 0}^{k-3}\mathbb{F}_p\{y^{i+1}\gamma_1 - xy^i\upsilon_1\}.\]
  We set
  \[P_{2,*}  = R_*\gamma_2 \oplus R_*w_2 \oplus  R_*z_2 \oplus R_*a_2 \oplus R_*\upsilon_2,\]
 and define $d^2$ by  $d^2(\gamma_2) = x \gamma_1$, $d^2(w_2) = y^{k-1}\gamma_1$, $d^2(z_2) =  y^{k-1} \upsilon_1$, $d^2(a_2) = xy^{k-2}\upsilon_1$ and $d^2(\upsilon_2) = y \gamma_1 - x\upsilon_1$. 
   Then, the bidegrees of the  generators of $P_{2,*}$ are given by $|\gamma_2| = (2, 2  (2r-1))$, $|w_2| = (2, 2p-3)$, $|z_2| = (2, 2p-2)$, $|a_2| = (2, 2p-3)$, $|\upsilon_2| = (2, 2 \cdot 2r-1)$ and the sequence
   \[ \begin{tikzcd}
   P_{2,*} \ar{r}{d^2} & P_{1,*} \ar{r}{d^1} & P_{0,*} \ar{r}{d^0} &  \mathbb{F}_p \ar{r} & 0
   \end{tikzcd}\]
   is exact. The $\mathbb{F}_p$-vector space 
   \begin{eqnarray*}
 & &    \bigoplus_{i = 0}^{k-2} \mathbb{F}_p\{xy^i \gamma_2\} \oplus \mathbb{F}_p\{y^{k-1}\gamma_2\} \oplus \mathbb{F}_p\{y^{k-1} \upsilon_2\} \oplus \mathbb{F}_p\{xy^{k-2}\upsilon_2\}\\
 & \oplus & \bigoplus_{i = 0}^{k-3}\mathbb{F}_p\{y^{i+1} \gamma_2 - x y^i\upsilon_2\} 
     \oplus  \mathbb{F}_p\{-w_2+y^{k-2}\upsilon_2 + a_2\}
    \end{eqnarray*}
   is included in $\ker d^2$ and contains every element in $\ker d^2$ with total degree $\leq 2p$. 
   We set 
   \[P_{3,*} = R_*\gamma_3 \oplus R_*w_3 \oplus R_* z_3 \oplus R_*a_3 \oplus R_*\upsilon_3 \oplus R_*b_3,\] 
   and  define $d^3$ by 
   $d^3(\gamma_3) = x \gamma_2$, $d^3(w_3) = y^{k-1} \gamma_2$, $d^3(z_3) = y^{k-1}\upsilon_2$, $d^3(a_3) = xy^{k-2}\upsilon_2$, $d^3(\upsilon_3) = y \gamma_2-x \upsilon_2$ and $d^3(b_3) = -w_2+y^{k-2}\upsilon_2 +  a_2$.  
   We then have $|\gamma_3| = (3, 3(2r-1))$, $|w_3| = (3, 2p-2+2r-2)$, $|z_3| = (3, 2p-2+2r-1)$, $|a_3| = (3, 2p-2+2r-2)$, 
   $|\upsilon_3| = (3, 3\cdot 2r-2)$, $|b_3| = (3, 2p-3)$, the composition 
   \[\begin{tikzcd}
   P_{3,*} \ar{r}{d^3} & P_{2,*} \ar{r}{d^2} & P_{1,*} 
   \end{tikzcd}\]  
   is zero and every class in $\ker d^2$ with total degree $\leq 2p$ is in the image of $d^3$. 
   The $\mathbb{F}_p$-vector space 
   \[ \bigoplus_{i = 0}^{k-2}\mathbb{F}_p\{xy^i\gamma_3\} \oplus \mathbb{F}_p\{y^{k-1}\gamma_3\} \oplus \mathbb{F}_p\{y^{k-1}\upsilon_3\} \oplus \mathbb{F}_p\{xy^{k-2}\upsilon_3\} \oplus \bigoplus_{i = 0}^{k-3}\mathbb{F}_p\{y^{i+1}\gamma_3-xy^i\upsilon_3\}  \]
   is included in the kernel of $d^3$ and contains every element in the kernel that has a total degree $\leq 2p$. 
   For  $i \geq 4$ we set
   \[ P_{i,*} = R_*\gamma_i \oplus R_*w_i \oplus R_*z_i \oplus R_*a_i \oplus R_*\upsilon_i, \]
   where the internal degrees of the generators are defined to be $|\gamma_i| = i(2r-1)$, $|w_i| = 2p-2 +(i-2)2r-i+1$, $|z_i| = 2p-2 +(i-2)2r-i+2$, $|a_i| = 2p-2+(i-2)2r-i+1$ and $|\upsilon_i| = 2ri-i+1$. 
We set   $d^i(\gamma_i) = x \gamma_{i-1}$, $d^i(w_i) = y^{k-1} \gamma_{i-1}$, $d^i(z_i) = y^{k-1}\upsilon_{i-1}$, $d^i(a_i) = xy^{k-2}\upsilon_{i-1}$ and $d^i(\upsilon_i) = y \gamma_{i-1} - x \upsilon_{i-1}$. 
For $i \geq 4$ the $\mathbb{F}_p$-vector space 
\[  \bigoplus_{j = 0}^{k-2}\mathbb{F}_p\{xy^j\gamma_i\} \oplus \mathbb{F}_p\{y^{k-1}\gamma_i\} \oplus \mathbb{F}_p\{y^{k-1}\upsilon_i\} \oplus \mathbb{F}_p\{xy^{k-2}\upsilon_i\} \oplus \bigoplus_{j = 0}^{k-3}\mathbb{F}_p\{y^{j+1}\gamma_i-xy^j\upsilon_i\}  \]
is included in $\ker d^i$ and contains every class in $\ker d^i$ that has total degree $\leq 2p$. 
This shows that for $i \geq 3$ the following holds: The composition
\[\begin{tikzcd}
 P_{i+1, *} \ar{r}{d^{i+1}} &  P_{i,*} \ar{r}{d^i} &  P_{i-1,*}
 \end{tikzcd}\] 
is zero and  every element in $\ker d^i$ with a total degree $\leq 2p$ is in the image of $d^{i+1}$. 

The complex $P_{*,*} \otimes_{R_*} \mathbb{F}_p$ is given by
\[\begin{tikzcd} 
  \mathbb{F}_p \\
  \mathbb{F}_p \gamma_1 \oplus \mathbb{F}_p \upsilon_1 \ar{u}{0}  \\
  \mathbb{F}_p\gamma_2 \oplus \mathbb{F}_pw_2 \oplus \mathbb{F}_pz_2 \oplus \mathbb{F}_pa_2 \oplus \mathbb{F}_p\upsilon_2 \ar{u}{0} \\
  \mathbb{F}_p\gamma_3 \oplus \mathbb{F}_pw_3 \oplus \mathbb{F}_pz_3 \oplus \mathbb{F}_pa_3  \oplus \mathbb{F}_p\upsilon_3 \oplus \mathbb{F}_pb_3 \ar{u}{d^3} \\
  \mathbb{F}_p\gamma_4 \oplus \mathbb{F}_pw_4 \oplus \mathbb{F}_pz_4 \oplus \mathbb{F}_p a_4 \oplus \mathbb{F}_p \upsilon_4 \ar{u}{0} \\
  \dots \ar{u}{0}
\end{tikzcd}\]
Here,  $d^3$  maps all generators to  zero, except for $b_3$. It maps $b_3$ to $-w_2 + a_2$ if $k > 2$ and it maps $b_3$ to $ -w_2 + a_2 + \upsilon_2$ if $k = 2$. 
The bigraded abelian group $H_*(P_{*,*} \otimes_{R_*} \mathbb{F}_p)$ is in total degree $2p-3$ zero, in total degree 
$2p-2$ given by $\mathbb{F}_p\gamma_k$, in total degree $2p-1$ given by $\mathbb{F}_p\upsilon_k \oplus \mathbb{F}_pw_2$ and in total degree $2p$ given by $\mathbb{F}_pz_2$.  Thus, the same is true for the $E^2$-page of the spectral sequence (\ref{ahlss}). The differentials of $w_2$ and $\upsilon_k$ cannot hit $\gamma_k$, because the homological degree of $\gamma_k$ is greater as or equal to the homological degree of $w_2$ and $\upsilon_k$.  For the same reason $z_2$ has to be an infinite cycle. This proves the lemma. 
\end{proof}
\begin{lem} \label{Fall4}
We consider case (\ref{4}). If $r > 1$ the spectral sequence (\ref{sset2}) collapses at the $E^2$-page. There are no multiplicative extensions. 
\end{lem}
\begin{proof}
As in case (\ref{3}) the only possible differentials on the canonical  algebra generators of the $E^2$-page are
\[ d^{2p-1}(\lambda_1) \doteq \sigma x^{k} \text{~~~and~~~} d^{2p}(\mu_1) \doteq \sigma x^{k-1} \sigma y .\]
We conclude that $d^i = 0$ for $i = 2, \dots, 2p-2$. The $E^{2p-1}$-page is in total degree $2p-2$ given by 
$\mathbb{F}_p\{\sigma x^k\}$. Therefore, by Lemma \ref{AHL3}, the differential $d^{2p-1}(\lambda_1) \doteq \sigma x^k$ cannot exist and we get  $d^{2p-1} = 0$.  In total degree $2p-1$ the $E^{2p}$-page is given by 
\[\mathbb{F}_p\{\lambda_1\} \oplus \mathbb{F}_p\{\sigma x^{k-1}\sigma y\}.\]
Hence, by Lemma \ref{AHL3}, the differential $d^{2p}(\mu_1) \doteq \sigma x^{k-1} \sigma y$ cannot exist and we conclude that the spectral sequence collapses at the $E^2$-page.
\end{proof}

\begin{rmk}
It seems likely that Lemma \ref{Fall4} is also true for $r = 1$. However, the above proof does not work in this case, because some of the degree arguments require $r > 1$. 
\end{rmk}
\subsection{The $V(1)$-homotopy of $\THH\bigl(\K(\mathbb{F}_q)_p\bigr)$ in the first case}  \label{et3}

In this subsection we consider the spectral sequence (\ref{sset3})
\[E^2_{*,*} \cong V(1)_*\K \otimes \THH_*(\K; H\mathbb{F}_p) \Longrightarrow V(1)_*\THH(\K)\]
in case (\ref{1}). 
By Lemma \ref{v1K} and Theorem \ref{thhkfp} we have
\[ E^2_{*,*} \cong E(x) \otimes E(a, \lambda_2) \otimes P(\mu_2) \otimes \Gamma(b) \]
with $|x| = (0, 2p-3)$, $|a| = (p(2p-2)+1, 0)$, $|\lambda_2| = (2p^2-1, 0)$, $|\mu_2| = (2p^2,0)$ and $|b| = (p(2p-2), 0)$. 

\begin{thm} \label{case1V1THH}
In case (\ref{1}) the spectral sequence (\ref{sset3}) has the differential 
\[ d^{2p-2}(\lambda_2) \doteq xa.\]
We have
\[V(1)_*\THH(\K) \cong \Omega_*^{\infty} \otimes P([\mu_2]) \otimes \Gamma ([b]),\]
where $\Omega^{\infty}_*$ is the  graded-commutative $\mathbb{F}_p$-algebra with generators $x$, $c$, $d$, $e$ in degrees  $|x| = 2p-3$, $|c| = 2p^2+2p-4$, $|d| = 4p^2-2p$, $|e| = p(2p-2)+1$ and 
relations
\begin{eqnarray} \label{rela}
& d^2  = c^2 = 0, \nonumber \\
& xe  = xc = 0,  \nonumber \\
& de  = dc = 0,  \nonumber \\
& ec  = -xd.
\end{eqnarray}
\end{thm}
\begin{proof}
Note that the spectral sequence only has two non-trivial lines, namely line zero which is 
\[ C_* := E(a, \lambda_2) \otimes P(\mu_2) \otimes \Gamma(b)\] and line $2p-3$ which 
is 
\[ \mathbb{F}_p\{x\} \otimes E(a, \lambda_2) \otimes P(\mu_2) \otimes \Gamma(b).\]
We claim that the classes $\gamma_{p^i}(b)$ are infinite cycles for $i \geq 0$. We have 
\[d^{2p-2}\bigl(\gamma_{p^i}(b)\bigr) = x w\] for a class  $w$ in  $C_*$ in degree $p^{i+1}(2p-2) -2p+2$. Since $w$ has even degree it lies in 
\[ P(\mu_2) \otimes \Gamma(b) \, \, \, \oplus \, \, \, \mathbb{F}_p\{a\} \otimes \mathbb{F}_p\{\lambda_2\} \otimes P(\mu_2) \otimes \Gamma(b).\]
Every  class in this graded abelian group has a degree divisible by $2p$. Since  
 $p^{i+1}(2p-2)-2p+2$ is not divisible by $2p$, we get $d^{2p-2}(\gamma_{p^i}(b)) = 0$.  
The classes $a$ and $\mu_2$ are infinite cycles, because 
 $C_*$ is trivial in degrees  $0 < * < p(2p-2)$ and in degree $p(2p-2)+2$.



 
 We claim that there is a differential $d^{2p-2}(\lambda_2) \doteq xa$. Since in total degree $2p^2-2$ the $(2p-3)$th line of the $E^2$-page  is given by $\mathbb{F}_p\{xa\}$, it suffices to show that $\lambda_2$ is not an infinite cycle. To prove this, we note that by  \cite[Lemma 3.15, proofs of Theorem 4.11 and Lemma  4.13]{Brunss}  the edge homomorphism 
 \[
 \begin{tikzcd}
 V(1)_*\THH(\K) \ar{r} & \THH_*(\K; H\mathbb{F}_p)
 \end{tikzcd}
\]
 is induced by a map $V(1) \wedge_S^L \THH(\K) \to \THH(\K, H\mathbb{F}_p)$ in $\mathscr{D}_S$. We therefore  
 have a commutative diagram
\begin{equation} \label{diagfin}
\begin{tikzcd}
 V(1)_*\THH(\K) \ar{r}{\varepsilon} \ar{d}[swap]{h} & \THH_*(\K; H\mathbb{F}_p) \ar{d}{h} \\
{H\mathbb{F}_p}_* \bigl(V(1) \wedge_S^L \THH(\K)\bigr) \ar{r} & (H\mathbb{F}_p)_* \THH(\K; H\mathbb{F}_p),
\end{tikzcd}
\end{equation}
where the upper horizontal map is the edge homomorphism and where the vertical maps are the Hurewicz homomorphisms. We suppose that $\lambda_2$ is an infinite cycle. Let $[\lambda_2] \in V(1)_{2p^2-1}\THH(\K)$ be a representative of $\lambda_2$. We have $\varepsilon([\lambda_2]) = \lambda_2 \neq 0$. Since $\THH(\K; H\mathbb{F}_p)$ is a module over the $S$-ring spectrum $H\mathbb{F}_p$, the right Hurewicz homomorphism in (\ref{diagfin}) is injective. We get $h(\varepsilon([\lambda_2])) \neq 0$ and therefore $h([\lambda_2]) \neq 0$. This is a contradiction, because by \cite[Corollary 17.14]{Swi}  the image of the Hurewicz morphism is always contained in the subspace of comodule primitives and because by Lemma \ref{primimj} there is no non-trivial comodule primitive in 
\[(H\mathbb{F}_p)_{2p^2-1}\bigl(V(1) \wedge_S^L\THH(\K)\bigr).\]
We conclude that $d^{2p-2}(\lambda_2) \doteq xa$.   
We get 
\begin{equation*} 
E^{\infty}_{*,*} = E^{2p-1}_{*,*}  = H_*\bigl((E(x) \otimes E(a,\lambda_2)\bigr) \otimes P(\mu_2) \otimes \Gamma(b)
\end{equation*}
and one easily sees that as an $\mathbb{F}_p$-vector space one has 
\[ H_*\bigl(E(x) \otimes E(a, \lambda_2)\bigl)  \cong \mathbb{F}_p\{1\} \oplus \mathbb{F}_p\{x\} \oplus \mathbb{F}_p\{a\} \oplus \mathbb{F}_p\{x \lambda_2\} \oplus \mathbb{F}_p\{a \lambda_2\} \oplus \mathbb{F}_p\{xa\lambda_2\}.\]
The $(2p-3)$th line of $E^{\infty}_{*,*}$ is therefore given by 
\[\mathbb{F}_p\{x\} \otimes P(\mu_2) \otimes \Gamma(b)\, \,  \oplus \, \, \mathbb{F}_p\{x \lambda_2\} \otimes P(\mu_2) \otimes \Gamma(b) \,\, \oplus \,\, \mathbb{F}_p\{xa\lambda_2\} \otimes P(\mu_2) \otimes \Gamma(b).\]
Thus, line $2p-3$ is zero in  total degrees divisible by $2p$. 
It follows that the classes $\gamma_{p^i}(b)$, $\mu_2$ and $a \lambda_2$ have unique representatives $[\gamma_{p^i}(b)]$, $[\mu_2]$ and $d$ in $V(1)_*\THH(\K)$. The class $a$  has a unique representative $e$ because $E^{\infty}_{*,2p-3}$ is zero in  total degrees $2p-3 < * < 2p^2-3$.  
 The classes $x$ and  $x\lambda_2$ also have unique representatives $x$ and $c$  because
they lie in line $2p-3$. 
Since $\gamma_{p^i}(b)^p = 0$  and $(a \lambda_2)^2 = 0$ in $E^{\infty}_{*,*}$ and since the total degrees of $\gamma_{p^i}(b)^p$ and $(a \lambda_2)^2$ are divisible by $2p$, we get $[\gamma_{p^i}(b)]^p = 0$ and $d^2 = 0$ in $V(1)_*\THH(\K)$. 
The equations $c^2 = 0$, $xe = 0$, $xc = 0$, $de = 0$, $dc = 0$ and $ec= -xd$  holds, because the corresponding equations in $E^{\infty}_{*,*}$ are true and because $c^2$, $xe$, $xc$ $de$, $dc$ and $ec + xd$ reduce to classes in lines $\geq 2p-3$. 
Hence, we have a map of $\mathbb{F}_p$-algebras
\[\begin{tikzcd}
g: \Omega^{\infty}_* \otimes P([\mu_2]) \otimes P_p([b], [\gamma_p(b)], \dots)  \ar{r} & V(1)_*\THH(\K). 
\end{tikzcd}\]
Because of the relations (\ref{rela}) the classes $1$, $x$, $e$. $c$, $d$ and $xd$ generate $\Omega^{\infty}_*$ as an $\mathbb{F}_p$-vector space. Thus, $g$ maps a generating set bijectively onto a basis of $V(1)_*\THH(\K)$ and therefore is an isomorphism. 
\end{proof}
We mention some ideas for the differentials of the spectral sequence (\ref{sset3}) in the other cases: 
\begin{rmk}
In case (\ref{2}) the $E^2$-page of the spectral sequence (\ref{sset3}) 
is given by 
\[ E^2_{*,*} = E(x) \otimes E(\lambda_1, \lambda_2) \otimes \Gamma(\sigma x) \otimes P(\mu_2), \] 
where the total degrees are $|x| = 2p-3$, $|\lambda_1| = 2p-1$, $|\lambda_2| = 2p^2-1$, $|\sigma x| = 2p-2$ and $|\mu_2| = 2p^2$. 
By our result obtained with the B\"okstedt spectral sequence (Theorem \ref{casetwores}), the spectral sequence has to collapse.

We now consider the cases (\ref{3}) and (\ref{4}). In case (\ref{4}) we assume that $r > 1$. 
Then, the $E^2$-page of the spectral sequence (\ref{sset3}) is given by
\[E^2_{*,*} = E(x) \otimes P_k(y)\otimes \Gamma(\sigma x) \otimes E(\sigma y) \otimes E(\lambda_1) \otimes P(\mu_1).\] 
 In case (\ref{3}) we have the equation $y^k = 0$ in $(H\mathbb{F}_p)_*\K$. In case (\ref{4}) we have the relations $y^k = 0$ and $xy^{k-1} = 0$ in $(H\mathbb{F}_p)_*\K$. It  seems plausible that, analogous to the case of $\ku_p$ (see \cite{Brunss}), we get a differential 
 \[ d^{2p-2r-1}(\mu_1) = y^{k-1} \sigma y\]
in case (\ref{3}) and  differentials 
 \[  d^{2p-2r-2}(\lambda_1) = xy^{k-2} \sigma y \text{~~~and~~~} d^{2p-2r-1}(\mu_1) = y^{k-1} \sigma y \]
 in case (\ref{4}). In case (\ref{4}) it seems plausible that there are additional differentials, similar to case (\ref{1}). 
\end{rmk}
\printbibliography
\end{document}